%% file: main.tex
\documentclass[11pt,letterpaper]{article}
\usepackage{fullpage}
\usepackage[pdftex]{graphicx}
\usepackage{graphicx}
\usepackage{epstopdf,subcaption}
\usepackage{amssymb}
\usepackage{amsthm}
\usepackage{amsfonts}
\usepackage{amsmath}
\usepackage{multicol,multirow}
\usepackage{algorithm,algcompatible}
\usepackage[margin=1in]{geometry}

\newtheorem{theorem}{Theorem}
\newtheorem{lemma}[theorem]{Lemma}
\newtheorem{corollary}[theorem]{Corollary}
\newtheorem{proposition}[theorem]{Proposition}
\newtheorem{definition}{Definition}

\input{macros}

\title{Learning Semidefinite Regularizers}

\author{Yong Sheng Soh$^\dag$ and Venkat Chandrasekaran$^{\dag,\ddag}$ \thanks{Email: ysoh@caltech.edu, venkatc@caltech.edu} \vspace{0.25in} \\ $^\dag$ Department of Computing and Mathematical Sciences\\ $^\ddag$ Department of Electrical Engineering \\ California Institute of Technology \\ Pasadena, CA 91125}
\date{Jan 4, 2017, revised Dec 3, 2018}

\begin{document}

\maketitle

\begin{abstract}
Regularization techniques are widely employed in optimization-based approaches for solving ill-posed inverse problems in data analysis and scientific computing.  These methods are based on augmenting the objective with a penalty function, which is specified based on prior domain-specific expertise to induce a desired structure in the solution.  We consider the problem of learning suitable regularization functions from data in settings in which precise domain knowledge is not directly available.  Previous work under the title of `dictionary learning' or `sparse coding' may be viewed as learning a regularization function that can be computed via linear programming.  We describe generalizations of these methods to learn regularizers that can be computed and optimized via semidefinite programming.  Our framework for learning such semidefinite regularizers is based on obtaining structured factorizations of data matrices, and our algorithmic approach for computing these factorizations combines recent techniques for rank minimization problems along with an operator analog of Sinkhorn scaling.  Under suitable conditions on the input data, our algorithm provides a locally linearly convergent method for identifying the correct regularizer that promotes the type of structure contained in the data.  Our analysis is based on the stability properties of Operator Sinkhorn scaling and their relation to geometric aspects of determinantal varieties (in particular tangent spaces with respect to these varieties). The regularizers obtained using our framework can be employed effectively in semidefinite programming relaxations for solving inverse problems.
\end{abstract}


\input{sc_intro}
\input{sc_algo}
\input{sc_analysis}
\input{sc_numexp}

\section*{Appendix}
\appendix

\input{appendix_riptrivials}
\input{appendix_randesmb}
\input{appendix_sinkhorn}

\input{appendix_randlinmap}
\input{appendix_normstab}
\input{appendix_variety}
\input{appendix_pseudoinverse}
\input{appendix_induction}
\input{appendix_stopping}

\section*{Acknowledgements}

The authors were supported in part by NSF Career award CCF-1350590, by Air Force Office of Scientific Research grants FA9550-14-1-0098 and FA9550-16-1-0210, by a Sloan research fellowship, and an A*STAR (Agency for Science, Technology, and Research, Singapore) fellowship.  The authors thank Joel Tropp for a helpful remark that improved the result in Proposition \ref{thm:concentrationwishart}.

\bibliography{bib_psdreg}
\end{document}

%% file: macros.tex
\usepackage{graphicx,amssymb,amsfonts,amsmath}
\usepackage{multicol,multirow}
\usepackage{algorithm,algcompatible}

\newcommand{\ones}{\mathbf{1}}

\newcommand{\bx}{\mathbf{x}}
\newcommand{\ba}{\mathbf{a}}
\newcommand{\be}{\mathbf{e}}
\newcommand{\by}{\mathbf{y}}
\newcommand{\bz}{\mathbf{z}}
\newcommand{\bu}{\mathbf{u}}
\newcommand{\bv}{\mathbf{v}}
\newcommand{\bs}{\mathbf{s}}
\newcommand{\bw}{\mathbf{w}}
\newcommand{\bt}{\mathbf{t}}

\newcommand{\R}{\mathbb{R}}
\renewcommand{\S}{\mathbb{S}}

\newcommand{\A}{\mathcal{A}}
\newcommand{\ct}{\mathcal{T}}
\newcommand{\cp}{\mathcal{P}}
\newcommand{\cs}{\mathcal{S}}
\newcommand{\cv}{\mathcal{V}}
\renewcommand{\L}{\mathcal{L}}
\newcommand{\W}{\mathcal{W}}

\newcommand{\sfe}{\mathsf{E}}
\newcommand{\sfd}{\mathsf{D}}
\newcommand{\sff}{\mathsf{F}}
\newcommand{\sfg}{\mathsf{H}}
\newcommand{\sfi}{\mathsf{I}}

\newcommand{\sfq}{\mathsf{Q}}

\renewcommand{\aa}{\mathtt{A}}
\newcommand{\bb}{\mathtt{B}}
\newcommand{\xx}{\mathtt{X}}

\newcommand{\dd}{\mathtt{D}}

\newcommand{\err}{G}

\newcommand{\eu}{\ell_2}
\newcommand{\fro}{\ell_2}
\newcommand{\spec}{2}

\newcommand{\coveig}{\Lambda}
\newcommand{\covsup}{\Delta}
\newcommand{\roc}{\Omega}

\newcommand{\argmin}{\mathrm{arg~min}}
\newcommand{\botimes}{\boldsymbol \otimes}
\newcommand{\ssy}{\mathcal{G}}
\newcommand{\sas}{\mathcal{H}}
\newcommand{\spi}{\mathcal{I}}

\newcommand{\tstar}{\ct(X^\star)} 
\newcommand{\ensb}{\mathfrak{X}^{\star}} 
\newcommand{\ensbfull}{\{{X^{(j)}}^\star\}_{j=1}^n}
\newcommand{\ensa}{\mathfrak{X}} 
\newcommand{\ensafull}{\{X^{(j)}\}_{j=1}^n}

\DeclareMathAlphabet      {\mathsfit}{OT1}{cmss}{m}{it}
\newcommand{\sfiw}{\mathsfit{W}}
\newcommand{\sfin}{\mathsfit{N}}
\newcommand{\sfim}{\mathsfit{M}}
\newcommand{\sfiq}{\mathsfit{Q}}

\DeclareMathAlphabet      {\mathsfbi}{OT1}{cmss}{m}{n}
\newcommand{\cpt}{\mathsfbi{T}}

\newcommand{\aut}{\mathrm{End}}

%% file: sc_intro.tex
\section{Introduction} \label{sec:intro}

Regularization techniques are widely employed in the solution of inverse problems in data analysis and scientific computing due to their effectiveness in addressing difficulties due to ill-posedness.  In their most common manifestation, these methods take the form of penalty functions added to the objective in optimization-based approaches for solving inverse problems.  The purpose of the penalty function is to induce a desired structure in the solution, and these functions are specified based on prior domain-specific expertise.  For example, regularization is useful for promoting smoothness, sparsity, low energy, and large entropy in solutions to inverse problems in image analysis, statistical model selection, and the geosciences \cite{BDE:09,CanRec:09,CRT:06,CRPW:12,CDS:98,Don:06,MeiBuh:06,RFP:10,Tib:94}.  In this paper, we study the question of \emph{learning} suitable regularization functions from data in settings in which precise domain knowledge is not directly available.  The regularizers obtained using our framework are specified as convex functions that can be computed efficiently via semidefinite programming, and therefore they can be employed in tractable convex optimization approaches for solving inverse problems.

We begin our discussion by highlighting the geometric aspects of regularizers that make them effective in promoting a desired structure.  In particular, we focus on a family of convex regularizers that are useful for inducing a general form of sparsity in solutions to inverse problems.  Sparse data descriptions provide a powerful formalism for specifying low-dimensional structure in high-dimensional data, and they feature prominently in a range of problem domains.  For example, natural images are often well-approximated by a small number of wavelet coefficients, financial time series may be characterized by low-complexity factor models, and a small number of genetic markers may constitute a signature for disease.  Concretely, suppose $\A \subset \R^d$ is a (possibly infinite) collection of elementary building blocks or atoms.  Then $\by \in \R^d$ is said to have a sparse representation using the atomic set $\A$ if $\by$ can be expressed as follows:
\begin{equation*}
\by = \sum_{i=1}^k c_i \ba_i, ~~~ \ba_i \in \A, c_i \geq 0,
\end{equation*}
for a relatively small number $k$.  As an illustration, if $\A =\{\pm \be^{(j)}\}_{j=1}^d \subset \R^d$ is the collection of signed standard basis vectors in $\R^d$, then concisely described objects with these atoms are those vectors in $\R^d$ consisting of a small number of nonzero coordinates.  Similarly, if $\A$ is the set of rank-one matrices, then the corresponding sparsely represented entities are low-rank matrices; see \cite{CRPW:12} for a more exhaustive collection of examples.  An important virtue of sparse descriptions based on an atomic set $\A$ is that employing the \emph{atomic norm} induced by $\A$ --- the gauge function of the atomic set $\A$ --- as a regularizer in inverse problems offers a natural convex optimization approach for obtaining solutions that have a sparse represention using $\A$ \cite{CRPW:12}.  Continuing with the examples of vectors with few nonzero coordinates and of low-rank matrices, regularization with the $\ell_1$ norm (the gauge function of the signed standard basis vectors) and with the matrix nuclear norm (the gauge function of the unit-Euclidean-norm rank-one matrices) are prominent techniques for promoting the corresponding sparse descriptions in solutions to inverse problems \cite{CanRec:09,CRT:06,CDS:98,Don:06,Faz:02,MeiBuh:06,RFP:10,Tib:94}.  The reason for the effectiveness of atomic norm regularization is the favorable facial structure of the convex hull of $\A$, which has the feature that all its low-dimensional faces contain points that have a sparse description using $\A$.  Indeed, in many contemporary data analysis applications the solutions of regularized optimization problems with generic input data tend to lie on low-dimensional faces of sublevel sets of the regularizer \cite{CT:06,Don:06,RFP:10}.  Based on this insight, atomic norm regularization has been shown to be effective in a range of tasks such as statistical denoising, model selection, and system identification \cite{BTR:13,OymHas:16,SBTR:12}.

The difficulty with employing an atomic norm regularizer in practice is that one requires prior domain knowledge of the atomic set $\A$ -- the extreme points of the atomic norm ball -- that underlies a sparse description of the desired solution in an inverse problem.  While such information may be available based on domain expertise in some problems (e.g., certain classes of signals having a sparse representation in a Fourier basis), identifying a suitable atomic set is challenging for many contemporary data sets that are high-dimensional and are typically presented to an analyst in an unstructured fashion.  In this paper, we study the question of learning a suitable regularizer directly from observations $\{\by^{(j)}\}_{j=1}^n \subset \R^d$ of a collection of structured signals or models of interest.  Specifically, as motivated by the preceding discussion, our objective is to identify a norm $\|\cdot\|$ in $\R^d$ such that each $\by^{(j)} / \|\by^{(j)}\|$ lies on a low-dimensional face of the unit ball of $\|\cdot\|$.  An equivalent formulation of this question in terms of extreme points is that we want to obtain an atomic set $\A$ such that each $\by^{(j)}$ has a sparse representation using $\A$; the corresponding regularizer is simply the atomic norm induced by $\A$.  A norm with these characteristics is adapted to the structure contained in the data $\{\by^{(j)}\}_{j=1}^n$, and it can be used subsequently as a regularizer in inverse problems to promote solutions with the same type of structure as in the collection $\{\by^{(j)}\}_{j=1}^n$.

When considered in full generality, our question is somewhat ill-posed for several reasons.  First, if $\|\cdot\|$ is a norm that satisfies the properties described above with respect to the data $\{\by^{(j)}\}_{j=1}^n$, then so does $\alpha \|\cdot\|$ for any positive scalar $\alpha$.  This issue is addressed by learning a norm from a suitably scaled class of regularizers.  A second source of difficulty is that the Euclidean norm $\|\cdot\|_{\ell_2}$ trivially satisfies our requirements for a regularizer as each $\by^{(j)} / \|\by^{(j)}\|_{\ell_2}$ is an extreme point of the Euclidean norm ball in $\R^d$; indeed, this is the regularizer employed in ridge regression.  The atomic set in this case is the collection of all points with Euclidean norm equal to one, i.e., the dimension of this set is $d-1$.  However, data sets in many applications throughout science and engineering are well-approximated as sparse combinations of elements of atomic sets of much smaller dimension \cite{Bar:93,BDE:09,CRPW:12,DevTem:96,Jon:92,OlsFie:96,Pis:81}.  Identifying such lower-dimensional atomic sets is critical in inverse problems arising in high-dimensional data analysis in order to address the curse of dimensionality; in particular, as discussed in some of these preceding references, the benefits of atomic norm regularization in problems with large ambient dimension $d$ are a consequence of measure concentration phenomena that crucially rely on the small dimensionality of the associated atomic set in comparison to $d$.  We circumvent this second difficulty in learning a regularizer by considering atomic sets with appropriately bounded dimension.  A third challenge with our question as it is stated is that the gauge function of the set $\{\pm \by^{(j)} / \|\by^{(j)}\|_{\ell_2}\}_{j=1}^n$ also satisfies the requirements for a suitable atomic norm as each $\by^{(j)} / \|\by^{(j)}\|_{\ell_2}$ is an extreme point of the unit ball of this regularizer.  However, such a regularizer suffers from overfitting and does not generalize well as it is excessively tuned to the data set $\{\by^{(j)}\}_{j=1}^n$.  Further, for large $n$ this gauge function becomes intractable to characterize and it does not offer a computationally efficient approach for regularization.  We overcome this complication by considering regularizers that have effectively parametrized sets of extreme points, and consequently are tractable to compute.

The problem of learning a suitable polyhedral regularizer -- an atomic norm with a unit ball that is a polytope -- from data points $\{\by^{(j)}\}_{j=1}^n$ corresponds to identifying an appropriate \emph{finite} atomic set to concisely describe each $\by^{(j)}$.  This problem is equivalent to the question of `dictionary learning' (also called `sparse coding') on which there is a substantial amount of prior work \cite{AAJN:16,AAN:17,AEB:06,AGTM:15,AGM:14,BKS:15,GJBKS:16,OlsFie:96,Sch:14,Sch:16,SWW:12,SQW:16a,SQW:16b,VMB:11} (see also the survey articles in \cite{Ela:10,MBP:14}).  To see this connection, suppose without loss of generality that we parametrize a finite atomic set via a matrix $L \in \R^{d \times p}$ so that the columns of $L$ and their negations specify the atoms.  The associated atomic norm ball is the image under $L$ of the $\ell_1$ ball in $\R^p$.  The columns of $L$ are typically scaled to have unit Euclidean norm to address the scaling issues mentioned previously (see Section \ref{sec:algorithm_cdl}).  The number of columns $p$ may be larger than $d$ (i.e., the `overcomplete' regime), and it controls the complexity of the atomic set as well as the computational tractability of describing the atomic norm.  With this parametrization, learning a polyhedral regularizer to promote the type of structure contained in $\{\by^{(j)}\}_{j=1}^n$ may be viewed as obtaining a matrix $L$ (given a target number of columns $p$) such that each $\by^{(j)}$ is well-approximated as $L \bx^{(j)}$ for a vector $\bx^{(j)} \in \R^p$ with few nonzero coordinates. Computing such a representation of the data is precisely the objective in dictionary learning, although this problem is typically not phrased as a quest for a polyhedral regularizer in the literature.  We remark further on some recent algorithmic developments in dictionary learning in Sections \ref{sec:relatedwork_dictionarylearning} and \ref{sec:algorithm_cdl}, and we contrast these with the methods proposed in the present paper.

\subsection{From Polyhedral to Semidefinite Regularizers} \label{sec:intro_semidefleadup}

The objective of this paper is to investigate the problem of learning more general non-polyhedral atomic norm regularizers; in other words, the associated set of extreme points may be \emph{infinite}.  On the approximation-theoretic front, infinite atomic sets offer the possibility of concise descriptions of data sets with much richer types of structure than those with a sparse representation using finite atomic sets; in turn, the associated regularizers could promote a broader class of structured solutions to inverse problems than polyhedral regularizers.  On the computational front, many families of convex optimization problems beyond linear programs can be solved tractably and reliably \cite{NesNem:94}.  However, building on the challenges outlined previously, there are two important factors in identifying non-polyhedral regularizers from data.  First, it is crucial that any infinite atomic set $\A$ we consider has an effective parametrization so that it is tractable to characterize data that have a sparse representation using the elements of $\A$.  Second, we require that the convex hull of the atomic set $\A$ has an efficient description so that the associated atomic norm provides a computationally tractable regularizer.  As described next, we address these concerns by considering atomic sets that are efficiently parametrized as algebraic varieties (of a particular form) and that have convex hulls with tractable semidefinite descriptions.  Thus, previous efforts in the dictionary learning literature on identifying finite atomic sets may be viewed as learning zero-dimensional ideals, whereas our approach corresponds to learning atomic sets that are larger-dimensional varieties.  From a computational viewpoint, dictionary learning provides atomic norm regularizers that are computed via linear programming, while our framework leads to semidefinite programming regularizers.  Consequently, although our framework is based on a much richer family of atomic sets in comparison with the finite sets considered in dictionary learning, we still retain efficiency of parametrization and computational tractability based on semidefinite representability.

Formally, we consider atomic sets in $\R^d$ that are images of rank-one matrices:
\begin{equation}
\A_{q}(\L) = \left\{\L (\bu \bv') ~|~ \bu, \bv \in \R^q, ~ \|\bu\|_{\ell_2} = 1, \|\bv\|_{\ell_2} = 1 \right\}, \label{eq:lowrankatoms}
\end{equation}
where $\L : \R^{q \times q} \rightarrow \R^d$ specifies a linear map. We focus on settings in which the dimension $q$ is such that $q^2 > d$, so the atomic sets $\A_{q}(\L)$ that we study in this paper are projections of rank-one matrices from a larger-dimensional space (in analogy to the overcomplete regime in dictionary learning).  By construction, elements of $\R^d$ that have a sparse representation using the atomic set $\A_{q}(\L)$ are those that can be specified as the image under $\L$ of \emph{low-rank matrices} in $\R^{q \times q}$.  As the convex hull of unit-Euclidean-norm rank-one matrices in $\R^{q \times q}$ is the nuclear norm ball in $\R^{q \times q}$, the corresponding atomic norm ball is given by:
\begin{equation}
\mathrm{conv}\left(\A_{q}(\L)\right) = \left\{ \L (X) ~|~ X \in \R^{q \times q}, ~ \|X\|_\star \leq 1 \right\}, \label{eq:nuclearimage}
\end{equation}
where $\|X\|_\star := \sum_{i} \sigma_i(X)$.  As the nuclear norm ball has a tractable semidefinite description \cite{Faz:02,RFP:10}, the atomic norm induced by $\A_{q}(\L)$ can be computed efficiently using semidefinite programming.

Given a collection of data points $\{\by^{(j)}\}_{j=1}^n \subset \R^d$ and a target dimension $q$, our goal is to find a linear map $\L : \R^{q \times q} \rightarrow \R^d$ such that each $\by^{(j)}$, upon normalization by the gauge function of $\A_{q}(\L)$, lies on a low-dimensional face of $\mathrm{conv}(\A_{q}(\L))$.  For each $\by^{(j)}$ to have this property, it must have a sparse representation using the atomic set $\A_{q}(\L)$; that is, there must exist a low-rank matrix $X^{(j)} \in \R^{q \times q}$ with $\by^{(j)} = \L (X^{(j)})$.  The matrix $X^{(j)}$ provides a concise description of $\by^{(j)} \in \R^d$ in the higher-dimensional space $\R^{q \times q}$.  Consequently, the problem of learning a semidefinite-representable regularizer with a unit ball that is a linear image of the nuclear norm ball may be phrased as one of \emph{matrix factorization}.  In particular, let $Y = [\by^{(1)} | \cdots | \by^{(n)}] \in \R^{d \times n}$ denote the data matrix, and let $\L_{i} \in \R^{q \times q}, ~ i=1,\dots,d$ be the matrix that specifies the linear functional corresponding to the $i$'th component of a linear map $\L : \R^{q \times q} \rightarrow \R^d$.  Then our objective can be viewed as one of finding a collection of matrices $\{\L_{i}\}_{i=1}^d \subset \R^{q \times q}$ specifying linear functionals and a set of low-rank matrices $\ensafull \subset \R^{q \times q}$ specifying concise descriptions such that:
\begin{equation}
Y_{i,j} = \langle \L_{i}, X^{(j)} \rangle ~~~ i=1,\dots,d, ~ j=1,\dots,n. \label{eq:lowrankfactor}
\end{equation}
Here $\langle A,B \rangle = \mathrm{trace}(A^{\prime} B)$ denotes the trace inner product between matrices.  Note the distinction with dictionary learning in which one seeks a factorization of the data matrix $Y$ such that the $X^{(j)}$'s are sparse vectors as opposed to low-rank matrices as in our approach.  Figure \ref{fig:comparison} summarizes the key differences between dictionary learning and the present paper.

\begin{figure*}
\centering
\small
\begin{tabular}{|c|c|c|}
\hline

{\multirow{2}{*}{}} & {\multirow{2}{*}{\bf Dictionary learning}} & {\multirow{2}{*}{\bf Our work}}
\\
& &
\\

\hline
\hline

{\multirow{4}{*}{Atomic set}} & $\{\pm L \be^{(i)} ~|~ \be^{(i)} \in \R^p ~\text{is~the}~ i\text{'th}$& {$\{ \L (\bu \bv') ~|~ \bu, \bv \in \R^q,$}
\\
& $\text{standard~basis~vector}\}$ & $\hspace{0.1in} \|\bu\|_{\ell_2} = \|\bv\|_{\ell_2} = 1 \}$
\\
& {\multirow{2}{*}{$L : \R^p \rightarrow \R^d$ (linear map)}} & {\multirow{2}{*}{$\L: \R^{q \times q} \rightarrow \R^d$ (linear map)}}
\\
& &
\\

\hline

Algebraic/geometric & {\multirow{2}{*}{Zero-dimensional ideal}} & {\multirow{2}{*}{Image of determinantal variety}}
\\
structure of atoms & &
\\

\hline

Concisely specified & {Image under $L$ of} & {Image under $\L$ of}
\\
data using atomic set & {sparse vectors} & {low-rank matrices}
\\

\hline

{\multirow{2}{*}{Atomic norm ball}} & {\multirow{2}{*}{$\left\{ L \bx ~|~ \bx \in \R^p, ~ \|\bx\|_{\ell_1} \leq 1 \right\}$}} & {\multirow{2}{*}{$\left\{ \L (X) ~|~ X \in \R^{q \times q}, ~ \|X\|_\star \leq 1 \right\}$}}
\\
 & &
\\

\hline

Computing atomic & {\multirow{2}{*}{Linear programming}} & {\multirow{2}{*}{Semidefinite programming}}
\\
norm regularizer & &
\\

\hline

Learning regularizer & Identify $L$ and sparse $\bx^{(j)} \in \R^p$ & Identify $\L$ and low-rank $X^{(j)} \in \R^{q \times q}$
\\
from data $\{\by^{(j)}\}_{j=1}^n$ & such that $\by^{(j)} \approx L \bx^{(j)}$ for each $j$ & such that $\by^{(j)} \approx \L(X^{(j)})$ for each $j$
\\

\hline
\end{tabular}
\caption{A comparison between prior work on dictionary learning and the present paper.} \label{fig:comparison}
\end{figure*}

\subsection{An Alternating Update Algorithm for Matrix Factorization} \label{sec:intro_am}

A challenge with identifying a semidefinite regularizer by factoring a given data matrix as in \eqref{eq:lowrankfactor} is that such a factorization is not unique.  Specifically, consider any linear map $\sfim : \R^{q \times q} \rightarrow \R^{q \times q}$ that is a rank-preserver, i.e., $\mathrm{rank}(\sfim(X)) = \mathrm{rank}(X)$ for all $X \in \R^{q \times q}$; examples of rank-preservers include operators that act via conjugation by non-singular matrices and the transpose operation.  If each $\by^{(j)} = \L (X^{(j)})$ for a linear map $\L$ and low-rank matrices $\ensafull$, then we also have that each $\by^{(j)} = \L \circ \sfim^{-1} (\sfim(X^{(j)}))$, where by construction each $X^{(j)}$ has the same rank as the corresponding $\sfim(X^{(j)})$.  This non-uniqueness presents a difficulty as the image of the nuclear norm ball under a linear map $\L$ is, in general, different than it is under $\L \circ {\sfim}^{-1}$ for an arbitrary rank-preserver $\sfim$.  Consequently, due to its invariances the factorization \eqref{eq:lowrankfactor} does not uniquely specify a regularizer.  We investigate this point in Section \ref{sec:algorithm_normalization} by analyzing the structure of rank-preserving linear maps, and we describe an approach to associate a unique regularizer to a family of linear maps obtained from equivalent factorizations.  Our method entails putting linear maps in an appropriate `canonical' form using the Operator Sinkhorn iterative procedure, which was developed by Gurvits to solve certain quantum matching problems \cite{Gur:04}; this algorithm is an operator analog of the diagonal congruence scaling technique for nonnegative matrices developed by Sinkhorn \cite{Sin:64}.

In Section \ref{sec:algorithm} we describe an alternating update algorithm to compute a factorization of the form \eqref{eq:lowrankfactor}.  With the $\L_{i}$'s fixed, updating the $X^{(j)}$'s entails the solution of affine rank minimization problems.  Although this problem is intractable in general \cite{Nat:93}, in recent years several tractable heuristics have been developed and proven to succeed under suitable conditions \cite{GM:11,JMD:10,RFP:10}.  With the $X^{(j)}$'s fixed, the $\L_{i}$'s are updated by solving a least-squares problem followed by an application of the Operator Sinkhorn iterative procedure to put the map $\L$ in a canonical form as described above.  Our alternating update approach is a generalization of methods that are widely employed in dictionary learning for identifying finite atomic sets (see Section \ref{sec:algorithm_cdl}).

Section \ref{sec:analysis} contains the main theorem of this paper on the local linear convergence of our alternating update algorithm.  Specifically, suppose a collection of data points $\{\by^{(j)}\}_{j=1}^n \subset \R^d$ is generated as $\by^{(j)} = \L^\star ({X^{(j)}}^\star ), ~ j=1,\dots,n$ for a linear map $\L^\star : \R^{q \times q} \rightarrow \R^d$ that is nearly isometric restricted to low-rank matrices (formally, $\L^\star$ satisfies a \emph{restricted isometry property} \cite{RFP:10}) and a collection $\{{X^{(j)}}^\star \}_{j=1}^n \subset \R^{q \times q}$ of low-rank matrices that is isotropic in a well-defined sense.  Given the data $\{\by^{(j)}\}_{j=1}^n$ as input, our alternating update approach is locally linearly convergent to a linear map $\hat{\L} : \R^{q \times q} \rightarrow \R^d$ with the property that the image of the nuclear norm ball in $\R^{q \times q}$ under $\hat{\L}$ is equal to its image under $\L^\star$, i.e., our procedure identifies the appropriate regularizer that promotes the type of structure contained in the data $\{\by^{(j)}\}_{j=1}^n$; see Theorem \ref{thm:localconvergence}.  Our analysis relies on geometric aspects of determinantal varieties (in particular tangent spaces with respect to these varieties) and their relation to stability properties of Operator Sinkhorn scaling.

We demonstrate the utility of our framework with a series of experimental results on synthetic as well as real data in Section \ref{sec:numexp}.

\subsection{Related Work} \label{sec:intro_relatedwork}

\subsubsection{Dictionary Learning} \label{sec:relatedwork_dictionarylearning}
As outlined above, our approach for learning a regularizer from data may be viewed as a semidefinite programming generalization of dictionary learning.  The alternating update algorithm we propose in Section \ref{sec:algorithm_am} for computing a factorization \eqref{eq:lowrankfactor} generalizes similar methods previously developed for dictionary learning \cite{AAJN:16,AEB:06,AGTM:15,OlsFie:96} (see Section \ref{sec:algorithm_cdl}), and the local convergence analysis of our algorithm in Section \ref{sec:analysis} also builds on previous analyses for dictionary learning \cite{AAJN:16,AGTM:15}.  In contrast to these previous results, the development and the analysis of our method in the present paper are more challenging due to the invariances and associated identifiability issues underlying the factorization \eqref{eq:lowrankfactor}, which necessitate the incorporation of the Operator Sinkhorn scaling procedure in our algorithm.

An unresolved matter in our paper -- one that has been investigated previously in the context of dictionary learning -- is the question of a suitable initialization for our algorithm.  In particular, our theory states that our algorithm exhibits linear convergence to the desired solution provided the initial guess is sufficiently close to a linear map that specifies the correct regularizer (in an appropriate metric).  We employ random initializations in our experiments with real data in Section \ref{sec:numexp_raw}, and these are useful in identifying effective semidefinite regularizers that outperform polyhedral regularizers obtained via dictionary learning.  Random initialization is the most common technique utilized in practice in dictionary learning as well as in many other structured matrix factorization problems arising in data analysis.  To build support for this idea, several researchers have proven that random initialization succeeds with high probability in recovering a desired factorization under suitable conditions in a number of problems \cite{GLM:16,SQW:17}, including in a restricted form of dictionary learning \cite{SQW:16a,SQW:16b} in which the polyhedral regularizer is specified as the image of the $\ell_1$ ball under an invertible linear map (as described previously, dictionary learning in full generality allows for polyhedral regularizers that may be specified as an image of the $\ell_1$ ball under a many-to-one linear map).  In a different direction, some recent papers also describe data-driven initialization strategies for dictionary learning based on variants of clustering \cite{AAN:17,AGM:14}.  It would be of interest to develop both these sets of ideas in our context, and we comment on this point in Section \ref{sec:discussion}.

\subsubsection{Lifts of Convex Sets}
A second body of work with which our paper is conceptually related is the literature on lift-and-project representations (or extended formulations) of convex sets.  A tractable lift-and-project representation refers to a description of a `complicated' convex set in $\R^d$ as the projection of a more concisely specified convex set in $\R^{d'}$, with the lifted dimension $d'$ not being too much larger than the original dimension $d$.  As discussed in \cite{GPT:13,Yan:91}, obtaining a suitably structured factorization -- of a different nature than that considered in the present paper -- of the \emph{slack matrix} of a polytope (and more generally, of the slack operator of a convex set) corresponds to identifying an efficient lift-and-project description of the polytope.  On the other hand, we seek a structured factorization of a \emph{data matrix} to identify a convex set (i.e., the unit ball of a regularizer) with an efficient extended formulation and with the additional requirement that the data points (upon suitable scaling) lie on low-dimensional faces of the set.  This latter stipulation arises in our context from data analysis considerations, and it is a distinction between our setup and the optimization literature on extended formulations.

\subsubsection{Sinkhorn Scaling}
A third topic with which our paper has synergies -- and to which we make contributions in the course of our analysis -- is the literature on Sinkhorn scaling.  This algorithm is an iterative procedure for transforming an entrywise nonnegative matrix to a doubly stochastic matrix by diagonal congruence scaling \cite{Sin:64}.  There is a substantial body of work on the properties of this algorithm (see \cite{Ide:16} and the references therein) as well as on its applications in domains such as combinatorial optimization (approximating the permanent of a matrix \cite{LSW:00}) and data analysis (efficiently computing distances between probability distributions  \cite{Cut:13}).  The operator analog of Sinkhorn scaling was developed by Gurvits and this work was motivated by certain operator analogs of the bipartite matching problem that arise in matroid theory \cite{Gur:04}.  To the best of our knowledge, our work represents the first application of Operator Sinkhorn scaling in a problem in data analysis.  Further, in our investigation of the properties of Algorithm \ref{alg:osi}, we describe results on the stability of Operator Sinkhorn scaling; these may be of independent interest beyond the specific context of our paper (see Appendix \ref{apx:sinkhornstability}).

\subsection{Paper Outline}
In Section \ref{sec:algorithm} we discuss our alternating update algorithm for computing the factorization \eqref{eq:lowrankfactor} based on an analysis of the invariances arising in \eqref{eq:lowrankfactor}.  Section \ref{sec:analysis} gives the main theoretical result concerning the local linear convergence of the algorithm described in Section \ref{sec:algorithm}, and Section \ref{sec:numexp} describes numerical results obtained using our algorithm.  We conclude with a discussion of further research directions in Section \ref{sec:discussion}.

\paragraph{Notation}  We denote the Euclidean norm by $\|\cdot\|_{\ell_2}$.
We denote the operator or spectral norm by $\|\cdot\|_{\spec}$.  The $k$'th largest singular value of a linear map is denoted by $\sigma_k(\cdot)$, and the largest and smallest eigenvalues of a self-adjoint linear map are denoted by $\lambda_{\max}(\cdot)$ and $\lambda_{\min}(\cdot)$ respectively.  The space of $q \times q$ symmetric matrices is denoted $\mathbb{S}^q$ and the set of $q \times q$ symmetric positive-definite matrices is denoted $\mathbb{S}^q_{++}$.  The projection map onto a subspace $\cv$ is denoted $\cp_\cv$.  The restriction of a linear map $M$ to a subspace $\cv$ is denoted by $M_{\cv}$.  Given a self-adjoint linear map $M : \cv \rightarrow \cv$ with $\cv$ being a subspace of a vector space $\bar{\cv}$, we denote the extension of $M$ to $\bar{\cv}$ by $[M]_{\bar{\cv}} : \bar{\cv} \rightarrow \bar{\cv}$; the component in $\cv$ of the image of any $\bx \in \bar{\cv}$ under this map is $M \cp_{\cv}(\bx)$, while the component in $\cv^\perp$ is the origin.  
Given a vector space $\cv$, we denote the set of linear operators from $\cv$ to $\cv$ by $\aut(\cv)$.  Given matrices $A, B \in \R^{q \times q}$, the linear map $A \boxtimes B \in \aut(\R^{q \times q})$ is specified as $A \boxtimes B : X \rightarrow \langle B, X \rangle A$.  The Kronecker product between two linear maps is specified using the standard $\botimes$ notation.  For a collection of matrices $\ensa := \ensafull \subset \R^{q \times q}$, the covariance is specified as $\mathsf{\Sigma}(\ensa) = \frac{1}{n}\sum_{j=1}^n X^{(j)} \boxtimes X^{(j)}$.  Two quantities associated to this covariance that play a role in our analysis are $\coveig(\ensa) = \frac{1}{2}(\lambda_{\max}(\mathsf{\Sigma}(\ensa)) + \lambda_{\min}(\mathsf{\Sigma}(\ensa)) )$ and $\covsup(\ensa) = \frac{1}{2}(\lambda_{\max}(\mathsf{\Sigma}(\ensa)) - \lambda_{\min}(\mathsf{\Sigma}(\ensa)) )$.  Given a matrix $X \in \R^{q \times q}$ of rank $r$, the tangent space at $X$ with respect to the algebraic variety of $q \times q$ matrices of rank at most $r$ is specified as\footnote{A rank-$r$ matrix $X \in \R^{q \times q}$ is a smooth point with respect to the variety of $q \times q$ matrices of rank at most $r$.}:
\begin{equation*}
\begin{aligned}
\ct(X) = \{ X A + B X ~|~ A,B \in \R^{q \times q} \}.
\end{aligned}
\end{equation*} 

%% file: sc_algo.tex
\section{An Alternating Update Algorithm for Learning Semidefinite Regularizers} \label{sec:algorithm}

In this section we describe an alternating update algorithm to factor a given data matrix $Y = [\by^{(1)} | \cdots | \by^{(n)}] \in \R^{d \times n}$ as in \eqref{eq:lowrankfactor}.  As discussed previously, the difficulty with obtaining a semidefinite regularizer using a factorization \eqref{eq:lowrankfactor} is the existence of infinitely many equivalent factorizations due to the invariances underlying \eqref{eq:lowrankfactor}.  We begin by investigating and addressing this issue in Sections \ref{sec:algorithm_identifiablity} and \ref{sec:algorithm_normalization}, and then we discuss our algorithm to obtain a regularizer in Section \ref{sec:algorithm_am}.  We contrast our method with techniques that have previously been developed in the context of dictionary learning in Section \ref{sec:algorithm_cdl}.

\subsection{Identifiability Issues} \label{sec:algorithm_identifiablity}

Building on the discussion in the introduction, for a linear map $\L : \R^{q \times q} \rightarrow \R^d$ obtained from the factorization \eqref{eq:lowrankfactor} and for any linear rank-preserver $\sfim : \R^{q \times q} \rightarrow \R^{q \times q}$, there exists an equivalent factorization in which the linear map is $\L \circ \sfim$ (note that ${\sfim}^{-1}$ is also a rank-preserver if $\sfim$ is a rank-preserver).  As the image of the nuclear norm ball in $\R^{q \times q}$ is not invariant under an arbitrary rank-preserver, a regularizer cannot be obtained uniquely from a factorization due to the existence of equivalent factorizations that lead to non-equivalent regularizers.  To address this difficulty, we describe an approach to associate a \emph{unique} regularizer to a family of linear maps obtained from equivalent factorizations.  We begin by analyzing the structure of rank-preserving linear maps based on the following result \cite{MarMoy:59}:

\begin{theorem}(\cite[Theorem 1]{MarMoy:59},\cite[Theorem 9.6.2]{Tun:00}) \label{thm:external_rankpreserver}  An invertible linear operator $\sfim: \R^{q \times q} \rightarrow \R^{q \times q}$ is a rank-preserver if and only if $\sfim$ is of one of the following two forms for non-singular matrices $W_1,W_2 \in \R^{q \times q}$: $\sfim(X) = W_1 X W_2$ or $\sfim(X) = W_1 X' W_2$.
\end{theorem}

This theorem brings the preceding discussion into sharper focus, namely, that the lack of identifiability boils down to the fact that the nuclear norm is not invariant under conjugation of its argument by arbitrary non-singular matrices.  However, we note that the nuclear norm ball is invariant under the transpose operation and under conjugation by orthogonal matrices.  This observation leads naturally to the idea of employing the \emph{polar decomposition} to describe a rank-preserver:

\begin{corollary} \label{thm:external_rankpreserverpolardecomp}
	Every rank-preserver $\sfim: \R^{q \times q} \rightarrow \R^{q \times q}$ can be uniquely decomposed as $\sfim = \sfim^{\mathrm{or}} \circ \sfim^{\mathrm{pd}}$ for rank-preservers $\sfim^{\mathrm{pd}} : \R^{q \times q} \rightarrow \R^{q \times q}$ and $\sfim^{\mathrm{or}} : \R^{q \times q} \rightarrow \R^{q \times q}$ with the following properties:
	\begin{itemize}
		\item The operator $\sfim^{\mathrm{pd}}$ is specified as $\sfim^{\mathrm{pd}}(X) = P_1 X P_2$ for some positive-definite matrices $P_1, P_2 \in \S^q_{++}$.
		
		\item The operator $\sfim^{\mathrm{or}}$ is of one of the following two forms for orthogonal matrices $U_1,U_2 \in \R^{q \times q}$: $\sfim^{\mathrm{or}}(X) = U_1 X U_2$ or $\sfim^{\mathrm{or}}(X) = U_1 X' U_2$.
	\end{itemize}
\end{corollary}

\begin{proof}
The result follows by combining Theorem \ref{thm:external_rankpreserver} with the polar decomposition. \qed
\end{proof}

We refer to rank-preservers of the type $\sfim^{\mathrm{pd}}$ in this corollary as \emph{positive-definite rank-preservers} and to those of the type $\sfim^{\mathrm{or}}$ as \emph{orthogonal rank-preservers}.  This corollary highlights the point that the key source of difficulty in identifying a regularizer uniquely from a factorization is due to positive-definite rank-preservers.  A natural approach to address this challenge is to put a given linear map $\L$ into a `canonical' form that removes the ambiguity due to positive-definite rank-preservers.  In other words, we seek a distinguished subset of \emph{normalized} linear maps with the following properties: $(a)$ for a linear map $\L$, the set $\{\L \circ \sfim^{\mathrm{pd}} ~|~ \sfim^{\mathrm{pd}} ~\text{is a positive-definite rank-preserver}\}$ intersects the collection of normalized maps at precisely one point; and $(b)$ for any normalized linear map $\L$, every element of the set $\{\L \circ \sfim^{\mathrm{or}} ~|~ \sfim^{\mathrm{or}} ~\text{is an orthogonal rank-preserver}\}$ is also normalized.  The following definition possesses both of these attributes:

\begin{definition}
Let $\L : \R^{q \times q} \rightarrow \R^d$ be a linear map, and let $\L_{i} \in \R^{q \times q}, ~ i=1,\dots,d$ be the component linear functionals of $\L$.  Then $\L$ is said to be \emph{normalized} if $\sum_{i=1}^d \L_{i} {\L_{i}}' = qI$ and $\sum_{i=1}^d {\L_{i}}' \L_{i} = qI$.
\end{definition}

The utility of this definition in resolving our identifiability issue is based on a paper by Gurvits  \cite{Gur:04}.  Specifically, for a generic linear map $\L : \R^{q \times q} \rightarrow \R^d$, the results in \cite{Gur:04} imply that there exists a \emph{unique} positive-definite rank-preserver $\sfin_\L : \R^{q \times q} \rightarrow \R^{q\times q}$ so that $\L \circ \sfin_\L$ is normalized (see Corollary~\ref{thm:external_normalizablelinearmaps} in the sequel); this feature address our first requirement above.  One can also check that the second requirement above is satisfied by this definition -- any normalized linear map composed with any orthogonal rank-preserver is also normalized.  Further, the collection of normalized maps defined above may be viewed as an affine algebraic variety specified by polynomials of degree two.  One can check that any notion of normalization (specified as a real variety) that satisfies the two attributes described previously cannot be an affine space, and therefore must be specified by polynomials of degree at least two.  Consequently, our definition of normalization is in some sense also as `simple' as possible from an algebraic perspective.\footnote{Note that any affine variety over the reals may be defined by polynomials of degree at most two by suitably adding extra variables; in our discussion here on normalization, we consider varieties defined without additional variables.}

In addition to satisfying these appealing properties, our notion of normalization also possesses an important computational attribute -- given a (generic) linear map, a normalizing positive-definite rank-preserver for the map can be computed using the Operator Sinkhorn iterative procedure developed in \cite{Gur:04}.  Thus, the following method offers a natural approach for uniquely associating a regularizer to an equivalence class of factorizations.

\underline{\emph{Obtaining a regularizer from a linear map}}: Given a linear map $\L : \R^{q \times q} \rightarrow \R^d$ obtained from a factorization \eqref{eq:lowrankfactor}, the unit ball of the regularizer we associate to this factorization is the image of the nuclear norm ball in $\R^{q \times q}$ under the linear map $\L \circ \sfin_\L$; here $\sfin_\L$ is the unique positive-definite rank-preserver that normalizes $\L$ (as discussed in the sequel in Corollary \ref{thm:external_normalizablelinearmaps}, such unique normalizing rank-preservers exist for generic maps $\L$).

The soundness of this approach follows from the fact that linear maps from equivalent factorizations produce the same regularizer.  We prove a result on this point in the next section (see Proposition~\ref{thm:external_rankpreserversoundness}), and we also discuss algorithmic consequences of the Operator Sinkhorn scaling procedure of \cite{Gur:04}.

\subsection{Normalizing Maps via Operator Sinkhorn Scaling} \label{sec:algorithm_normalization}
From the discussion in the preceding section, a key step in associating a unique regularizer to a collection of equivalent factorizations is to normalize a given linear map $\L : \R^{q \times q} \rightarrow \R^d$.  In this section we describe how this may be accomplished by appealing to the work of Gurvits \cite{Gur:04}.

Given a linear operator $\cpt : \S^q \rightarrow \S^q$ that leaves the positive-semidefinite cone invariant, Gurvits consider the question of the existence (and computation) of positive-definite matrices $P_1,P_2 \in \S^q_{++}$ such that the rescaled operator $\tilde{\cpt} = (P_1 \botimes P_1) \circ \cpt \circ (P_2 \botimes P_2)$ has the property that $\tilde{\cpt}(I) = \tilde{\cpt}'(I) = I$, i.e., the identity matrix is an eigenmatrix of the rescaled operator $\tilde{\cpt}$ and its adjoint \cite{Gur:04}.  This problem is an operator analog of the classical problem of transforming entrywise square nonnegative matrices to doubly stochastic matrices by diagonal congruence scaling.  This \emph{matrix scaling} problem was originally studied by Sinkhorn \cite{Sin:64}, and he developed an iterative solution technique that is known as Sinkhorn scaling.   Gurvits developed an operator analog of classical Sinkhorn scaling that proceeds by alternately performing the updates $\cpt \leftarrow (\cpt(I)^{-1/2} \botimes \cpt(I)^{-1/2} ) \circ \cpt$ and $\cpt \leftarrow \cpt \circ (\cpt'(I)^{-1/2} \botimes \cpt'(I)^{-1/2} )$; this sequence of operations is known as the \emph{Operator Sinkhorn iteration}.  The next theorem concerning the convergence of this iterative method is proved in \cite{Gur:04}.  Following the terminology in \cite{Gur:04}, a linear operator $\cpt : \S^q \rightarrow \S^q$ is \emph{rank-indecomposable} if it satisfies the inequality $\mathrm{rank}\left(\cpt(Z) \right) > \mathrm{rank}(Z)$ for all $Z \succeq 0$ with $1 \leq \mathrm{rank}(Z) < q$; this condition is an operator analog of a matrix being irreducible.

\begin{theorem} (\cite[Theorem 4.6 and 4.7]{Gur:04}) \label{thm:external_osiconverges}
Let $\cpt : \S^q \rightarrow \S^q$ be a rank-indecomposable linear operator.  There exist \emph{unique} positive-definite matrices $P_1, P_2 \in \S^q_{++}$ with $\det(P_1) = 1$ such that $\tilde{\cpt} = (P_1 \botimes P_1) \circ \cpt \circ (P_2 \botimes P_2)$ satisfies the conditions $\tilde{\cpt}(I) = \tilde{\cpt}'(I) = I$.  Moreover, the Operator Sinkhorn Iteration initialized with $\cpt$ converges to $\tilde{\cpt}$.
\end{theorem}

\noindent \textbf{Remark.} The condition $\det(P_1) = 1$ is imposed purely to avoid the ambiguity that arises from setting $P_1 \leftarrow \alpha P_1$ and $P_2 \leftarrow \tfrac{1}{\alpha} P_2$ for positive scalars $\alpha$.  Other than this degree of freedom, there are no other positive-definite matrices that satisfy the property that the rescaled operator $\tilde{\cpt}$ in this theorem as well as its adjoint both have the identity as a eigenmatrix.

\begin{algorithm}[t]
    \caption{Normalizing a linear map via the Operator Sinkhorn iteration}
    \label{alg:osi}
    \textbf{Input}: A linear map $\L: \R^{q \times q} \rightarrow \R^d$ with component functionals $\L_{i}, ~i=1,\dots,d$ \\
    \textbf{Require}: A normalized map $\L \circ \sfim$ where $\sfim: \R^{q \times q} \rightarrow \R^d$ is a rank-preserver that acts via conjugation by positive-definite matrices \\
    \textbf{Algorithm}: Repeat until convergence \\
    \textbf{1.} $R = \sum_{i=1}^d \L_{i} {\L_{i}}'$ \\
    \textbf{2.} $\L_{i} \leftarrow \sqrt{q} R^{-\tfrac{1}{2}} \L_{i}, ~ i=1,\dots,d$ \\
    \textbf{3.} $C = \sum_{i=1}^d {\L_{i}}' \L_{i}$ \\
    \textbf{4.} $\L_{i} \leftarrow \sqrt{q} \L_{i} C^{-\tfrac{1}{2}}, ~ i=1,\dots,d$
\end{algorithm}

These ideas and results are directly relevant in our context as follows.  For any linear map $\L : \R^{q \times q} \rightarrow \R^d$, we may associate an operator $\cpt_\L : \S^q \rightarrow \S^q$ defined as $\cpt_\L(Z) = \frac{1}{q}\sum_{i=1}^d \L_{i} Z {\L_{i}}'$, which has the property that it leaves the positive-semidefinite cone invariant.  Rescaling the operator $\cpt_\L$ via positive-definite matrices $P_1,P_2 \in \S^q_{++}$ to obtain $\tilde{\cpt}_\L = (P_1 \botimes P_1) \circ \cpt_\L \circ (P_2 \botimes P_2)$ corresponds to conjugating the component linear functionals $\{\L_{i}\}_{i=1}^d$ of $\L$ by $P_1$ and $P_2$.  Consequently, rescaling $\cpt_\L$ so that $\tilde{\cpt}_\L = (P_1 \botimes P_1) \circ \cpt_\L \circ (P_2 \botimes P_2)$ and its adjoint both have the identity as an eigenmatrix is equivalent to composing $\L$ by a positive-definite rank-preserver $\sfin = P_1 \botimes P_2$ so that $\L \circ \sfin$ is normalized.  Based on this correspondence Algorithm \ref{alg:osi} gives a specialization of the general Operator Sinkhorn Iteration to our setting for normalizing a linear map $\L$.\footnote{Algorithm \ref{alg:osi} requires the computation of a matrix square root at every iteration.  By virtue of the fact that the operator $\cpt_\L$ which we wish to rescale is completely positive, it is possible to normalize $\L$ using only rational matrix operations via a modified scheme known as the Rational Operator Sinkhorn iteration \cite{Gur:04}.}  We also have the following corollary to Theorem \ref{thm:external_osiconverges}:

\begin{corollary} \label{thm:external_normalizablelinearmaps}
Let $\L : \R^{q \times q} \rightarrow \R^d$ be a linear map, and suppose $\mathrm{rank}(\sum_{i=1}^d \L_{i} Z {\L_{i}}' )  > \mathrm{rank}(Z)$ for all $Z \succeq 0$ with $1 \leq \mathrm{rank}(Z) < q$ (i.e., the operator $\cpt_\L(Z) = \frac{1}{q} \sum_{i=1}^d \L_{i} Z {\L_{i}}'$ is rank-indecomposable).  There exists a \emph{unique} positive-definite rank-preserver $\sfin_\L : \R^{q \times q} \rightarrow \R^{q \times q}$ such that $\L \circ \sfin_\L$ is normalized.  Moreover, Algorithm \ref{alg:osi} initialized with $\L$ converges to $\L \circ \sfin_\L$.
\end{corollary}

\begin{proof}
The existence of a positive-definite rank preserver $\sfin_\L$ as well as the convergence of Algorithm \ref{alg:osi} follow directly from Theorem \ref{thm:external_osiconverges}.  We need to prove that $\sfin_\L$ is unique.  Let $\tilde{N}_\L: \R^{q \times q} \rightarrow \R^{q \times q}$ be any positive-definite rank-preserver such that $\L \circ \tilde{\sfin}_\L$ is normalized.  By Theorem \ref{thm:external_rankpreserver}, there exists positive-definite matrices $P_1,P_2, \tilde{P}_1, \tilde{P}_2$ such that $\sfin_\L = P_1 \otimes P_2$ and $\tilde{\sfin}_\L = \tilde{P}_1 \otimes \tilde{P}_2$.  Without loss of generality, we may assume that $\mathrm{det}(P_1) = \mathrm{det}(\tilde{P}_1)=1$.  By Theorem \ref{thm:external_osiconverges} we have $P_1 = \tilde{P}_1$ and $P_2 = \tilde{P}_2$, and consequently that $\sfin_\L = \tilde{N}_\L$. \qed
\end{proof}

Generic linear maps $\L : \R^{q \times q} \rightarrow \R^d$ (for $d \geq 2$) satisfy the condition $\mathrm{rank}(\sum_{i=1}^d \L_{i} Z {\L_{i}}' ) > \mathrm{rank}(Z)$ for all $Z \succeq 0$ with $1 \leq \mathrm{rank}(Z) < q$.  Therefore, this assumption in Corollary \ref{thm:external_normalizablelinearmaps} is not particularly restrictive.  A consequence of the uniqueness of the positive-definite rank-preserver $\sfin_\L$ in Corollary \ref{thm:external_normalizablelinearmaps} is that our normalization scheme associates a unique regularizer to every collection of equivalent factorizations:

\begin{proposition}\label{thm:external_rankpreserversoundness}
	Let $\L : \R^{q \times q} \rightarrow \R^d$ be a linear map, and suppose $\mathrm{rank}(\sum_{i=1}^d \L_{i} Z {\L_{i}}' ) > \mathrm{rank}(Z)$ for all $Z \succeq 0$ with $1 \leq \mathrm{rank}(Z) < q$.  Let $\sfim: \R^{q \times q} \rightarrow \R^{q \times q}$ be any rank-preserver.  Suppose $\sfin_{\L}$ and $\sfin_{\L \circ \sfim}$ are positive-definite rank-preservers such that $\L \circ \sfin_{\L}$ and $\L \circ \sfim \circ \sfin_{\L \circ \sfim}$ are normalized.  Then the image of the nuclear norm ball under $\L \circ \sfin_\L$ is the same as it is under $\L \circ \sfim \circ \sfin_{\L \circ \sfim}$.
\end{proposition}

\noindent \textbf{Remark.} Note that if the linear map $\L$ satisfies the property that $\mathrm{rank}(\sum_{i=1}^d \L_{i} Z {\L_{i}}' ) > \mathrm{rank}(Z)$ for all $Z \succeq 0$ with $1 \leq \mathrm{rank}(Z) < q$, then so does the linear map $\L \circ \sfim$ for any rank-preserver $\sfim$.

\begin{proof} As $\sfim^{-1} \circ \sfin_{\L}$ is a rank-preserver, we can apply Corollary~\ref{thm:external_rankpreserverpolardecomp} to obtain the decomposition $\sfim^{-1} \circ \sfin_{\L} = \bar{\sfim}^{\mathrm{or}} \circ \bar{\sfim}^{\mathrm{pd}}$, where $\bar{\sfim}^{\mathrm{or}}$ is an orthogonal rank-preserver and $\bar{\sfim}^{\mathrm{pd}}$ is a positive-definite rank-preserver.
	
We claim that $\sfin_{\L \circ \sfim} = \sfim^{-1} \circ \sfin_{\L} \circ \bar{\sfim}^{\mathrm{or}'} $.  First, we have $\sfim^{-1} \circ \sfin_{\L} \circ \bar{\sfim}^{\mathrm{or}'} = \bar{\sfim}^{\mathrm{or}} \circ \bar{\sfim}^{\mathrm{pd}} \circ \bar{\sfim}^{\mathrm{or}'}$, which implies that this operator is positive-definite.  Next, we note that a linear map that is obtained by right multiplication of a normalized linear map with an orthogonal rank-preserver is also normalized, and hence the linear map $\L \circ \sfim \circ \sfim^{-1} \circ \sfin_{\L} \circ \bar{\sfim}^{\mathrm{or}'} = \L \circ \sfin_{\L} \circ \bar{\sfim}^{\mathrm{or}'}$ is normalized.  By applying Corollary \ref{thm:external_normalizablelinearmaps}, we conclude that $\sfin_{\L \circ \sfim} = \sfim^{-1} \circ \sfin_{\L} \circ \bar{\sfim}^{\mathrm{or}'} $.
	
	Consequently, we have $\L \circ \sfim \circ \sfin_{\L \circ \sfim} = \L \circ \sfin_{\L} \circ \bar{\sfim}^{\mathrm{or}'}$.  As the nuclear norm ball is invariant under the action of the orthogonal rank-preserver $\bar{\sfim}^{\mathrm{or}'}$, it follows that the image of the nuclear norm ball under the map $\L \circ \sfin_\L$ is the same as it is under the map $\L \circ \sfim \circ \sfin_{\L \circ \sfim}$. \qed
\end{proof}

The polynomial-time complexity of the (general) Operator Sinkhorn iterative procedure -- in terms of the number of iterations required to obtain a desired accuracy to the fixed-point -- has recently been established in \cite{GGOW:16}.  In summary, this approach provides a computationally tractable method to normalize linear maps, and consequently to associate a unique regularizer to a collection of equivalent factorizations.

\subsection{An Alternating Update Algorithm for Matrix Factorization} \label{sec:algorithm_am}

Given the resolution of the identifiability issues in the preceding two sections, we are now in a position to describe an algorithmic approach for computing a factorization \eqref{eq:lowrankfactor} of a data matrix $Y = [\by^{(1)} | \cdots | \by^{(n)}] \in \R^{d \times n}$ to obtain a semidefinite regularizer that promotes the type of structure contained in $Y$.  Specifically, given a target dimension $q$, our objective is to obtain a normalized linear map $\L : \R^{q \times q} \rightarrow \R^d$ and a collection $\ensafull$ of low-rank matrices such that $\sum_{i=1}^n \|\by^{(j)} - \L(X^{(j)})\|_{\ell_2}^2$ is minimized.  Our procedure is an alternating update technique that sequentially updates the low-rank $X^{(j)}$'s followed by an update of $\L$.  We assume that our algorithm is provided with a data matrix $Y \in \R^{d \times n}$, a target dimension $q$, and an initial guess for the normalized map $\L$.  Our method is summarized in Algorithm \ref{alg:amsdpreg}.

\subsubsection{Updating the low-rank matrices $\ensafull$} \label{sec:algorithm_am_lrmrecovery}

In this stage a normalized linear map $\L : \R^{q \times q} \rightarrow \R^d$ is fixed, and the objective is to find low-rank matrices $\ensafull$ such that $\by^{(j)} \approx \L(X^{(j)})$ for each $j=1,\dots,n$. 
Without the requirement that the $X^{(j)}$'s be low-rank, such linear inverse problems are ill-posed in our context as $q^2$ is typically taken to be larger than $d$.  With the low-rank restriction, this problem is well-posed and it is known as the \emph{affine rank minimization problem}.  This problem is NP-hard in general \cite{Nat:93}.  However, due to its prevalence in a range of application domains \cite{Faz:02,RFP:10}, significant efforts have been devoted towards the development of tractable heuristics that are useful in practice and that succeed on certain families of problem instances.  We describe next two popular heuristics for this problem.

The first approach -- originally proposed by Fazel in her thesis \cite{Faz:02} and subsequently analyzed in \cite{CanRec:09,RFP:10} -- is based on a convex relaxation in which the rank constraint is replaced by the nuclear norm penalty, which leads to the following convex program:
\begin{equation} \label{eq:nuclearnormlasso}
\begin{aligned}
\hat{X} = \underset{X \in \R^{q \times q}}{\argmin} ~~~ \tfrac{1}{2} \|\by - \L(X) \|_{\ell_2}^2 + \lambda \|X\|_\star.
\end{aligned}
\end{equation}
Here $\by \in \R^d$ and $\L : \R^{q \times q} \rightarrow \R^d$ are the problem data specifying the affine space near which we seek a low-rank solution, and the parameter $\lambda > 0$ provides a tradeoff between fidelity to the data (i.e., fit to the specified affine space) and rank of the solution $\hat{X}$.  This problem is a semidefinite program and it can solved to a desired precision in polynomial-time using standard software \cite{NesNem:94,sdpt3:1}.

\begin{algorithm}[t]
    \caption{Obtaining a low-rank matrix near an affine space via Singular Value Projection}
    \label{alg:svp}
    \textbf{Input}: A linear map $\L: \R^{q \times q} \rightarrow \R^d$, a point $\by \in \R^d$, a target rank $r$, an initial guess $X \in \R^{q \times q}$, and a damping parameter $\nu \in (0,1]$ \\
    \textbf{Require}: A matrix $\hat{X}$ of rank at most $r$ such that $\|\by - \L(\hat{X}) \|_{\ell_2}$ is minimized, i.e., solve \eqref{eq:varietyconstrainedopt} \\
    \textbf{Initialization} $X=0$ \\
    \textbf{Algorithm}: Repeat until convergence \\
    \textbf{1.} $X \leftarrow X + \nu \L' (\by - \L(X))$ (i.e., take a gradient step with respect to the objective of \eqref{eq:varietyconstrainedopt}) \\
    \textbf{2.} Compute top-$r$ singular vectors and singular values of $X$: $U_r, V_r \in \R^{q \times r}, ~ \Sigma_r \in \R^{r \times r}$ \\
    \textbf{3.} $X \leftarrow U_r \Sigma_r V_r'$
\end{algorithm}

Another popular method for the affine rank minimization problem is based on directly attempting to solve the following non-convex optimization problem via alternating projection for a specified rank $r < q$:
\begin{equation} \label{eq:varietyconstrainedopt}
\begin{aligned}
\hat{X} = \underset{X \in \R^{q \times q}}{\argmin} & ~~~ \|\by - \L(X) \|_{\ell_2}^2 \\ \text{s.t.} & ~~~ \mathrm{rank}(X) \leq r.
\end{aligned}
\end{equation}
This problem is intractable to solve globally in general, but the heuristic described in Algorithm \ref{alg:svp} provides an approach that provably succeeds under certain conditions \cite{GM:11,JMD:10}.  The utility of this method in comparison to the convex program \eqref{eq:nuclearnormlasso} is that applying the procedure described in Algorithm \ref{alg:svp} is much more tractable in large-scale settings in comparison to solving \eqref{eq:nuclearnormlasso}.

The analyses in \cite{FCRP:08,GM:11,JMD:10,RFP:10} rely on the map $\L$ satisfying the following type of restricted isometry condition introduced in \cite{RFP:10}:
\begin{definition} Consider a linear map $\L : \R^{q \times q} \rightarrow \R^d$.  For each $k = 1,\dots,q$ the \emph{restricted isometry constant of order $k$} is defined as the smallest $\delta_k(\L)$ such that:
\begin{equation*}
1-\delta_k(\L) \leq \frac{\|\L(X)\|_{\ell_2}^2}{\|X\|_{\fro}^2} \leq 1 + \delta_k(\L)
\end{equation*}
for all matrices $X \in \R^{q \times q}$ with rank less than or equal to $k$.
\end{definition}
If a linear map $\L$ has a small restricted isometry constant for some order $k$, then the affine rank minimization problem is, in some sense, well-posed when restricted to matrices of rank less than or equal to $k$.  The results in \cite{FCRP:08,GM:11,JMD:10,RFP:10} go much further by demonstrating that if $\by = \L(X^\star) + \boldsymbol \epsilon$ for $\boldsymbol \epsilon \in \R^d$ and with $\mathrm{rank}(X^\star) \leq r$, and if the map $\L$ satisfies a bound on the restricted isometry constant $\delta_{4r}(\L)$, then both the convex program \eqref{eq:nuclearnormlasso} as well as the procedure in Algorithm \ref{alg:svp} applied to solve \eqref{eq:varietyconstrainedopt} provide solutions $\hat{X}$ such that $\|\hat{X} - X^\star\|_{\fro} \lesssim C \|\boldsymbol \epsilon\|_{\ell_2}$.  Due to the qualitative similarity in the performance guarantees for these approaches, either of them is appropriate as a subroutine for updating the $X^{(j)}$'s in our alternating update method for computing a factorization of a given data matrix $Y \in \R^{d \times n}$.  Algorithm \ref{alg:amsdpreg} is therefore stated in a general manner to retain this flexibility.  In our main theoretical result in Section \ref{sec:analysis_mainresult}, we assume that the $X^{(j)}$'s are updated by solving \eqref{eq:varietyconstrainedopt} using the heuristic outlined in Algorithm \ref{alg:svp}; our analysis could equivalently be carried out by assuming that the $X^{(j)}$'s are updated by solving \eqref{eq:nuclearnormlasso}.

\subsubsection{Updating the linear map $\L$}
In this stage the low-rank matrices $\ensafull$ are fixed and the goal is to obtain a normalized linear map $\L$ such that $\sum_{i=1}^n \|\by^{(j)} - \L(X^{(j)})\|_{\ell_2}^2$ is minimized.  Our procedure for this update consists of two steps.  First we solve the following least-squares problem:
\begin{equation} \label{eq:dictionaryleastsquaresupdate}
\tilde{\L} = \underset{\substack{\bar{\L} : \R^{q \times q} \rightarrow \R^d \\ \bar{\L}~\text{is a linear map}}}{\argmin} ~~~ \sum_{i=1}^n \|\by^{(j)} - \bar{\L}(X^{(j)})\|_{\ell_2}^2
\end{equation}
This problem can be solved, for example, via a pseudoinverse computation.  Next, we apply the procedure described in Algorithm \ref{alg:osi} to the updated $\tilde{\L}$ obtained from \eqref{eq:dictionaryleastsquaresupdate} in order to normalize it.

\begin{algorithm}[t]
    \caption{Computing a factorization via alternating updates}
    \label{alg:amsdpreg}
    \textbf{Input}: A data matrix $Y = [\by^{(1)} | \cdots | \by^{(n)}] \in \R^{d \times n}$, a target dimension $q$, an initial guess for a normalized linear map $\L: \R^{q \times q} \rightarrow \R^d$, a target rank $r < q$ \\
    \textbf{Require}: A normalized linear map $\hat{\L} : \R^{q \times q} \rightarrow \R^d$ and a collection of matrices $\{\hat{X}^{(j)}\}_{j=1}^n$ with rank at most $r$ such that $\sum_{i=1}^n \|\by^{(j)} - \hat{\L}(\hat{X}^{(j)})\|_{\ell_2}^2$ is minimized \\
    \textbf{Algorithm}: Repeat until convergence \\
    \textbf{1.}[Update $X^{(j)}$'s; $\L$ fixed] Obtain matrices $\ensafull$ of rank at most $r$ such that $\sum_{i=1}^n \|\by^{(j)} - \L(X^{(j)})\|_{\ell_2}^2$ is minimized.  This can be accomplished either via Algorithm \ref{alg:svp} or by solving \eqref{eq:nuclearnormlasso} for a suitable choice of $\lambda$. \\
    \textbf{2.}[Update $\L$; $X^{(j)}$'s fixed] $\tilde{\L} \leftarrow \underset{\substack{\bar{\L} : \R^{q \times q} \rightarrow \R^d \\ \bar{\L}~\text{is a linear map}}}{\argmin} ~~~ \sum_{i=1}^n \|\by^{(j)} - \bar{\L}(X^{(j)})\|_{\ell_2}^2$ \\
    \textbf{3.}[Normalize $\L$] Normalize updated linear map from previous step using Algorithm \ref{alg:osi}.
\end{algorithm}

\subsection{Comparison with Dictionary Learning} \label{sec:algorithm_cdl}
As described in Section \ref{sec:intro_semidefleadup}, the dictionary learning literature considers the following factorization problem: given a collection of data points $\{\by^{(j)}\}_{j=1}^n \subset \R^d$ and a target dimension $p$, find a linear map $L : \R^p \rightarrow \R^d$ and a collection of sparse vectors $\{\bx^{(j)}\}_{j=1}^n \subset \R^p$ such that $\by^{(j)} = L \bx^{(j)}$ for each $j$.  As with \eqref{eq:lowrankfactor}, the linear map $L$ does not lead to a unique polyhedral regularizer.  Specifically, for any linear sparsity-preserver $M : \R^p \rightarrow \R^p$, there is an equivalent factorization in which the linear map is $L M$.  In parallel to Corollary \ref{thm:external_rankpreserverpolardecomp}, one can check that $M$ is a sparsity-preserver if and only if $M$ is a composition of a positive-definite diagonal matrix and a signed permutation matrix.  Since the $\ell_1$ ball is invariant under the action of a signed permutation, the main source of difficulty in obtaining a unique regularizer from a factorization is due to sparsity-preservers that are positive-definite diagonal matrices.  A common convention in dictionary learning that addresses this identifiability issue is to require that each of the columns of $L$ has unit Euclidean norm; for a generic linear map $L$, there is a unique positive-definite diagonal matrix $D$ such that $LD$ consists of unit-norm columns.  Adopting a similar reasoning as in Section \ref{sec:algorithm_normalization}, one can check that this normalization resolves the issue of associating a unique regularizer to an equivalence of factorizations.

The most popular approach for computing a factorization in dictionary learning is based on alternately updating the map $L$ and the sparse vectors $\{\bx^{(j)}\}_{j=1}^n$.  For a fixed linear map $L$, updating the $\bx^{(j)}$'s entails the solution of a sparse linear inverse problem for each $j$.  That is, for each $j$ we seek a sparse vector $\bx^{(j)}$ in the affine space $\by^{(j)} = L \bx$.  Although this problem in NP-hard in general, there is a significant literature on tractable heuristics that succeed under suitable conditions \cite{CRT:06,CT:06,CDS:98,Don:06,Don:06b,DonHuo:01}; indeed, this work predates and served as a foundation for the literature on the affine rank minimization problem.  Prominent examples include the lasso \cite{Tib:94}, which is a convex relaxation approach akin to \eqref{eq:nuclearnormlasso}, and iterative hard thresholding \cite{BD:09}, which is analogous to Algorithm \ref{alg:svp}.  For a fixed collection $\{\bx^{(j)}\}_{j=1}^n$, the linear map $L$ is then updated by solving a least-squares problem followed by a rescaling of the columns so that they have unit Euclidean norm.

We note that each step in this procedure has a direct parallel to a corresponding step of Algorithm \ref{alg:amsdpreg}.  In summary, our proposed approach for obtaining a semidefinite regularizer via matrix factorization is a generalization of previous methods in the dictionary learning literature for obtaining a polyhedral regularizer. 

%% file: sc_analysis.tex
\section{Convergence Analysis of Our Algorithm} \label{sec:analysis}

This section describes the main theoretical result on the local convergence of our algorithm.  We begin by discussing the setup and an outline of our analysis in Sections \ref{sec:analysis_setup} and \ref{sec:analysis_proofhighlevel} respectively.  The statement of our main theorem with deterministic conditions is given in Section \ref{sec:analysis_mainresult}, and we describe natural random ensembles that satisfy these deterministic conditions with high probability in Section \ref{sec:analysis_randensemble}.  The proof of our theorem is discussed in Section \ref{sec:analysis_maintheoremproof}.

\subsection{Theoretical Setup} \label{sec:analysis_setup}

The setup underlying our main theorem is as follows.  We assume that we are given a collection of data points $\{{\by^{(j)}}^\star\}_{j=1}^n \subset \R^d$ with each ${\by^{(j)}}^\star = \L^\star({X^{(j)}}^\star)$, where $\L^\star : \R^{q \times q} \rightarrow \R^d$ is a linear map and $\ensb := \ensbfull \subset \R^{q \times q}$ is a collection of low-rank matrices.  Without loss of generality, we may take $\L^\star$ to be normalized and surjective.  Our objective is to obtain a linear map $\hat{\L} : \R^{q \times q} \rightarrow \R^d$ with the property that the image of the nuclear norm ball in $\R^{q \times q}$ under $\L^\star$ is the same as it is under $\hat{\L}$.  To this end, we seek a linear map $\hat{\L}$ that can be expressed as the composition of $\L^\star$ with an orthogonal rank-preserver (recall that the nuclear norm ball is invariant under the action of an orthogonal rank-preserver).

As this goal is distinct from the more restrictive requirement that $\hat{\L}$ must equal $\L^\star$, we need an appropriate measure of the ``distance'' of a linear map to $\L^\star$.  A convenient approach to addressing this issue is to express a linear map $\L : \R^{q \times q} \rightarrow \R^d$ in terms of $\L^\star$ as follows, given any linear rank-preserver $\sfim : \R^{q \times q} \rightarrow \R^{q \times q}$:
\begin{equation} \label{eq:dicterr_rp_factorization}
\L = \L^\star \circ (\sfi + \sfe) \circ \sfim,
\end{equation}
Here $\sfi \in \aut(\R^{q \times q})$ is the identity map and the error term $\sfe = {\L^{\star}}^+ \circ (\L \circ \sfim^{-1} - \L^\star) \in \aut(\R^{q \times q})$; the assumption that $\L^\star$ is surjective is key as ${\L^{\star}}^+$ is the right-inverse of $\L^\star$.  By varying the rank-preserver $\sfim$ in \eqref{eq:dicterr_rp_factorization} the error term $\sfe$ changes. If there exists an \emph{orthogonal} rank-preserver $\sfim$ such that the corresponding error $\sfe$ is small, then in some sense the image of the nuclear norm ball under $\L$ is close to the image under $\L^\star$.  This observation suggests that the closeness between $\L$ and $\L^\star$ may be measured as the smallest error $\sfe$ that one can obtain by varying $\sfim$ over the set of orthogonal rank-preservers.  The following result suggests that one can in fact vary $\sfim$ over \emph{all} rank-preservers, provided we have the additional condition that $\L$ is also normalized.  The additional flexibility provided by varying $\sfim$ over all rank-preservers is well-suited to characterizing the effects of normalization via Operator Sinkhorn scaling in our analysis, as described in the next section.

\begin{proposition} \label{thm:normalizednearothogonal}
Suppose $\L, \L^\star : \R^{q \times q} \rightarrow \R^d$ are normalized linear maps such that $(i)$ $\L^\star$ satisfies the restricted isometry condition $\delta_1(\L^\star) \leq 1/10$, and $(ii)$ $\L = \L^\star \circ (\sfi + \sfe) \circ \sfim$ for a linear rank-preserver $\sfim$ with $\|\sfe\|_{\eu} \leq 1 / (150 \sqrt{q} \|\L^\star\|_{\spec})$.  Then there exists an orthogonal rank-preserver $\sfim^{\mathrm{or}}$ such that $\|\sfim^{\mathrm{or}} - \sfim\|_{\spec} \leq 300 \sqrt{q} \|\L^\star\|_{\spec} \|\sfe\|_{\eu}$.
\end{proposition}

In words, if both $\L$ and $\L^\star$ are normalized and if there exists a rank-preserver $\sfim$ such that $\|\sfe\|_{\eu}$ is small in \eqref{eq:dicterr_rp_factorization}, then $\sfim$ is close to an orthogonal rank-preserver\footnote{The restricted isometry condition in Proposition \ref{thm:normalizednearothogonal} is a mild one; we require a stronger restricted isometry condition on $\L^\star$ in Theorem \ref{thm:localconvergence}.}; in turn, this implies that the image of the nuclear norm ball under $\L^\star$ is close to the image of the nuclear norm ball under $\L$.  These observations motivate the following definition as a measure of the distance between normalized linear maps $\L^\star, \L : \R^{q \times q} \rightarrow \R^d$ for surjective $\L^\star$:
\begin{align} \label{eq:defn_distancebetweendicts}
\xi_{\L^\star}(\L) & := \inf\{\|\sfe\|_{\eu} ~|~ \exists \sfe  \in \aut(\R^{q \times q}) ~\text{and a rank-preserver}~ \sfim \in \aut(\R^{q \times q}) \nonumber \\
& \quad \quad \quad \quad ~\text{s.t.}~ \L = \L^\star \circ (\sfi + \sfe) \circ \sfim \}.
\end{align}
In Section \ref{sec:analysis_mainresult}, our main result gives conditions under which the sequence of normalized linear maps obtained from Algorithm \ref{alg:amsdpreg} converges to $\L^\star$ in terms of the distance measure $\xi$.

\subsection{An Approach for Proving a Local Convergence Result} \label{sec:analysis_proofhighlevel}

We describe a high-level approach for proving a local convergence result, which motivates the definition of the key parameters that govern the performance of our algorithm.  Our proof strategy is to demonstrate that under appropriate conditions the sequence of normalized iterates $\L^{(t)}$ obtained from Algorithm \ref{alg:amsdpreg} satisfies $\xi_{\L^\star}(\L^{(t+1)}) \leq \gamma \xi_{\L^\star}(\L^{(t)})$ for a suitable $\gamma < 1$.  To bound $\xi_{\L^\star}(\L^{(t+1)})$ with respect to $\xi_{\L^\star}(\L^{(t)})$, we consider each of the three steps in Algorithm \ref{alg:amsdpreg}.  Fixing notation before we proceed, let $\L^{(t)} = {\L^{\star}} \circ (\sfi + \sfe^{(t)}) \circ \sfim^{(t)}$ for some linear rank-preserver $\sfim^{(t)}$ and for a corresponding error term $\sfe^{(t)}$.  Our objective is to show that there exists a linear rank-preserver $\sfim^{(t+1)}$ and corresponding error term $\sfe^{(t+1)}$ with $\L^{(t+1)} = {\L^{\star}} \circ (\sfi + \sfe^{(t+1)}) \circ \sfim^{(t+1)}$, so that $\|\sfe^{(t+1)}\|_{\eu}$ is suitably bounded above in terms of $\|\sfe^{(t)}\|_{\eu}$.  By taking limits we obtain the desired result in terms of $\xi_{\L^\star}(\L^{(t)})$ and $\xi_{\L^\star}(\L^{(t+1)})$.

The first step of Algorithm \ref{alg:amsdpreg} involves the solution of the following optimization problem for each $j = 1,\dots,n$:
\begin{equation*}
\hat{X}^{(j)} = ~ \underset{X \in \R^{q \times q}}{\argmin} ~ \left\|{\by^{(j)}}^\star - \L^{(t)} (X) \right\|_{\ell_2}^2 ~ \text{s.t.} ~ \mathrm{rank}(X) \leq r.
\end{equation*}
As $\L^{(t)} = \L^\star \circ (\sfi + \sfe^{(t)}) \circ \sfim^{(t)}$ and as ${\by^{(j)}}^\star = \L^\star ({X^{(j)}}^\star)$, the preceding problem can be reformulated in the following manner:
\begin{equation*}
\begin{aligned}
\sfim^{(t)}(\hat{X}^{(j)}) = ~ \underset{\tilde{X} \in \R^{q \times q}}{\argmin} ~ & \left\|\L^\star \circ (\sfi + \sfe^{(t)})({X^{(j)}}^\star) - \L^\star \circ \sfe^{(t)} ({X^{(j)}}^\star) - \L^\star \circ (\sfi + \sfe^{(t)})(\tilde{X}) \right\|_{\ell_2}^2 \\ \text{s.t.} ~ & \mathrm{rank}(\tilde{X}) \leq r.
\end{aligned}
\end{equation*}
If $\L^{\star} \circ (\sfi + \sfe^{(t)})$ satisfies a suitable restricted isometry condition and if $\|\L^\star \circ \sfe^{(t)} ({X^{(j)}}^\star)\|_{\ell_2}$ is small, then the results in \cite{GM:11,JMD:10} (as described in Section \ref{sec:algorithm_am_lrmrecovery}) imply that
$\sfim^{(t)}(\hat{X}^{(j)}) \approx {X^{(j)}}^\star$.  In other words, if $\|\sfe^{(t)}\|_{\eu}$ is small and if $\L^\star$ satisfies a restricted isometry condition, then $\sfim^{(t)}(\hat{X}^{(j)}) \approx {X^{(j)}}^\star$; the following result states matters formally:

\begin{proposition}\label{thm:varietyconstrainederr}
Let $\L^\star : \R^{q \times q} \rightarrow \R^d$ be a linear map such that $(i)$ $\L^\star$ is normalized, and $(ii)$ $\L^\star$ satisfies the restricted isometry condition $\delta_{4r}(\L^{\star}) \leq \frac{1}{20}$.  Suppose $\L = \L^\star \circ (\sfi + \sfe) \circ \sfim$ such that $(i)$ $\sfim$ is a linear rank-preserver, and $(ii)$ $\|\sfe\|_{\eu} \leq \min\{1 / (50\sqrt{q}), 1/(120r^2\|\L^\star\|_{\spec}) \}$. Finally, suppose $\by = \L^\star(X^\star)$, where $X^\star \in \R^{q \times q}$ is a rank-$r$ matrix such that $\sigma_r(X^\star) \geq \sigma_1(X^\star) / 2$, and that $\hat{X}$ is the optimal solution to
\begin{equation} \label{eq:varietyoptimization}
\hat{X} = ~ \underset{X \in \R^{q \times q}}{\argmin} ~ \left\|\by - \L(X) \right\|_{\ell_2}^2 ~ \mathrm{s.t.} ~ \mathrm{rank}(X) \leq r.
\end{equation}
Then
\begin{equation*}
\sfim(\hat{X}) = X^\star -\left[\left({\L^{\star\prime}_{\tstar}} \L^\star_{\tstar}\right)^{-1} \right]_{\R^{q \times q}}\circ {\L^\star}' \L^\star \circ \sfe \left(X^\star\right) + \err,
\end{equation*}
where $\|\err\|_{\fro} \leq 800 r^{5/2} \|\L^\star\|_{\spec}^2 \|X^\star\|_{\spec} \| \sfe \|_{\eu}^2$.
\end{proposition}

In this proposition, the conclusion is well-defined as the linear map ${\L^{\star\prime}_{\tstar}} \L^\star_{\tstar} : \tstar \rightarrow \tstar$ is invertible due to the restricted isometry condition on $\L^\star$ (see Lemma \ref{thm:weakincoherencebound}).  The proof appears in Appendix \ref{apx:varietyanalysis}, and it relies primarily on the first-order optimality conditions of the problem \eqref{eq:varietyconstrainedopt}.  To ensure that the conditions required by this proposition hold, we assume in our main theorem in Section \ref{sec:analysis_mainresult} that $\L^\star$ satisfies the restricted isometry property for rank-$r$ matrices and that the initial guess $\L^{(0)}$ that is supplied to Algorithm \ref{alg:amsdpreg} is such that $\xi_{\L^\star}(\L^{(0)})$ is small (with a sufficiently good initial guess and by an inductive hypothesis, we have that there exists an error term $\sfe^{(t)}$ at iteration $t$ such that $\|\sfe^{(t)}\|_{\eu}$ is small).

The second step of Algorithm \ref{alg:amsdpreg} entails the solution of a least-squares problem.  To describe the implications of this step in detail, we consider the linear maps $\xx^\star : \bz \mapsto \sum_{j=1}^n {X^{(j)}}^\star \bz_j$ and $\hat{\xx} :   \bz \mapsto \sum_{j=1}^n \hat{X}^{(j)} \bz_j$ from $\R^n$ to $\R^{q \times q}$.  With this notation, the second step of Algorithm \ref{alg:amsdpreg} results in the linear map $\L^{(t)}$ being updated as follows:
\begin{equation} \label{eq:intermediatedictestimate}
\tilde{\L}^{(t+1)} = \L^\star \circ \xx^\star \circ \hat{\xx}^+.
\end{equation}
In order for the normalized version of $\tilde{\L}^{(t+1)}$ to be close to $\L^\star$ (in terms of the distance measure $\xi$), we require a deeper understanding of the structure of $\xx^\star \circ \hat{\xx}^+$, which is the focus of the next proposition.  This result relies on the set $\ensb$ being suitably isotropic, as characterized by the quantities $\covsup(\ensb)$ and $\coveig(\ensb)$.

\begin{proposition} \label{thm:structuralresult}
Let $\{A^{(j)}\}_{j=1}^n \subset \R^{q \times q}$ and $\{B^{(j)}\}_{j=1}^n \subset \R^{q \times q}$ be two collections of matrices, and let $\aa : \bz \mapsto \sum_{j=1}^n A^{(j)} \bz_j$ and $\bb : \bz \mapsto \sum_{j=1}^n B^{(j)} \bz_j$ be linear maps from $\R^n$ to $\R^{q \times q}$ associated to these ensembles.  Let $\sfq : \R^{q \times q} \rightarrow \R^{q \times q}$ be any invertible linear operator and denote $\omega = \max_j \left\|\sfq(B^{(j)})  - A^{(j)} \right\|_{\fro}$.  If $\omega \leq \frac{\sqrt{\coveig\left(\{A^{(j)}\}_{j=1}^n\right)}}{20}$ and if $\frac{\covsup\left(\{A^{(j)}\}_{j=1}^n\right)} {\coveig\left(\{A^{(j)}\}_{j=1}^n\right)} \leq \frac{1}{6}$, then
\begin{equation} \label{eq:ensembleofdifferences}
\aa \circ \bb^{+} = \left(\sfi - \frac{1}{n\coveig\left(\{A^{(j)}\}_{j=1}^n\right)} \sum_{j=1}^{n} \left(\sfq(B^{(j)}) - A^{(j)}\right) \boxtimes A^{(j)} + \sff \right) \circ \sfq,
\end{equation}
where $\|\sff\|_{\eu} \leq 20 q \frac{\omega^2}{\coveig\left(\{A^{(j)}\}_{j=1}^n\right)} + 2 q  \frac{\omega \covsup\left(\{A^{(j)}\}_{j=1}^n\right)} {\coveig\left(\{A^{(j)}\}_{j=1}^n\right)^{3/2}} $.
\end{proposition}

The proof of this proposition appears in Appendix \ref{apx:pseudoinverseexpand}, and it consists of two key elements.  First, as $\omega$ is bounded, the operator $\aa \circ \bb^{+}$ may be approximated as $\aa \circ {\aa}^+ \circ \sfq$.  Second, as the set $\{A^{(j)}\}_{j=1}^n$ is near-isotropic based on the assumptions involving $\covsup(\{A^{(j)}\}_{j=1}^n)$ and $\coveig(\{A^{(j)}\}_{j=1}^n)$, one can show that $\aa \circ \aa^+$ can be expanded suitably around the identity map $\sfi$.  In the context of our analysis, we apply the conclusions of Proposition \ref{thm:structuralresult} with the choice of $A^{(j)}={X^{(j)}}^{\star}$, $B^{(j)}=\hat{X}^{(j)}$, and $\sfq = \sfim^{(t)}$.

The final step of our analysis is to consider the effect of normalization on the map $\tilde{\L}^{(t)}$ in \eqref{eq:intermediatedictestimate}.  Denoting the positive-definite rank-preserver that normalizes $\tilde{\L}^{(t+1)}$ by $\sfin_{\tilde{\L}^{(t+1)}}$, we have from Propositions \ref{thm:varietyconstrainederr} and \ref{thm:structuralresult} that the normalized map $\L^{(t+1)}$ obtained after the application of the Operator Sinkhorn iterative procedure to $\tilde{\L}^{(t+1)}$ can be expressed as:
\begin{equation*}
\L^{(t+1)} = \L^\star \circ \left(\sfi - \frac{1}{n\coveig(\ensb)} \sum_{j=1}^{n} \left(\sfim^{(t)}(\hat{X}^{(j)}) - {X^{(j)}}^\star\right) \boxtimes {X^{(j)}}^\star + \sff \right) \circ \sfim^{(t)} \circ \sfin_{\tilde{\L}^{(t+1)}},
\end{equation*}
where $\sff \in \aut(\R^{q\times q})$ is suitably bounded.  As $\sfim^{(t)}$ and $\sfin_{\tilde{L}^{(t+1)}}$ are both rank-preservers, we need to prove that the expression within parentheses $\sfi  -  \frac{1}{n\coveig(\ensb)} \allowbreak \sum_{j=1}^{n} (\sfim^{(t)}(\hat{X}^{(j)}) - {X^{(j)}}^\star) \boxtimes {X^{(j)}}^\star + \sff$ is well-approximated as a rank-preserver so that $\xi_{\L^\star}(\L^{(t+1)})$ is suitably controlled.  To make progress on this front, we note that $\sfi = I \botimes I$ is a rank-preserver. Therefore, if $- \frac{1}{n\coveig(\ensb)} \sum_{j=1}^{n} (\sfim^{(t)}(\hat{X}^{(j)}) - {X^{(j)}}^\star) \boxtimes {X^{(j)}}^\star + \sff$ is small, a natural approach to characterizing how close $\sfi  -  \frac{1}{n\coveig(\ensb)}  \sum_{j=1}^{n} (\sfim^{(t)}(\hat{X}^{(j)}) - {X^{(j)}}^\star) \boxtimes {X^{(j)}}^\star + \sff$ is to a rank-preserver is to express this quantity in terms of the following \emph{tangent space} at $\sfi$ with respect to the set of rank-preservers acting on the space of $q \times q$ matrices:
\begin{equation} \label{eq:tangentspaceati}
\W = \mathrm{span}\{I \botimes W_1 + W_2 \botimes I ~|~ W_1, W_2 \in \R^{q \times q}\}
\end{equation}
The next result gives such an expression.

\begin{proposition} \label{thm:induction}
Suppose $\sfd : \R^{q \times q} \rightarrow \R^{q \times q}$ is a linear operator such that $\|\sfd \|_{\eu} \leq 1/10$ and $\sfi : \R^{q \times q} \rightarrow \R^{q \times q}$ is the identity operator.  Then we have that
\begin{equation*}
\sfi + \sfd  = (\sfi + \cp_{\W^{\perp}} (\sfd) + \sfg ) \circ \sfiw
\end{equation*}
where $\sfg : \R^{q \times q} \rightarrow \R^{q \times q}$ is a linear operator such that $\|\sfg\|_{\eu} \leq 5 \| \sfd \|_{\eu}^2 / \sqrt{q}$ and $\sfiw : \R^{q \times q} \rightarrow \R^{q \times q}$ is a linear rank-preserver such that $\|\sfiw - \sfi\|_{\spec} \leq  3\|\sfd\|_{\eu}/\sqrt{q}$.  Here, the space $\W$ is as defined in \eqref{eq:tangentspaceati}.
\end{proposition}

The proof of this proposition appears in Appendix \ref{apx:linearizemanifold}.  As detailed in the proof of Theorem \ref{thm:localconvergence} in Section \ref{sec:analysis_maintheoremproof}, one can combine the preceding three results along with the observation that $c ~ \cp_{\tstar} \preceq [({\L^\star_{\tstar}}' \L^\star_{\tstar})^{-1} ]_{\R^{q \times q}} \preceq \tilde{c} ~ \cp_{\tstar}$ for suitable constants $c,\tilde{c} > 0$ (from Lemma \ref{thm:weakincoherencebound} in Section \ref{sec:analysis_maintheoremproof} based on $\L^\star$ satisfying a suitable restricted isometry condition) to conclude that there exists an error term $\sfe^{(t+1)}$ at iteration $t+1$ (corresponding to the error term $\sfe^{(t)}$ at iteration $t$ that we fixed at the beginning of this argument) such that
\begin{equation} \label{eq:candidate_nexte}
\begin{aligned}
\sfe^{(t+1)} = & \cp_{\W^\perp} \circ \left[\frac{1}{n \coveig (\ensb)} \sum_{j=1}^n \left({X^{(j)}}^\star \boxtimes {X^{(j)}}^\star \right) \botimes \cp_{\ct({X^{(j)}}^\star)} \right]({\L^\star}' \L^\star \circ \sfe^{(t)})  \\ & + \cp_{\W^\perp}(\sff) + \mathcal{O}(\|\sfe^{(t)}\|_{\eu}^2).
\end{aligned}
\end{equation}
Thus, there are two `significant' terms in this expression that govern the size of $\|\sfe^{(t+1)}\|_{\eu}$.  To control the first term, we require a bound on the following operator norm:
\begin{equation} \label{eq:defn_omegaoperator}
\roc(\ensb) := \left\|\cp_{\W^\perp} \circ \left[\frac{1}{n} \sum_{j=1}^n \left({X^{(j)}}^\star \boxtimes {X^{(j)}}^\star \right) \botimes \cp_{\ct({X^{(j)}}^\star)} \right] \right\|_{\spec}.
\end{equation}
Note that this operator belongs to $\aut(\aut(\R^{q \times q}))$.  In Section \ref{sec:analysis_maintheoremproof} we show that the first significant term in \eqref{eq:candidate_nexte} is bounded as $\frac{2 \|\L^\star\|_{\spec}^2 \roc(\ensb)}{\coveig(\ensb)} \|\sfe^{(t)}\|_{\eu}$.  For the second term in \eqref{eq:candidate_nexte}, we show in Section \ref{sec:analysis_maintheoremproof} that $\|\sff\|_{\eu} \lesssim \frac{q^2 \|\L^\star\|_{\spec} \covsup(\ensb)}{\coveig(\ensb)} \|\sfe^{(t)}\|_{\eu}$ based on a bound on $\xi_{\L^\star}(\L^{(0)})$ on the initial guess.  Consequently, two of the key assumptions in Theorem \ref{thm:localconvergence} concern bounds on the quantities $\frac{\roc(\ensb)} {\coveig(\ensb)}$ and $\frac{\covsup(\ensb)} {\coveig(\ensb)}$.

We note that the Operator Sinkhorn scaling procedure for normalization is crucial in our algorithm.  Aside from addressing the identifiability issues as discussed in Section \ref{sec:algorithm_identifiablity}, the incorporation of this method also plays an important role in the convergence of Algorithm \ref{alg:amsdpreg}.  Specifically, if we do not apply this procedure in each iteration of Algorithm \ref{alg:amsdpreg}, then the estimate of $\L^\star$ at the end of iteration $t+1$ would be $\tilde{\L}^{(t+1)}$ from \eqref{eq:intermediatedictestimate}.  In analyzing how close the image of the nuclear norm ball under $\tilde{\L}^{(t+1)}$ is to the image of the nuclear norm ball under $\L^\star$, we would need to consider how close $\xx^\star \circ \hat{\xx}^+$ is to an \emph{orthogonal} rank-preserver as opposed to an arbitrary rank preserver; in particular, we cannot apply Proposition \ref{thm:normalizednearothogonal} as $\tilde{\L}^{(t+1)}$ is not normalized.  In analogy to the discussion preceding Proposition \ref{thm:induction} and by noting that $\sfi = I \botimes I$ is an orthogonal rank-preserver, we could attempt to express $\xx^\star \circ \hat{\xx}^+$ in terms of the following tangent space at $\sfi$ with respect to the set of orthogonal rank-preservers:
\begin{equation} \label{eq:defn_subspacekronskewsym}
\cs = \mathrm{span}\{I \botimes S_1 + S_2 \botimes I ~|~ S_1, S_2 \in \R^{q \times q} ~\text{and skew-symmetric}\}.
\end{equation}
Following similar reasoning as in the preceding paragraph, the convergence of our algorithm without normalization would be governed by $\|\cp_{\cs^\perp} \circ [\frac{1}{n} \sum_{j=1}^n ({X^{(j)}}^\star \boxtimes {X^{(j)}}^\star ) \botimes \cp_{\ct({X^{(j)}}^\star)} ] \|_{\spec}$.  This operator norm is, in general, much larger than the quantity $\roc(\ensb)$ defined in \eqref{eq:defn_omegaoperator} as $\cs \subset \W$, which can in turn affect the convergence of our algorithm.  In particular, for a natural random ensemble $\ensb$ of low-rank matrices described in Proposition \ref{thm:randensemble} in Section \ref{sec:analysis_randensemble}, the condition on $\roc(\ensb)$ in Theorem \ref{thm:localconvergence} is satisfied while the analogous condition on $\|\cp_{\cs^\perp} \circ [\frac{1}{n} \sum_{j=1}^n ({X^{(j)}}^\star \boxtimes {X^{(j)}}^\star ) \botimes \cp_{\ct({X^{(j)}}^\star)} ] \|_{\spec}$ is violated (both of these conclusions hold with high probability), thus highlighting the importance of the inclusion of the normalization step for the convergence of our method; see the remarks following Proposition \ref{thm:randensemble} for details.

\subsection{Main Result} \label{sec:analysis_mainresult}

The following theorem gives the main result concerning the local convergence of our algorithm:

\begin{theorem} \label{thm:localconvergence} Let $\by^{(j)} = \L^\star({X^{(j)}}^\star ), ~ j=1,\dots,n$, where $\L^\star : \R^{q \times q} \rightarrow \R^d$ is a linear map and $\ensb := \ensbfull \subset \R^{q \times q}$.  Suppose the collection $\ensb$ satisfies the following conditions:
\begin{enumerate}
\item There exists $r < q$ and $s > 0$ such that $\mathrm{rank}({X^{(j)}}^\star ) = r$ and $s \geq \sigma_1({X^{(j)}}^\star ) \geq \sigma_r({X^{(j)}}^\star ) \geq s / 2$ for each $j = 1,\dots,n$;
\item $\frac{\roc(\ensb )}{\coveig( \ensb )} \leq \frac{d}{40 q^2}$; and
\item $\frac{\covsup(\ensb )}{\coveig( \ensb )} \leq \frac{\sqrt{d}}{100 q^3}$.
\end{enumerate}
Suppose the linear map $\L^\star : \R^{q \times q} \rightarrow \R^d$ satisfies the following conditions:
\begin{enumerate}
\item $\L^\star$ satisfies the restricted isometry condition $\delta_{4r}(\L^\star) \leq \frac{1}{20}$, where $r$ is the rank of each ${X^{(j)}}^\star$;
\item $\L^\star$ is normalized and surjective; and
\item $\|\L^\star\|_{\spec}^2 \leq \frac{5 q^2}{d}$.
\end{enumerate}
If we supply Algorithm \ref{alg:amsdpreg} with a normalized initial guess $\L^{(0)} : \R^{q \times q} \rightarrow \R^d$ with $\xi_{\L^\star}(\L^{(0)}) < \frac{1}{20000 q^{7/2} r^2 \|\L^\star\|_{\spec}^2}$, then the sequence $\{\L^{(t)}\}$ produced by the algorithm satisfies $\limsup_{t \rightarrow \infty} \frac{\xi_{\L^\star}(\L^{(t+1)})}{\xi_{\L^\star}(\L^{(t)})} \leq 2 \|\L^\star\|_{\spec}^2 \frac{\roc(\ensb )}{\coveig( \ensb )} + 10 q^2 \|\L^\star\|_{\spec} \frac{\covsup(\ensb )}{\coveig( \ensb )} < 1$.  In other words, $\xi_{\L^\star}(\L^{(t)}) \rightarrow 0$ with the rate of convergence bounded above by $2 \|\L^\star\|_{\spec}^2 \frac{\roc(\ensb )}{\coveig( \ensb )} + 10 q^2 \|\L^\star\|_{\spec} \frac{\covsup(\ensb )}{\coveig( \ensb )}$.  We assume here that Step $1$ of Algorithm \ref{alg:amsdpreg} is computed via Algorithm \ref{alg:svp}.
\end{theorem}

\noindent \textbf{Remark.} $(i)$ In this result the assumption that Step $1$ of Algorithm \ref{alg:amsdpreg} is computed via Algorithm \ref{alg:svp} is made for the sake of concreteness.  A similar result and proof are possible if Step $1$ of Algorithm \ref{alg:amsdpreg} is instead computed by solving \eqref{eq:nuclearnormlasso} for a suitable choice of the regularization parameter. $(ii)$ In conjunction with Proposition \ref{thm:normalizednearothogonal}, this result implies that we obtain a linear map $\hat{\L}$ upon convergence of our algorithm such that the image of the nuclear norm ball in $\R^{q \times q}$ under $\hat{\L}$ is the same as it is under $\L^\star$.

The proof of this theorem is given in Section \ref{sec:analysis_maintheoremproof}.  In words, our result states that under a restricted isometry condition on the linear map $\L^\star$ and an isotropy condition on the low-rank matrices $\ensbfull$, Algorithm \ref{alg:amsdpreg} is locally linearly convergent to the appropriate semidefinite-representable regularizer that promotes the type of structure contained in the data $\{\L^\star({X^{(j)}}^\star) \}_{j=1}^n$.  The restricted isometry condition on $\L^\star$ ensures that the geometry of the set of points $\ensbfull$ in $\R^{q \times q}$ is (approximately) preserved in the lower-dimensional space $\R^d$.  The isotropy condition on the collection $\ensbfull$ ensures that we have observations that lie on most of the low-dimensional faces of the regularizer, which gives us sufficient information to reconstruct the regularizer.

Results of this flavor have previously been obtained in the classical dictionary learning literature \cite{AAJN:16,AGTM:15}, although our analysis is more challenging in comparison to this prior work for two reasons.  First, two nearby sparse vectors with the same number of nonzero entries have the same support, while two nearby low-rank matrices with the same rank have different row/column spaces; geometrically, this translates to the point that two nearby sparse vectors have the same tangent space with respect to a suitably defined variety of sparse vectors, while two nearby low-rank matrices generically have different tangent spaces with respect to an appropriate variety of low-rank matrices.  Second (and more significant), the normalization step in classical dictionary learning is simple -- corresponding to scaling the columns of a matrix to have unit Euclidean norm, as discussed in Section \ref{sec:algorithm_cdl} -- while the normalization step in our setting based on Operator Sinkhorn scaling is substantially more complicated.  Indeed, one of the key aspects of our analysis is the relation between the stability properties of Operator Sinkhorn scaling and the tangent spaces to varieties of low-rank matrices, as is evident from the appearance of the parameter $\roc(\ensbfull)$ in Theorem \ref{thm:localconvergence}.

The distance measure $\xi_{\L^{\star}}$ that appears in Theorem \ref{thm:localconvergence} is defined up to an equivalence relation, and with respect to the linear map $\L^{\star}$ to which we do not have access.  In practice, it is useful to have a stopping criterion that only depends on the sequence of iterates. To this end, the next result states that under the same conditions as in Theorem \ref{thm:localconvergence}, the sequence of iterates $\{\L^{(t)}\}$ obtained from our algorithm also converges (the limit point is generically different from $\L^\star$, although they specify the same regularizer):

\begin{proposition} \label{thm:cauchy}
Under the same setup and assumptions as in Theorem \ref{thm:localconvergence}, the sequence of iterates $\{\L^{(t)}\}$ obtained from our algorithm is a Cauchy sequence.
\end{proposition}

This result is proved in Appendix \ref{apx:cauchy}.

\paragraph{Extension to the noisy case.} In practice the data points $\by^{(j)}$ may be corrupted by noise, and it is of interest to investigate if our algorithm is robust to noise.  One can extend our analysis to demonstrate the robustness of our algorithm in a stylized setting in which the data points $\by^{(j)}$ in Theorem~\ref{thm:localconvergence} are corrupted by additive noise.  Briefly, such an extension comprises two key steps.  First, one can show that there exists a normalized linear map $\check{\L}$ that is close to $\L^\star$ (up to composition by an orthogonal rank-preserver), and which is a fixed-point of our algorithm.  The key ingredient in demonstrating this is to prove that each iteration of our algorithm is \emph{contractive} in a neighborhood of $\L^\star$ and to appeal to a suitable fixed-point theorem.  The proximity of the regularizer defined by $\check{\L}$ to the regularizer defined by $\L^\star$ is determined by the radius of contraction, which depends linearly (under the conditions of Theorem~\ref{thm:localconvergence}) on the size of the noise corrupting the measurements and inverse-polynomially on the size of the data set.  Second, one can show that our algorithm is locally linearly convergent to $\check{\L}$ (up to composition by an orthogonal rank-preserver).  This step essentially follows the same sequence of arguments as in the proof of Theorem~\ref{thm:localconvergence}, and it relies on the radius of contraction from the first step being smaller than the basin of attraction defined in Theorem~\ref{thm:localconvergence}; this is true as long as the noise corruption is suitably small and the number of data points is sufficiently large.

\subsection{Ensembles Satisfying the Conditions of Theorem \ref{thm:localconvergence}} \label{sec:analysis_randensemble}

Theorem \ref{thm:localconvergence} gives deterministic conditions on the underlying data under which our algorithm recovers the correct regularizer.  In this section we demonstrate that these conditions are in fact satisfied with high probability by certain natural random ensembles.  Our first result states that random Gaussian linear maps upon normalization satisfy the requirements on the linear map in Theorem \ref{thm:localconvergence}:

\begin{proposition}\label{thm:gaussianmapsatisfy}
Let $\tilde{\L} : \R^{q \times q} \rightarrow \R^d$ be a linear map in which each of the $d$ component linear functionals are specified by matrices $\tilde{\L}_{i} \in \R^{q \times q}$ with i.i.d random Gaussian entries with mean zero and variance $1/d$.  Let $\L$ represent a normalized map obtained by composing $\tilde{\L}$ with a positive-definite rank-preserver.  Fix any $\delta <1$.  Then there exist positive constants $c_1,c_2,c_3$ depending only on $\delta$ such that if $d \geq c_1 r q$, then $(i)$ $\delta_{4r}(\L) \leq \delta$ and $(ii)$ $\|\L\|_{\spec} \leq \sqrt{\frac{5q^2}{d}}$ with probability greater than $1 - c_2\exp(-c_3 d)$.
\end{proposition}

The proof of this result is given in Appendix \ref{apx:randlinmaps}.  As shown in \cite{CanPla:11} random Gaussian linear maps from $\R^{q \times q}$ to $\R^d$ satisfy the restricted isometry property for rank-$4r$ matrices if $d \gtrsim rq$ (and this bound is tight).  Our result shows that under the same scaling assumption on $d$, `most' linear maps satisfy the more restrictive requirements of Theorem \ref{thm:localconvergence}.  Next we consider families of random low-rank matrices:

\begin{proposition} \label{thm:randensemble} Let $\ensa:= \ensafull$ be an ensemble of matrices generated as $X^{(j)} = \sum_{i=1}^{r} s_{i}^{(j)} \bu_{i}^{(j)} \bv_{i}^{(j)\prime}$ with each $U^{(j)} = [\bu_{1}^{(j)}| \ldots | \bu_{r}^{(j)}], V^{(j)} = [\bv_{1}^{(j)}| \ldots | \bv_{r}^{(j)}] \in \R^{q \times r}$ being drawn independently from the Haar measure on $q \times r$ matrices with orthonormal columns, and each $s_{i}^{(j)}$ being drawn independently from $\mathcal{D}$, where $\mathcal{D}$ is any distribution supported on $[s/2,s]$ for some $s>0$.  Then for any $0 < t_1 \leq 1/4$ and $0 < t_2$, the conditions $(i)$ $\frac{\covsup(\ensa )}{\coveig( \ensa )} \leq t_1 $ and (ii) $\frac{\roc(\ensa )}{\coveig( \ensa )} \leq 80\frac{ r}{q} + t_2$, are satisfied with probability greater than $1- 2q\exp(-\frac{n t_1^2}{200q^4}) - q \exp(-\frac{n t_2^2}{200q^4})$.  In particular, the requirements in Theorem \ref{thm:localconvergence} for $d \gtrsim rq$ are satisfied with high probability by the ensemble $\ensa$ provided $n \gtrsim \frac{q^{10}}{d}$.
\end{proposition}

Considering the requirements of Theorem \ref{thm:localconvergence} in the regime $d \gtrsim rq$ is not restrictive as this condition is necessary for the restricted isometry assumptions of Theorem \ref{thm:localconvergence} on $\L^\star$ to hold.  The proof of this result is given in Appendix \ref{apx:randensemble}.  Thus, in some sense, `most' (sufficiently large) sets of low-rank matrices satisfy the requirements of Theorem \ref{thm:localconvergence}.  We also note that for a collection of low-rank matrices $\ensa$ generated according to the ensemble in this proposition, the ratio $\frac{\covsup(\ensa )}{\coveig( \ensa )} \rightarrow 0$ as $n \rightarrow \infty$, while one can show that the ratio $\frac{\roc(\ensa )}{\coveig( \ensa )} \asymp \frac{r}{q}$ as $n \rightarrow \infty$.  Based on Theorem \ref{thm:localconvergence}, this observation implies that for data generated according to the ensemble in Proposition \ref{thm:randensemble}, the rate of convergence of Algorithm \ref{alg:amsdpreg} improves with an increase in the amount of data, but only up to a certain point beyond which the convergence rate plateaus.  We illustrate this property with a numerical experiment in Section \ref{sec:numexp_synthetic}.

\noindent \textbf{Remark.} It is critical in the preceding result that we project onto the orthogonal complement of the subspace $\W$ from \eqref{eq:defn_omegaoperator} in the definition of $\roc( \ensa )$.  For a set of low-rank matrices $\ensa$ drawn from the same ensemble as in Proposition \ref{thm:randensemble}, one can show that
$\| \cp_{\cs^{\perp}} \circ \frac{1}{n}\sum_{j=1}^{n} (X^{(j)} \boxtimes X^{(j)}) \boldsymbol \botimes \cp_{\ct(X^{(j)})} \|_{\spec} > c \coveig( \ensa )$ for a constant $c>0$ with high probability, where the subspace $\cs$ is defined in \eqref{eq:defn_subspacekronskewsym}.  In the context of the discussion at the end of the preceding section, we have that the conditions of Theorem \ref{thm:localconvergence} are violated if we do not incorporate the normalization step via Operator Sinkhorn scaling, which in turn impacts the convergence of our algorithm.

\subsection{Proof of Theorem \ref{thm:localconvergence}} \label{sec:analysis_maintheoremproof}

Before giving a proof of Theorem \ref{thm:localconvergence}, we state two relevant lemmas that are proved in Appendix \ref{apx:prelims}.

\begin{lemma}\label{thm:weakincoherencebound}
	Suppose a linear map $\L : \R^{q \times q} \rightarrow \R^d$ satisfies the restricted isometry condition $\delta_{2r}(\L) <1$.  For any $\ct:= \ct(X)$ with $X \in \R^{q \times q}$ and $\mathrm{rank}(X) \leq r$, we have that $(i)$ $1-\delta_{2r}\leq \lambda_{\min}(\L^{\prime}_{\ct} \L_{\ct}) \leq \lambda_{\max}(\L^{\prime}_{\ct} \L_{\ct}) \leq 1+\delta_{2r}$, $(ii)$ $\| (\L_\ct' \L_{\ct})^{-1} \|_{\spec}=\| [(\L_{\ct}^{\prime} \L_{\ct})^{-1}]_{\R^{q\times q}} \|_{\spec} \leq \frac{1}{1-\delta_{2r}} $, $(iii)$ $\| \cp_\ct \circ \L^{\prime} \L\|_{\spec} \leq \sqrt{1+\delta_{2r}}\|\L\|_{\spec}$, and $(iv)$ $\| [(\L_{\ct}^{\prime} \L_{\ct})^{-1}]_{\R^{q\times q}} \circ \L' \L\|_{\spec} \leq \frac{\sqrt{1+\delta_{2r}}}{1-\delta_{2r}}\|\L\|_{\spec}$.  Here $\L_\ct' \L_\ct : \ct \rightarrow \ct$ is a self-adjoint linear map.
\end{lemma}

\begin{lemma}\label{thm:boundondelta}
Let $\ensa := \ensafull \subset \mathbb{R}^{q\times q}$ be a collection of matrices, and let $s_{\min}:= \min_{j} \|X^{(j)}\|_{\fro}^2$ and $s_{\max}:=\max_{j} \|X^{(j)}\|_{\fro}^2$.  Then $ s_{\min}/q^2 - \covsup(\ensa) \leq \coveig (\ensa) \leq s_{\max}/q^2 + \covsup(\ensa)$.
\end{lemma}

\begin{proof}[Theorem \ref{thm:localconvergence}]
To simplify the presentation of our proof we define the following quantities $\alpha_0:=20000q^{7/2}r^2\|\L^{\star}\|_{\spec}^2$, $\alpha_1:= 800 r^{5/2} \|\L^{\star}\|_{\spec}^2$, $\alpha_2 := 2 \sqrt{r} \|\L^{\star}\|_{\spec}$, $\alpha_3:= 10 q^2 \|\L^{\star}\|_{\spec}$, $\alpha_4 := 5 (q^2/\sqrt{r}) \alpha_1$, $\alpha_5:= 100 q^3 \alpha_2^2$, $\alpha_6:=5 (q^2/\sqrt{r}) \alpha_2$, and $\alpha_7: = \alpha_3 + \alpha_6/6 + 1/4$.  The specific interpretation of these quantities is not essential to the proof -- the pertinent detail is that they only depend on $q,r,\|\L^{\star}\|_{\spec}$.

To simplify notation in the proof we denote $\covsup:= \covsup (\ensa)$, $\coveig:= \coveig (\ensa)$, and $\roc:= \roc (\ensa)$.  In addition we also denote $\ct^{(j)}:= \ct ({X^{(j)}}^{\star} )$.  Our proof proceeds by establishing the following assertion.  Suppose that the $t$-th iterate $\L^{(t)}$ is such that $\L^{(t)} = \L^{\star} \circ (\sfi + \sfe^{(t)}) \circ \sfim^{(t)}$, where $\sfim^{(t)}$ is a rank-preserver, and $\sfe^{(t)}$ is a linear operator that satisfies $\|\sfe^{(t)}\|_{\eu} < 1/ \alpha_0$.  Then the $t+1$-th iterate is of the form $\L^{(t+1)} = \L^{\star} \circ (\sfi + \sfe^{(t+1)}) \circ \sfim^{(t+1)}$ for some rank-preserver $\sfim^{(t+1)}$, and some linear operator $\sfe^{(t+1)}$ that satisfies
	\begin{equation} \label{eq:maininductivestep}
	\|\sfe^{(t+1)}\|_{\eu} \leq \gamma_0 \|\sfe^{(t)}\|_{\eu} + \gamma_1 \|\sfe^{(t)}\|_{\eu}^2,
	\end{equation}
	where $\gamma_0 = 2 \|\L^{\star}\|_{\spec}^2 (\roc/\coveig) + \alpha_6 (\covsup/\coveig) $, and $\gamma_1 = \alpha_4+\alpha_5 + 5\alpha_7^2/\sqrt{q}$.
	
	Before we prove this assertion, we note how it allows us to conclude the result.  By taking the infimum over $\sfe^{(t)}$ on the right hand side of \eqref{eq:maininductivestep} and by noting that $\xi_{\L^\star}(\L^{(t+1)}) \leq \|\sfe^{(t+1)}\|_{\eu}$, we have
	\begin{equation} \label{eq:maininductivestep_withxi}
	\xi_{\L^{\star}}(\L^{(t+1)}) \leq \gamma_0 \xi_{\L^{\star}}(\L^{(t)}) + \gamma_1 \xi_{\L^{\star}}(\L^{(t)})^2.
	\end{equation}
	One can check based on the initial assumption on $\xi_{\L^{\star}}(\L^{(0)})$ that $\gamma:=\gamma_0 + \gamma_1 \xi_{\L^{\star}}(\L^{(0)}) < 1$.  By employing an inductive argument one can establish that $\xi_{\L^{\star}}(\L^{(t+1)}) \leq \gamma \xi_{\L^{\star}}(\L^{(t)})$.  Thus $\xi_{\L^{\star}}(\L^{(t)}) \leq \gamma^{t} \xi_{\L^{\star}}(\L^{(0)}) \rightarrow 0$ as $t \rightarrow \infty$.  By dividing the expression in \eqref{eq:maininductivestep_withxi} throughout by $\xi_{\L^{\star}}(\L^{(t)})$, and subsequently taking the limit $t\rightarrow \infty$, we obtain the asymptotic rate of convergence
	\begin{equation*}
	\limsup_{t \rightarrow \infty} \frac{\xi_{\L^{\star}}(\L^{(t+1)})}{\xi_{\L^{\star}}(\L^{(t)})} \leq \limsup_{t\rightarrow \infty} \bigl(\gamma_0 + \gamma_1 \xi_{\L^{\star}}(\L^{(t)}) \bigr) = \gamma_0.
	\end{equation*}
	
	We proceed to prove the assertion.
		
	[Applying Proposition \ref{thm:varietyconstrainederr}]: Since $\|\sfe^{(t)}\|_{\eu} \leq \min\{1/(50\sqrt{q}), 1/(120r^2\|\L^{\star}\|_{\spec}) \}$, by applying Proposition \ref{thm:varietyconstrainederr} with the choice of $X^{\star} = {X^{(j)}}^{\star}$,  $\sfe = \sfe^{(t)}$, $\sfim = \sfim^{(t)}$, and $\L^{\star}$, we have for each $j = 1, \dots, n$ that
\begin{equation} \label{eq:proofvaropt}
\begin{aligned}
	\sfim^{(t)} (\hat{X}^{(j)}) - {X^{(j)}}^{\star} = & - \biggl[ [(\L^{\star\prime}_{\ct^{(j)}}\L^{\star}_{\ct^{(j)}})^{-1}]_{\R^{q\times q}} \circ \L^{\star\prime} \L^{\star} \circ \sfe^{(t)} \biggr] \bigl({X^{(j)}}^{\star}\bigr)+ \err^{(j)},
\end{aligned}
\end{equation}
where $\err^{(j)}$ is a matrix that satisfies $\|\err^{(j)}\|_{\fro} \leq \alpha_1 \|{X^{(j)}}^{\star}\|_{\spec}\|\sfe^{(t)}\|_{\eu}^2$.
	
	[Applying Proposition \ref{thm:structuralresult}]:  The next step is to apply Proposition \ref{thm:structuralresult} to the collections of matrices $\ensbfull$ and $\{\hat{X}^{(j)} \}_{j=1}^{n}$.  Let $\xx^{\star}, \hat{\xx}$ denote the linear maps $\xx^{\star}:\bz \mapsto \sum_{j=1}^{n} {X^{(j)}}^{\star} \bz_j$, $\hat{\xx}:\bz \mapsto \sum_{j=1}^{n} \hat{X}^{(j)} \bz_j$.  First note that $\alpha_1 \|\sfe^{(t)} \|_{\eu} \leq \alpha_1 / \alpha_0 \leq \sqrt{r}\|\L^{\star}\|_{\spec} $.  Second from the assumptions we have $\covsup / \coveig \leq 1/21$.  Hence by Lemma \ref{thm:boundondelta} we have $\coveig \leq s^2 r/q^2 + \covsup \leq s^2 r/q^2 + \coveig / 21$.  It follows that $\covsup \leq s^2 r / (20q^2)$, and thus by Lemma \ref{thm:boundondelta} we have $\coveig \geq s^2 r/(5q^2)$.  Third by applying these inequalities and Lemma \ref{thm:weakincoherencebound} to \eqref{eq:proofvaropt} we have $\|\sfim^{(t)} (\hat{X}^{(j)}) - {X^{(j)}}^{\star} \|_{\fro} \leq ((\sqrt{1+\delta_{4r}})/(1-\delta_{4r})) \|\L^{\star}\|_{\spec} \|{X^{(j)}}^{\star}\|_{\fro} \|\sfe^{(t)}\|_{\eu} +  \alpha_1 \|{X^{(j)}}^{\star}\|_{\spec} \| \sfe^{(t)} \|_{\eu}^{2} \leq s\alpha_2 / \alpha_0 \leq s \alpha_2 \| \sfe^{(t)} \|_{\eu} \leq \sqrt{\coveig} / 20$.  Fourth note that the assumptions imply $\covsup/\coveig \leq 1/6$.  Hence by Proposition \ref{thm:structuralresult} applied to $\ensbfull$ and $\{\hat{X}^{(j)} \}_{j=1}^{n}$ with the choice of $\sfq = \sfim^{(t)}$ we have
	\begin{equation*}
	\xx^{\star} \circ \hat{\xx}^{+} = ( \sfi + \sfd  ) \circ \sfim^{(t)},
	\end{equation*}
where
\begin{align*}
	\sfd := & \frac{1}{n \coveig} \sum_{j=1}^{n} ([ [(\L^{\star\prime}_{\ct^{(j)}}\L^{\star}_{\ct^{(j)}})^{-1}]_{\R^{q\times q}} \ \circ \L^{\star\prime} \L^{\star} \circ \sfe^{(t)}] ({X^{(j)}}^{\star})) \boxtimes {X^{(j)}}^{\star} \\
	& - \frac{1}{n\coveig} \sum_{j=1}^{n} G^{(j)} \boxtimes {X^{(j)}}^{\star} + \sff,
\end{align*}
and
	\begin{eqnarray} \label{eq:feucbound}
	\|\sff\|_{\eu} & \leq & 20 q (s \alpha_2 \|\sfe^{(t)}\|_{\eu})^2 / \coveig  + 2 q (s \alpha_2 \|\sfe^{(t)}\|_{\eu}) \covsup / \coveig^{3/2} \nonumber \\
	& \leq & \alpha_5\|\sfe^{(t)}\|_{\eu}^{2} + \alpha_6(\covsup/\coveig)\| \sfe^{(t)}\|_{\eu}.
	\end{eqnarray}
	
	[Applying Proposition \ref{thm:induction}]:  We proceed to bound $\|\sfd \|_{\eu}$.  Given a collection $\{A^{(j)} \}_{j=1}^{n}, \{B^{(j)} \}_{j=1}^{n} \subset \mathbb{R}^{q\times q}$ one has $\frac{1}{n} \|\sum_{j=1} A^{(j)}\boxtimes B^{(j)}\|_{\eu} \leq \max_j \|A^{(j)}\boxtimes B^{(j)}\|_{\eu} = \max_j \|A^{(j)}\|_{\fro} \|B^{(j)}\|_{\fro}$.  By combining this inequality with Lemma \ref{thm:weakincoherencebound} we obtain the bounds
	\begin{eqnarray} \label{eq:heucbound}
	& & \frac{1}{n\coveig} \left\|\sum_{j=1}^{n}  \bigl(\bigl[ [(\L^{\star\prime}_{\ct^{(j)}}\L^{\star}_{\ct^{(j)}})^{-1}]_{\R^{q\times q}} \circ \L^{\star\prime} \L^{\star} \circ \sfe^{(t)} \bigr] ({X^{(j)}}^{\star})\bigr) \boxtimes {X^{(j)}}^{\star}  \right\|_{\eu} \nonumber\\
	& \leq & (2 s^2 r \|\L^{\star}\|_{\spec}/\coveig) \|\sfe^{(t)}\|_{\eu} \leq \alpha_3 \|\sfe^{(t)}\|_{\eu},
	\end{eqnarray}
	and
	\begin{equation}\label{eq:geucbound}
	(1/n\coveig)\left\| \sum_{j=1}^{n} G^{(j)} \boxtimes {X^{(j)}}^{\star} \right\|_{\eu} \leq (\alpha_1 s^2 \sqrt{r}/ \coveig) \| \sfe^{(t)} \|_{\eu}^2 \leq \alpha_4 \| \sfe^{(t)} \|_{\eu}^2.
	\end{equation}
	Hence by combining \eqref{eq:feucbound}, \eqref{eq:heucbound}, and \eqref{eq:geucbound} we have $\|\sfd\|_{\eu} \leq \alpha_3 \|\sfe^{(t)}\|_{\eu} + \alpha_4 \| \sfe^{(t)} \|_{\eu}^2 + \alpha_5\|\sfe^{(t)}\|_{\eu}^{2} + \alpha_6(\covsup/\coveig)\| \sfe^{(t)}\|_{\eu} \leq \alpha_7 \| \sfe^{(t)}\|_{\eu} \leq \alpha_7 /\alpha_0 \leq 1/10$.  Consequently, by applying Proposition \ref{thm:induction} with this choice of $\sfd$, we have
	\begin{equation} \label{eq:xxhat}
	\xx^{\star} \circ \hat{\xx}^{+} = (\sfi + \cp_{\W^{\perp}} (\sfd) + \sfg ) \circ \sfiw \circ \sfim^{(t)}, \quad \|\sfg\|_{\eu}\leq (5\alpha_7^2/\sqrt{q})\|\sfe^{(t)}\|^2_{\eu},
	\end{equation}
	for some rank-preserver $\sfiw$.
	
	[Conclusion]: Recall from the description of the algorithm that the next iterate is given by $\L^{(t+1)} = \L^{\star} \circ \xx^{\star} \circ \hat{\xx}^{+} \circ \sfin_{\L^{\star} \circ \xx^{\star} \circ \hat{\xx}^{+}}$, where $\sfin_{\L^{\star} \circ \xx^{\star} \circ \hat{\xx}^{+}}$ is the unique positive definite rank-preserver that normalizes $\L^{\star} \circ \xx^{\star} \circ \hat{\xx}^{+}$.  We define $\sfe^{(t+1)} := \cp_{\W^{\perp}} (\sfd) + \sfg$, and hence
	\begin{equation} \label{eq:nextdictestimate}
	\L^{(t+1)} = \L^{\star} \circ (\sfi + \sfe^{(t+1)}) \circ \sfim^{(t+1)},
	\end{equation}
	where $\sfim^{(t+1)} = \sfiw \circ \sfim^{(t)} \circ \sfin_{\L^{\star} \circ \xx^{\star} \circ \hat{\xx}^{+}}$ is a composition of rank-preservers, and hence is also a rank-preserver.  It remains to bound $\|\sfe^{(t+1)}\|_{\eu}$.
	
	As $\| [(\L^{\star\prime}_{\ct^{(j)}}\L^{\star}_{\ct^{(j)}})^{-1}]_{\R^{q \times q}} \|_{\spec} \leq 2$ from Lemma \ref{thm:weakincoherencebound}, we have $ [(\L^{\star\prime}_{\ct^{(j)}}\L^{\star}_{\ct^{(j)}})^{-1}]_{\R^{q \times q}} \allowbreak \preceq 2\cp_{\ct^{(j)}} $, and hence $ ({X^{(j)}}^{\star} \boxtimes {X^{(j)}}^{\star}) \botimes [(\L^{\star\prime}_{\ct^{(j)}}\L^{\star}_{\ct^{(j)}})^{-1}]_{_{\R^{q \times q}}} \preceq 2 ({X^{(j)}}^{\star} \boxtimes {X^{(j)}}^{\star}) \botimes \cp_{\ct^{(j)}} $.  Moreover, since $ ({X^{(j)}}^{\star} \boxtimes {X^{(j)}}^{\star}) \botimes [(\L^{\star\prime}_{\ct^{(j)}}\L^{\star}_{\ct^{(j)}})^{-1}]_{\R^{q \times q}}$ and $2 ({X^{(j)}}^{\star} \boxtimes {X^{(j)}}^{\star}) \botimes \cp_{\ct^{(j)}}$ are Kronecker products of positive semidefinite operators, they too are positive semidefinite operators, and hence $ \cp_{\W^{\perp}} \circ  (\frac{1}{n} \sum_{j=1}^{n} ({X^{(j)}}^{\star} \boxtimes {X^{(j)}}^{\star}) \botimes [(\L^{\star\prime}_{\ct^{(j)}}\L^{\star}_{\ct^{(j)}})^{-1}]_{\R^{q \times q}})^2 \circ \cp_{\W^{\perp}} \preceq \cp_{\W^{\perp}} \circ (\frac{2}{n} \sum_{j=1}^{n} ({X^{(j)}}^{\star} \boxtimes {X^{(j)}}^{\star}) \botimes \cp_{\ct^{(j)}})^2 \circ \cp_{\W^{\perp}}$.  This implies the bound
	\begin{equation*}
	2 \roc \geq \left\| \cp_{\W^{\perp}} \circ \left(\frac{1}{n} \sum_{j=1}^{n} ({X^{(j)}}^{\star} \boxtimes {X^{(j)}}^{\star}) \botimes [(\L^{\star\prime}_{\ct^{(j)}}\L^{\star}_{\ct^{(j)}})^{-1}]_{\R^{q \times q}}  \right) \right\|_{\spec}.
	\end{equation*}
	Combining this bound with the identity $\mathsf{L}(X_1) \boxtimes X_2 = \mathsf{L} \circ (X_1 \boxtimes X_2)$ we obtain
	\begin{eqnarray} \label{eq:opbound}
		& & \frac{1}{n\coveig}\biggl\| \cp_{\W^{\perp}} \biggl(\sum_{j=1}^{n} \biggl( \biggl[  [(\L^{\star\prime}_{\ct^{(j)}}\L^{\star}_{\ct^{(j)}})^{-1}]_{\R^{q \times q}} \circ \L^{\star\prime} \L^{\star} \circ \sfe^{(t)} \biggr] \bigl({X^{(j)}}^{\star}\bigr)\biggr) \boxtimes {X^{(j)}}^{\star} \biggr) \biggr\|_{\eu} \nonumber\\
		& = & \frac{1}{n\coveig} \biggl\| \biggl[\cp_{\W^{\perp}} \circ \biggl( \sum_{j=1}^{n} \bigl({X^{(j)}}^{\star} \boxtimes {X^{(j)}}^{\star} \bigr) \botimes [(\L^{\star\prime}_{\ct^{(j)}}\L^{\star}_{\ct^{(j)}})^{-1}]_{\R^{q \times q}} \biggr) \biggr] (\L^{\star\prime} \L^{\star} \circ \sfe^{(t)}) \biggr\|_{\eu} \nonumber \\	
		& \leq & (2\roc/\coveig) \| \L^{\star\prime} \L^{\star} \circ \sfe^{(t)} \|_{\eu}  \leq (2\roc/\coveig) \|\L^{\star}\|_{\spec}^2 \|\sfe^{(t)} \|_{\eu}.
	\end{eqnarray}	
	From the definition of $\sfe^{(t+1)}$ we have the relation
	\begin{align} \label{eq:nextiterate}
	\sfe^{(t+1)} ~ = ~ & \cp_{\W^{\perp}} \biggl( \frac{1}{n \coveig} \sum_{j=1}^{n} [  [(\L^{\star\prime}_{\ct^{(j)}}\L^{\star}_{\ct^{(j)}})^{-1}]_{\R^{q \times q}}  \circ \L^{\star\prime} \L^{\star} \circ \sfe^{(t)}] \bigl({X^{(j)}}^{\star}\bigr) \boxtimes {X^{(j)}}^{\star} \nonumber \\
	& + \frac{1}{n \coveig} \sum_{j=1}^{n} G^{(j)} \boxtimes {X^{(j)}}^{\star} + \sff \biggr) + \sfg.
	\end{align}
	Since $\cp_{\W^{\perp}}$ defines a projection, we have $(1/n \coveig) \| \cp_{\W^{\perp}} (\sum_{j=1}^{n} G^{(j)} \boxtimes {X^{(j)}}^{\star} ) \|_{\eu} \leq (1/n \coveig) \| \sum_{j=1}^{n} G^{(j)} \boxtimes {X^{(j)}}^{\star} \|_{\eu}$, and $\|\cp_{\W^{\perp}}(\sff)\|_{\eu} \leq \|\sff\|_{\eu}$.  Hence, by applying the bounds \eqref{eq:feucbound}, \eqref{eq:geucbound}, \eqref{eq:xxhat}, and \eqref{eq:opbound} to \eqref{eq:nextiterate}, we obtain
	\begin{eqnarray*}
	\|\sfe^{(t+1)}\|_{\eu} & \leq & ( (2\roc/\coveig) \|\L^{\star}\|_{\spec}^2 + \alpha_6 (\covsup/\coveig) ) \|\sfe^{(t)} \|_{\eu} + ( \alpha_4+\alpha_5 + 5\alpha_7^2/\sqrt{q} ) \|\sfe^{(t)}\|_{\eu}^2 \\
	& = & \gamma_0\|\sfe^{(t)}\|_{\eu} + \gamma_1\|\sfe^{(t)}\|_{\eu}^{2}.
	\end{eqnarray*}
	This completes the proof. \qed
\end{proof}

%% file: sc_numexp.tex
\section{Numerical Experiments} \label{sec:numexp}

\begin{figure}
	\centering
	\includegraphics[width=0.4\textwidth]{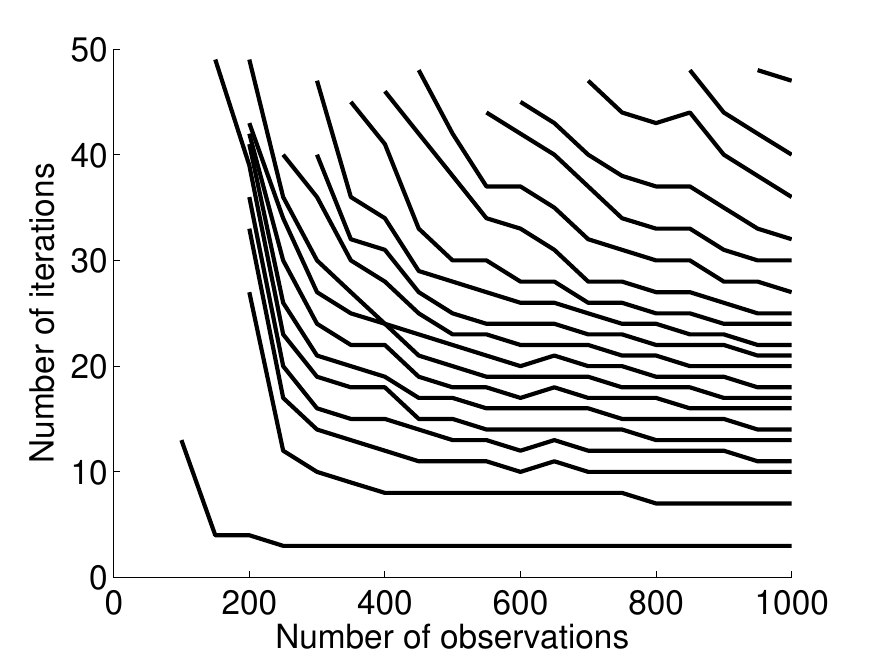}
	\label{fig:numberiterations_graph}
	\caption{Average number of iterations required to identify correct regularizer as a function of the number of observations; each line represents a fixed noise level $\sigma$ denoting the amount of corruption in the initial guess (see Section \ref{sec:numexp_synthetic} for details of the experimental setup).}
	\label{fig:numberiterations}
\end{figure}

\subsection{Illustration with Synthetic Data} \label{sec:numexp_synthetic}
We begin with a demonstration of the utility of our algorithm in recovering a regularizer from synthetic data.  Our experiment qualitatively confirms the predictions of Theorem \ref{thm:localconvergence} regarding the rate of convergence.

\textbf{Setup.} We generate a standard Gaussian linear map $\L : \R^{7 \times 7} \rightarrow \R^{30}$ and we normalize it; denote the normalized version as $\L^\star$.  We generate data $\{\by^{(j)}\}_{j=1}^{1000}$ as $\by^{(j)} = \L^\star(\bu^{(j)} {\bv^{(j)}}')$, where each $\bu^{(j)}, \bv^{(j)}$ is drawn independently from the Haar measure on the unit sphere in $\R^{7}$.  We generate standard Gaussian maps $\mathcal{E}^{(i)} : \R^{7 \times 7} \rightarrow \R^{30}, ~ i=1,\dots,20$ that are used to corrupt $\L^\star$ in providing the initial guess to our algorithm.  Specifically, for each $\sigma \in \{0.125, 0.25, \dots, 2.5\}$ and each $\mathcal{E}^{(i)}, ~ i=1,\dots,20$ we supply as initial guess to our algorithm the normalized version of $\L^\star + \sigma \mathcal{E}^{(i)}$.  In addition we supply the subset $\{\by^{(j)}\}_{j=1}^m$ for each $m \in \{50, 100, \dots, 1000\}$ to our algorithm.  The objective of this experiment is to investigate the role of the number of data points (denoted by $m$) and the size of the error in the initial guess (denoted by $\sigma$) on the performance of our algorithm.

\textbf{Characterizing recovery of correct regularizer.} Before discussing the results, we describe a technique assessing whether our algorithm recovers the correct regularizer.  In particular, as we do not know of a tractable technique for computing the distance measure $\xi$ between two linear maps \eqref{eq:defn_distancebetweendicts}, we consider an alternative approach for computing the `distance' between two linear maps. For linear maps from $\R^{q \times q}$ to $\R^d$, we fix a set of unit-Euclidean-norm rank-one matrices $\{\bs^{(k)} {\bt^{(k)}}' \}_{k=1}^{\ell}$, where each $\bs^{(k)}, \bt^{(k)} \in \R^{q}$ is drawn uniformly from the Haar measure on the sphere and $\ell$ is chosen to be larger than $q^2$.  Given an estimate $\L : \R^{q \times q} \rightarrow \R^{d}$ of a linear map $\L^\star : \R^{q \times q} \rightarrow \R^d$, we compute the following
\begin{equation} \label{eq:distmeasuretwomaps}
\mathrm{dist}_{\L^\star}(\L):=\frac{1}{\ell} \sum_{k=1}^{\ell} ~ \underset{\substack{X \in \R^{q \times q} \\ \mathrm{rank}(X) \leq 1}}{\inf} \left\| \L^\star \left(\bs^{(k)} {\bt^{(k)}}'\right) - \L(X) \right\|_{\ell_2}^2.
\end{equation}
To compute the minimum for each term in the sum, we employ the heuristic described in Algorithm \ref{alg:svp}.  If $\L^\star$ satisfies a suitable restricted isometry condition for rank-one matrices and if $\L$ is specified as $\L^\star$ composed with a near-orthogonal rank-preserver, then we have that $\mathrm{dist}_{\L^\star}(\L) \approx 0$; in the opposite direction, as $\ell > q^2$, we have that $\mathrm{dist}_{\L^\star}(\L) \approx 0$ implies $\xi_{\L^{\star}}(\L) \approx 0$.  In our setting with $q = 7$ we set $\ell = 100$.  If our algorithm provides an estimate $\L$ such that $\mathrm{dist}_{\L^\star}(\L) < 10^{-3}$, then we declare that our method has succeeded in recovering the correct regularizer.

\begin{figure}
	\centering
	\begin{minipage}{.5\textwidth}
		\centering
		\includegraphics[width=0.8\textwidth]{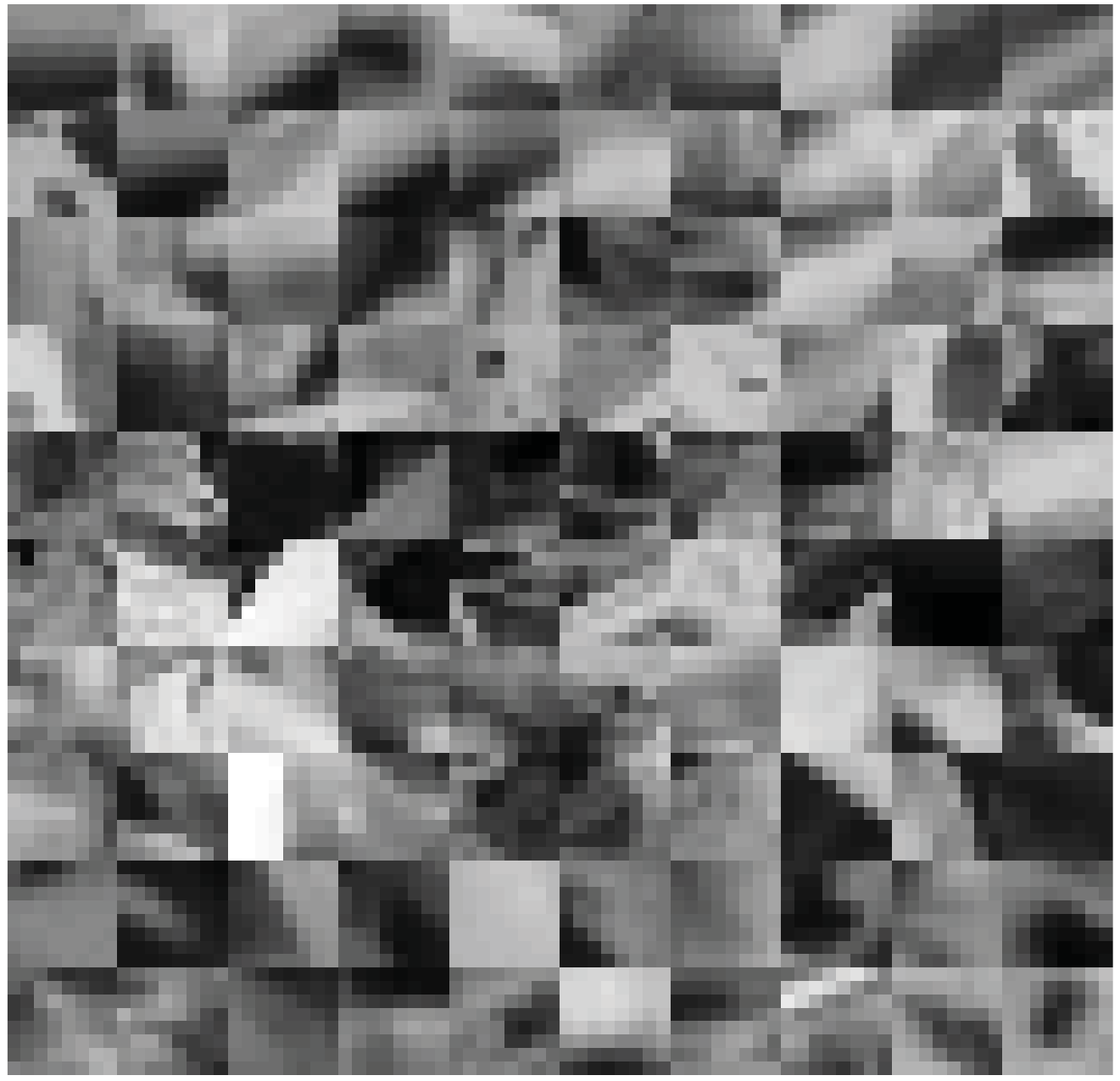}
		\label{fig:imagepatches}
	\end{minipage}%
	\begin{minipage}{.5\textwidth}
		\centering
		\includegraphics[width=0.62\textwidth]{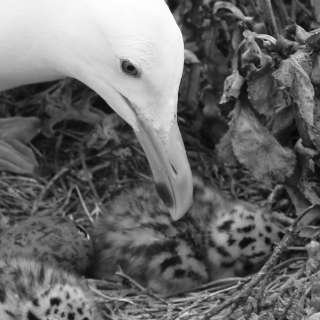}
		\label{fig:samplerawimg}
	\end{minipage}
	\caption{Image patches (left) obtained from larger raw images (sample on the right).}
	\label{fig:rawimages}
\end{figure}

\textbf{Results.} In Figure \ref{fig:numberiterations} we plot for each $\sigma \in \{0.125,0.25,\dots,2.5\}$ the average number of iterations -- taken over the $20$ different initial guesses specified by the normalized versions of $\L^\star + \sigma \mathcal{E}^{(i)}, ~ i=1,\dots,20$ -- required for Algorithm \ref{alg:amsdpreg} (with Step $1$ computed by solving \eqref{eq:varietyconstrainedopt} via Algorithm \ref{alg:svp}) to succeed in recovering the correct regularizer as a function of the number of data points $m$ supplied as input.  The different curves in the figure correspond to different noise levels (specified by $\sigma$) in the initial guess; that is, the curves higher up in the figure are associated to larger noise levels.  There are two main conclusions to be drawn from this result.  First, the average number of iterations grows as the initial guess is of increasingly poorer quality.  Second, and more interesting, is that the number of iterations required for convergence improves with an increase in the number of input data points, but only up to a certain stage beyond which the convergence rate seems to plateau (this is a feature at every noise level in this plot).  This observation confirms the predictions of Theorem \ref{thm:localconvergence} and of Proposition \ref{thm:randensemble} (specifically, see the discussion immediately following this proposition).

\subsection{Illustration with Natural Images} \label{sec:numexp_raw}

\subsubsection{Representing Natural Image Patches} \label{sec:numexp_raw_rep}

The first stage of this experiment contrasts projections of low-rank matrices and projections of sparse vectors purely from the perspective of representing a collection of image patches.

\textbf{Setup.} We consider a data set $\{ \by^{(j)} \}_{j=1}^{6480} \in \R^{64}$ of image patches.  This data is obtained by taking $8 \times 8$ patches from larger images of seagulls and considering these patches as well as their rotations, as is common in the dictionary learning literature; Figure \ref{fig:rawimages} gives an example of a seagull image as well as several smaller patches.  To ensure that we learned a centered and suitably isotropic norm, we center the entire data set to ensure that the average of the $\by^{(j)}$'s is the origin and then scale each datapoint so that it has unit Euclidean norm.  We apply Algorithm \ref{alg:amsdpreg} (with Step $1$ computed by solving \eqref{eq:varietyconstrainedopt} via Algorithm \ref{alg:svp}) and the analog of this procedure for dictionary learning described in Section \ref{sec:algorithm_cdl}.  We assess the quality of the description of the data set $\{ \by^{(j)} \}_{j=1}^{6480}$ as a projection of low-matrices (obtained using our approach) as opposed to a projection of sparse vectors (obtained using dictionary learning).

\textbf{Representation complexity.} To assess the performance of each representation framework, we require a characterization of the number of parameters needed to specify an image patch in each representation as well as the resulting quality of approximation.  Given a collection $\{\by^{(j)}\}_{j=1}^n \subset \R^d$, suppose we represent each point as $\by^{(j)} \approx \L (X^{(j)})$ for a linear map $\L : \R^{q \times q} \rightarrow \R^d$ and a rank-$r$ matrix $X^{(j)} \in \R^{q \times q}$.  The number of parameters required to specify each $X^{(j)}$ is $2qr - r^2$ and the number of parameters required to specify $\L$ is $d q^2$.  Consequently, the average number of parameters required to specify each $\by^{(j)}$ is $2qr - r^2 + \frac{dq^2}{n}$.  In a similar manner, if each $\by^{(j)} \approx L \bx^{(j)}$ for a linear map $L : \R^p \times \R^d$ and a vector $\bx^{(j)} \in \R^p$ with $s$ nonzero coordinates, the average number of parameters required to each $\by^{(j)}$ is $2s + \frac{dp}{n}$.  In each case, we assess the quality of the approximation by considering the average squared error over the entire set $\{\by^{(j)}\}_{j=1}^n$.

\textbf{Results.} We initialize both our algorithm and the dictionary learning method with random linear maps (suitably normalized in each case).  Before contrasting the two approaches we highlight the improvement in performance our method provides over a pure random linear map.  Specifically, Figure \ref{fig:progress} shows for several random initializations that our algorithm (as well as the alternating update method in dictionary learning) provides a significant refinement in approximation quality as the number of iterations increases.  Therefore, there is certainly value in employing our algorithm (even with a random initialization) to obtain better representations than pure random projections of low-rank matrices.  Next we proceed to a detailed comparison of the two representation frameworks.  We employ our approach to learn a representation of the image patch data set with $q \in \{9, 10,\dots, 15 \}$ and the values of the rank $r$ chosen so that the overall representation complexity lies in the range $[17,33]$.  Similarly, we employ dictionary learning with $p \in \{100, 200, \dots, 1400\}$ and the values of the sparsity level $s$ chosen so that the overall representation complexity lies in the range $[17,33]$.  The left subplot in Figure \ref{fig:representanddenoise} gives a comparison of these two frameworks.  (To interpret the $y$-axis of the plot, note that the each data point is scaled to have unit norm.) Our approach provides an improvement over dictionary learning for small levels of representation complexity and is comparable at larger levels.

\begin{figure}
	\centering
	\begin{minipage}{.5\textwidth}
		\centering
		\includegraphics[width=0.6\textwidth]{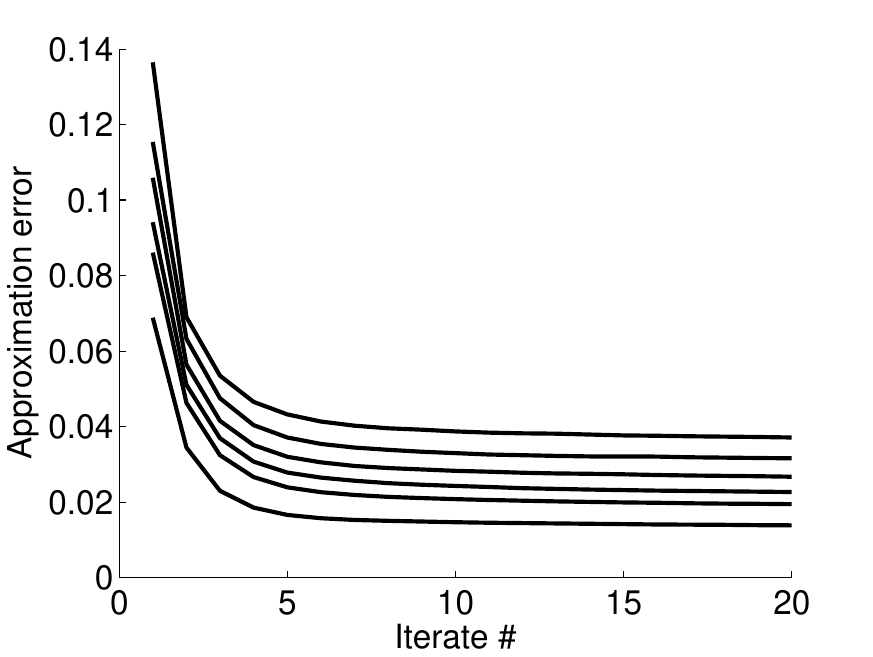}
		\label{fig:progress_lrm}
	\end{minipage}%
	\begin{minipage}{.5\textwidth}
		\centering
		\includegraphics[width=0.6\textwidth]{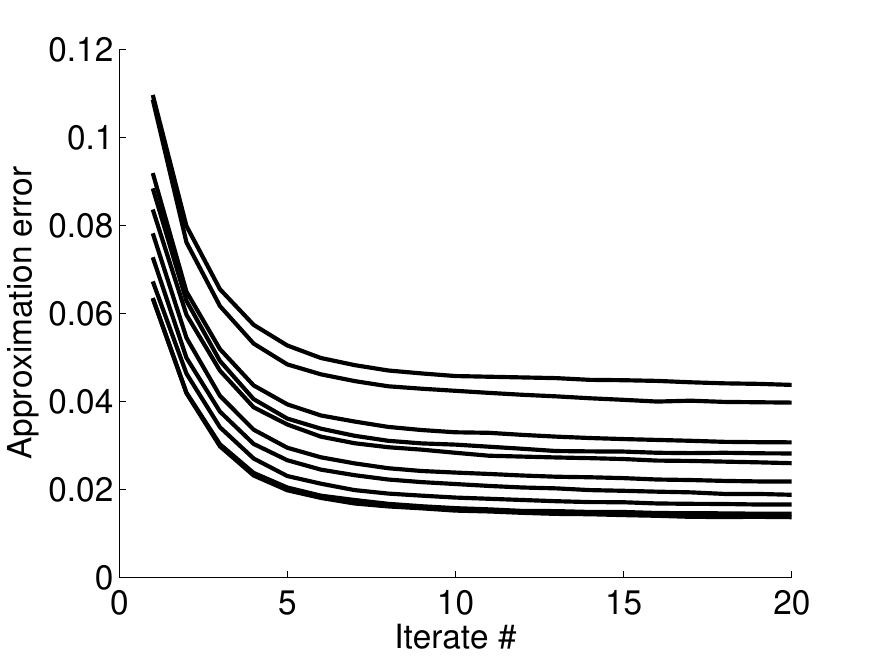}
		\label{fig:progress_sv}
	\end{minipage}
	\caption{Progression in mean-squared error with increasing number of iterations with random initializations for learning a semidefinite regularizer (left) and a polyhedral regularizer (right).}
	\label{fig:progress}
\end{figure}

	\textbf{Comparison of atoms.} Figure \ref{fig:atoms} gives an illustration of the atoms obtained from classical dictionary learning (i.e., learning a polyhedral regularizer) as well as those learned using our approach.  The left subplot shows the finite collection of atoms of a polyhedral regularizer (corresponding to the finite number of extreme points), and the right subplot shows a finite subset of the infinite collection of atoms learned using our approach.  The individual atoms in each case generally correspond to piecewise smooth regions separated by boundaries.  However, the geometry of the \emph{collection} of atoms is distinctly different in the two cases; in particular, the atoms learned using our approach better represent the transformations underlying natural images.  As we discuss in the next set of experiments, our framework provides regularizers that lead to improved denoising performance on natural images in comparison with polyhedral regularizers.

\begin{figure}
	\centering
	\begin{minipage}{.5\textwidth}
		\centering
		\includegraphics[width=0.5\textwidth]{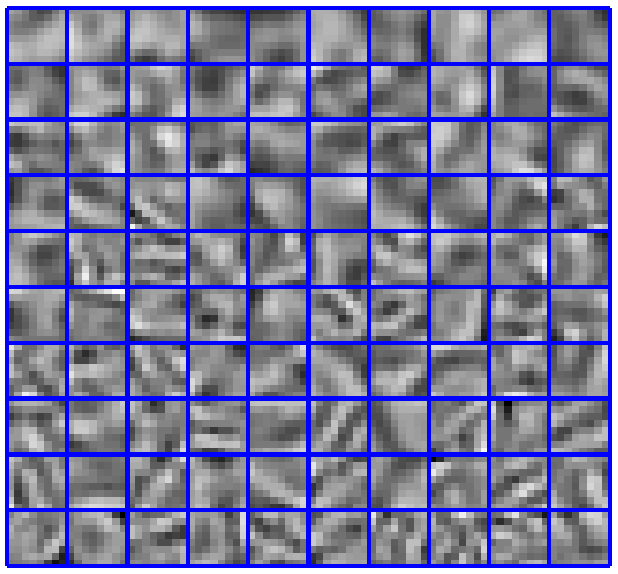}
		\label{fig:polyatoms}
	\end{minipage}%
	\begin{minipage}{.5\textwidth}
		\centering
		\includegraphics[width=0.54\textwidth]{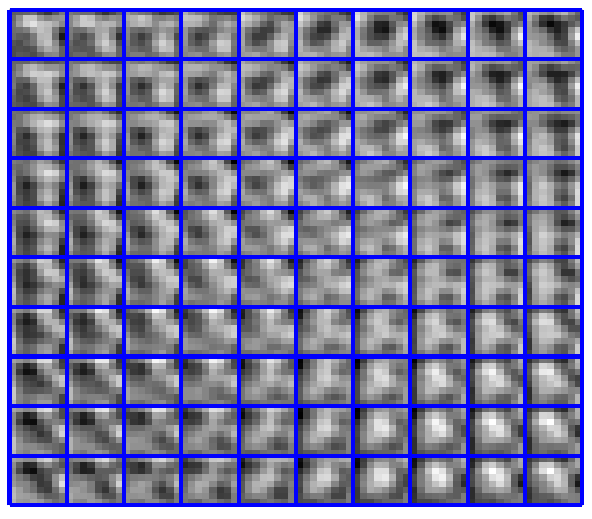}
		\label{fig:sdpatoms}
	\end{minipage}
	\caption{Comparison between atoms learned from dictionary learning (left) and our algorithm (right).}
	\label{fig:atoms}
\end{figure}

\subsubsection{Denoising Natural Image Patches} \label{sec:numexp_raw_den}

We compare the performance of polyhedral and semidefinite regularizers in denoising natural image patches corrupted by noise.

\textbf{Setup.} The $6480$ data points from the previous experiment are designated as a training set.  Here we consider an additional collection $\{\by^{(j)}_{\mathrm{test}}\}_{j=1}^{720} \subset \R^{64}$ of $8 \times 8$ test image patches obtained from larger seagull images (as with the training set), and subsequently shifted by an average of the pre-centered training set.  We corrupt each of these test points by i.i.d. Gaussian noise to obtain $\by^{(j)}_{\mathrm{obs}} = \by^{(j)}_{\mathrm{test}} + \bw^{(j)}, ~ j=1,\dots,720$, where each $\bw^{(j)} \sim \mathcal{N}(0, \sigma^2 I)$ with $\sigma^2$ chosen so that the average signal-to-noise ratio $\frac{1}{720} \sum_{j=1}^n \frac{\|\by^{(j)}_{\mathrm{test}}\|_{\ell_2}^2}{64 \sigma^2} \approx 18$.  Our objective is to investigate the denoising performance of the polyhedral and semidefinite regularizers (learned on the training set) on the data set $\{\by^{(j)}_{\mathrm{obs}}\}_{j=1}^{720}$.  Specifically, we analyze the following proximal denoising procedure:
\begin{equation} \label{eq:proxoperator}
\hat{\by}_{\mathrm{denoise}} = \underset{\by \in \R^{64}}{\argmin} ~~~ \tfrac{1}{2} \|\by_{\mathrm{obs}} - \by\|_{\ell_2}^2 + \lambda \|\by\|,
\end{equation}
where $\|\cdot\|$ is a regularizer learned on the training set and $\lambda > 0$ is a regularization parameter.

\textbf{Computational complexity of regularizer.} To compare the performances of different regularizers, it is instructive to consider the cost associated with employing a regularizer for denoising.  In particular, the regularizers learned on the training set have unit-balls that are specified as linear images of the nuclear norm ball and the $\ell_1$ ball.  Consequently, the main cost associated with employing a regularizer is the computational complexity of solving the corresponding proximal denoising problem \eqref{eq:proxoperator}.  Thus, we analyze the normalized mean-squared denoising error $\frac{1}{720} \sum_{j=1}^n \frac{\|\by^{(j)}_{\mathrm{obs}} - \by^{(j)}_{\mathrm{denoise}}\|_{\ell_2}^2}{64 \sigma^2}$ of a regularizer as a function of the computational complexity of solving \eqref{eq:proxoperator}.  For a polyhedral norm $\|\cdot\| : \R^d \rightarrow \R$ with unit ball specified as the image under a linear map $L : \R^p \rightarrow \R^d$ of the $\ell_1$ ball in $\R^p$, we solve \eqref{eq:proxoperator} as follows by representing the norm $\|\cdot\|$ in a lifted manner:
\begin{equation} \label{eq:prox_lp}
\begin{aligned}
\hat{\by}_{\mathrm{denoise}} = ~ &  \underset{\substack{\bx, \bz \in \R^p \\ s, t \in \R}}{\argmin}  ~~~ \tfrac{1}{2} s + \lambda t \\
~~~ & ~~~~ \mathrm{s.t.} ~~~ \|\by_{\mathrm{obs}} - L \bx\|_{\ell_2}^2 \leq s , ~~~ \sum_{i=1}^p \bz_i \leq t , ~~~ \begin{pmatrix}\bz - \bx \\ \bz + \bx \end{pmatrix} \geq 0.
\end{aligned}
\end{equation}
To solve \eqref{eq:prox_lp} to an accuracy $\epsilon$ using an interior-point method with the usual logarithmic barriers for the nonnegative orthant and the second-order cone, we have that the number of operations required is $\sqrt{2p+2} \log (\frac{2p+2}{\epsilon \eta}((d+2p+2)^3 + (2p+2)^3) )$ -- this represents the number of outer loop iterations of the interior point method -- multiplied by $(d+2q+2)^3 + (2q+2)^3$ -- this represents the number of operations required to solve the associated linear system in the inner loop -- for a barrier parameter $\eta$ \cite{NesNem:94,Ren:01}.
In a similar manner, for a semidefinite regularizer $\|\cdot\| : \R^d \rightarrow \R$ with unit ball specified as the image under a linear map $\L : \R^{q \times q} \rightarrow \R^d$ of the nuclear norm ball in $\R^{q \times q}$, we again solve \eqref{eq:proxoperator} as follows by representing the norm $\|\cdot\|$ in an analogous lifted manner:
\begin{equation} \label{eq:prox_sdp}
\begin{aligned}
\hat{\by}_{\mathrm{denoise}} = ~ & \underset{\substack{X \in \R^{q \times q} \\ Z_1, Z_2 \in \mathbb{S}^q \\ s, t \in \R}}{\argmin} ~~~ \tfrac{1}{2} s + \lambda t \\
& ~~~~ \mathrm{s.t.} ~~~ \|\by_{\mathrm{obs}} - \L(X)\|_{\ell_2}^2 \leq s, ~ \tfrac{1}{2}\mathrm{trace}(Z_1+Z_2) \leq t, ~ \begin{pmatrix}Z_1 & X \\ X' & Z_2 \end{pmatrix} \succeq 0.
\end{aligned}
\end{equation}
As before, to solve \eqref{eq:prox_sdp} to an accuracy $\epsilon$ using an interior-point method with the usual logarithmic barriers for the positive-semidefinite cone and the second-order cone, we have that the number of operations required is $\sqrt{2q+2} \log (\frac{2q+2}{\epsilon \eta}((d+2{q \choose 2}+2)^3 + (2{q \choose 2}+2)^3) )$ multiplied by $(d + {2q \choose 2} + 2 )^3 + ({2q \choose 2} + 2)^3$ for a barrier parameter $\eta$ \cite{Ren:01}.

\begin{figure}
	\centering
	\begin{minipage}{.5\textwidth}
		\centering
		\includegraphics[width=0.7\textwidth]{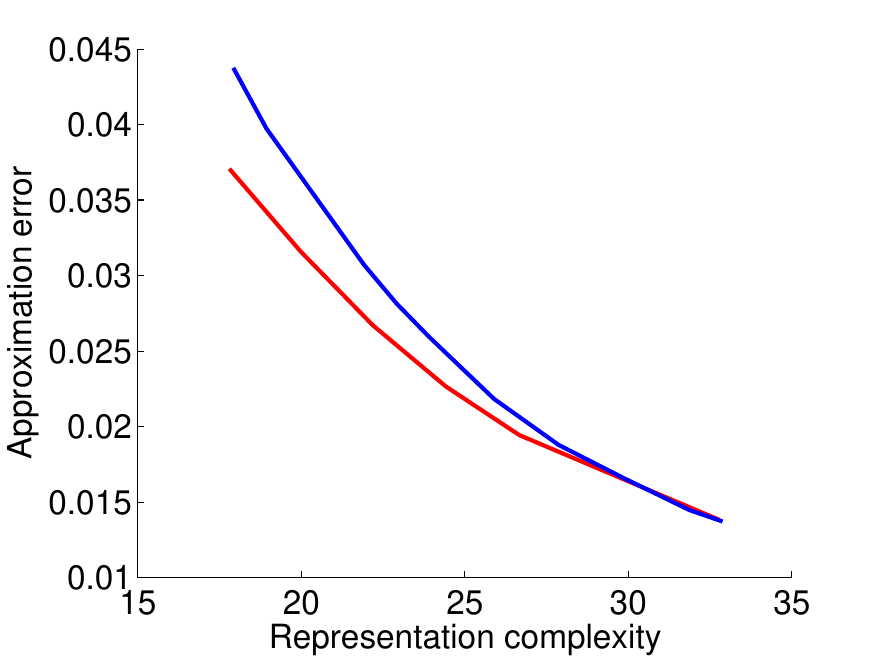}
		\label{fig:compresscompare}
	\end{minipage}%
	\begin{minipage}{.5\textwidth}
		\centering
		\includegraphics[width=0.7\textwidth]{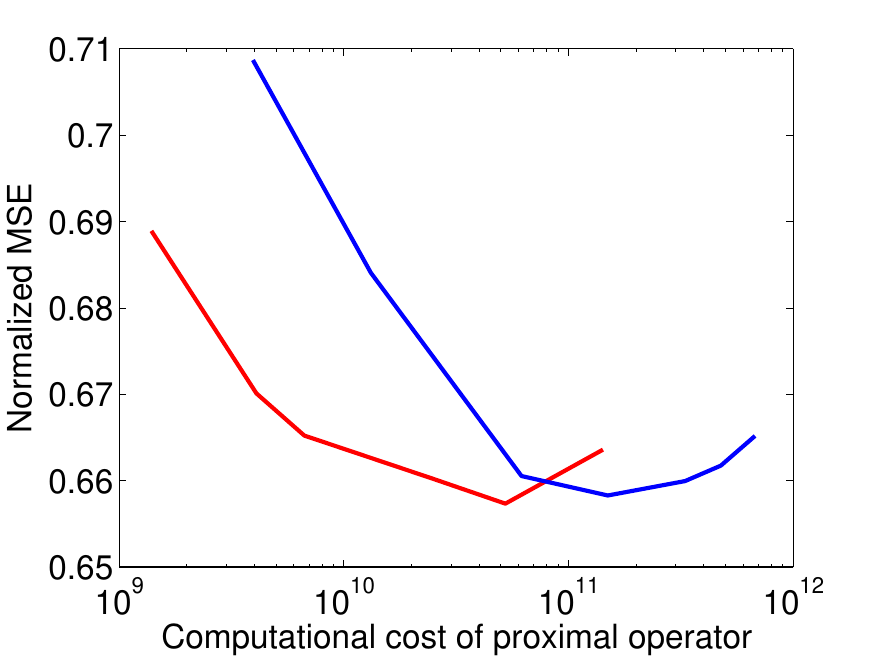}
		\label{fig:denoisingfixrank}
	\end{minipage}
	\caption{Comparison between dictionary learning (blue) and our approach (red) in representing natural image patches (left); comparison between polyhedral (blue) and semidefinite (right) regularizers in denoising natural image patches (right).}
	\label{fig:representanddenoise}
\end{figure}

\textbf{Results.} We learn semidefinite regularizers on the training set using Algorithm \ref{alg:amsdpreg} for $q \in \{9, \dots,20\}$ and for a rank of $1$.  We also learn polyhedral regularizers on the training set using dictionary learning for $p \in \{9^2, 10^2, \dots, 20^2\}$ and with corresponding sparsity levels in the range $\{\sqrt{p}-1, \sqrt{p}\}$ to ensure that the representation complexity matches the corresponding representation complexity of the images of rank-one matrices in the semidefinite case.  As the lifted dimensions $q^2$ and $p$ increase, the computational complexities of the associated proximal denoisers (with the learned regularizers) also increase.  The right subplot in Figure \ref{fig:representanddenoise} gives the average normalized mean-squared error over the noisy test data (generated as described above).  The optimal choice of the regularization parameter $\lambda$ for each regularizer is obtained by sweeping over a range to obtain the best denoising performance, as we have access to the underlying uncorrupted image patches $\{\by^{(j)}_{\mathrm{test}}\}_{j=1}^{720}$.  For both types of regularizers the denoising performance improves initially before degrading due to overfitting.  More significantly, given a fixed computational budget, these experiments suggest that semidefinite regularizers provide better performance than polyhedral regularizers in denoising image patches in our data set.  The denoising operation \eqref{eq:proxoperator} is in fact a basic computational building block (often referred to as a proximal operator) in first-order algorithms for solving convex programs that arise in a range of inverse problems \cite{PB:14}.  As such, we expect the results of this section to be qualitatively indicative of the utility of our approach in other inferential tasks beyond denoising.

\section{Discussion} \label{sec:discussion}
Our paper describes an algorithmic framework for learning regularizers from data in settings in which prior domain-specific expertise is not directly available.  We learn these regularizers by computing a structured factorization of the data matrix, which is accomplished by combining techniques for the affine rank minimization problem with the Operator Sinkhorn scaling procedure.  The regularizers obtained using our method are convex and they can be computed via semidefinite programming.  Our approach may be viewed as a semidefinite analog of dictionary learning, which can be interpreted as a technique for learning polyhedral regularizers from data.  We discuss next some directions for future work.

\subsection{Algorithmic questions} It would be of interest to better understand the question of initialization for our algorithm.  Random initialization often works well in practice and it would be useful to provide theoretical support for this approach by building on recent work on other factorization problems \cite{GLM:16,SQW:17}.  To this end, we describe two experimental setups on synthetic data showing instances where our algorithm recovers the true regularizer from random initialization.  In the first setup we generate a standard Gaussian linear map $\L : \R^{8 \times 8} \rightarrow \R^{50}$ and normalize it.  Let $\L^\star$ denote the resulting normalized map.  We generate data $\{\by^{(j)}\}_{j=1}^{10^4}$ as $\by^{(j)} = \L^\star(\bu^{(j)} {\bv^{(j)}}') / \|\L^\star(\bu^{(j)} {\bv^{(j)}}')\|_{\ell_2}$, where each $\bu^{(j)}, \bv^{(j)}$ is drawn independently from the Haar measure on the unit sphere in $\R^{8}$.  We apply our algorithm to the data, and we supply as initialization the normalization of a standard Gaussian linear map.  The left subplot of Figure \ref{fig:globalrecovery} shows the progression of the mean-squared error over $10$ different initializations.  As the measurements do not contain any additional noise, the minimum attainable error is zero.  We observe that our algorithm recovers the regularizer in all $10$ random initializations; moreover, we observe local, linear convergence in the neighborhood of the global minimizer, which agrees with our analysis.  Note that the progress of our algorithm reveals interesting behavior in that the global recovery of the regularizer is characterized by three distinct phases -- $(i)$ an initial phase in which progress is significant; $(ii)$ an intermediate phase in which progress is incremental but stable; and (iii) a terminal phase that corresponds to local, linear convergence.  In particular, these graphs indicate that global convergence to the underlying regularizer is \emph{not} linear.  The second setup is similar to the first one, with the two main differences being that we consider a linear map $\L^\star : \R^{8 \times 8} \rightarrow \R^{60}$ of slightly different dimensions, and that the data points $\{\by^{(j)}\}_{j=1}^{2\times 10^4}$ are images of rank-two matrices.  The right subplot of Figure \ref{fig:globalrecovery} shows the progression of our algorithm over $10$ different initializations.  In contrast to the previous setup where every initialization led to a global minimum, in this case our algorithm attains a local minimum in $4$ out of $10$ initializations and a global minimum in the remaining $6$ initializations.  In summary, our experiments suggest that random initialization may sometimes be effective, and understanding this effectiveness warrants further investigation.

\begin{figure}
	\centering
	\begin{minipage}{.5\textwidth}
		\centering
		\includegraphics[width=0.8\textwidth]{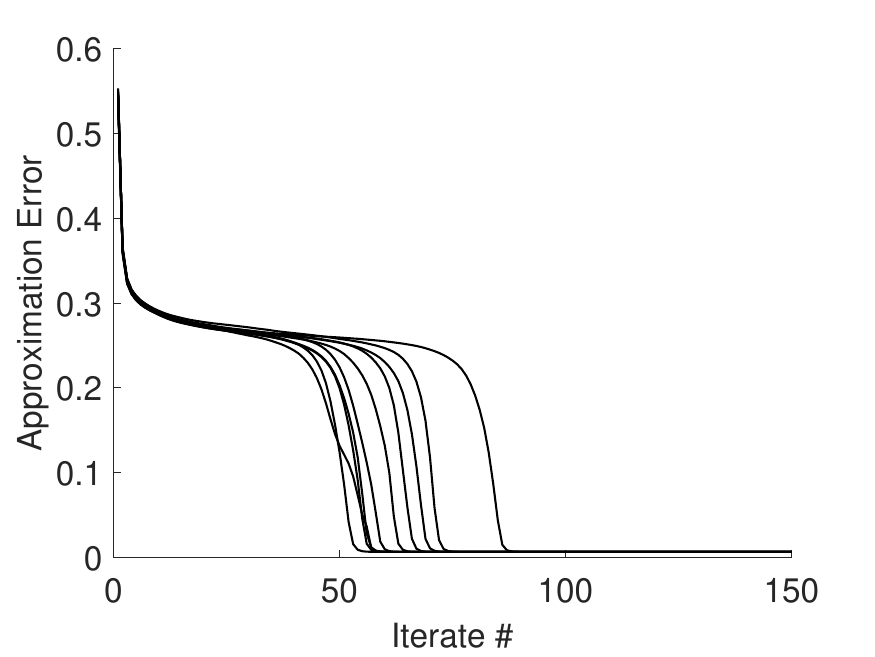}
	\end{minipage}%
	\begin{minipage}{.5\textwidth}
		\centering
		\includegraphics[width=0.8\textwidth]{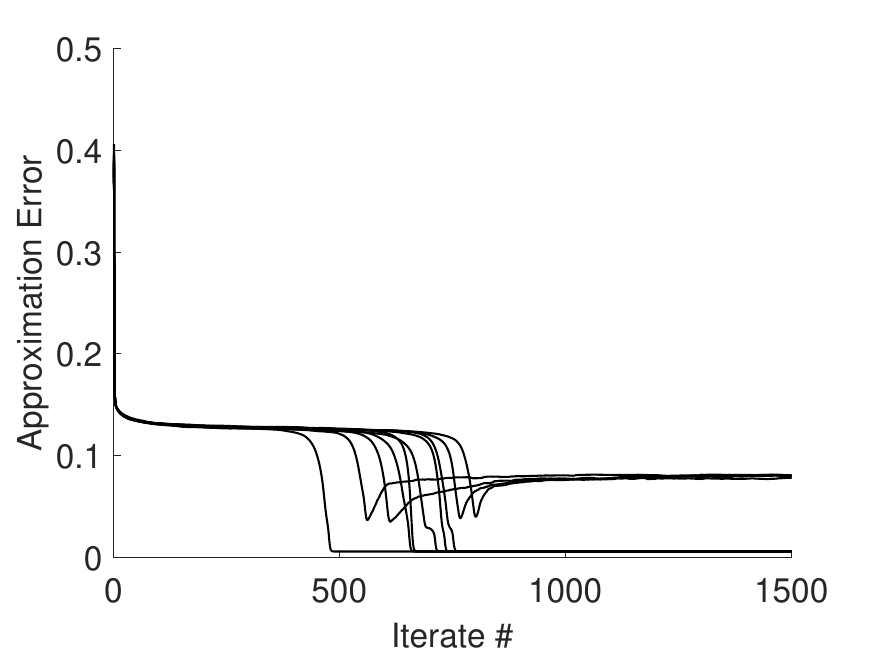}
	\end{minipage}
	\caption{Progression of our algorithm in recovering regularizers in a synthetic experimental set-up; the horizontal axis represents the number of iterations, and each line corresponds to a different random initialization.  The left subplot shows a problem instance in which all $10$ different random initializations recover a global minimizer, while the right subplot shows a different problem instance in which $4$ out of $10$ random initializations lead to local minima.}
	\label{fig:globalrecovery}
\end{figure}

Beyond random initialization, there have also been efforts on data-driven strategies for initialization in dictionary learning by reducing the question to a type of clustering / community detection problem \cite{AAN:17,AGM:14}.  While the relation between clustering and estimating the elements of a finite atomic set is conceptually natural, identifying an analog of the clustering problem for estimating the image of a variety of rank-one matrices (which is a structured but infinite atomic set) is less clear; we seek such a conceptual link in order to develop an initialization strategy for our algorithm.  In a completely different direction, there is also recent work on a convex relaxation for the dictionary learning problem that avoids the difficulties associated with local minima \cite{BKS:15}; while this technique is considerably more expensive computationally in comparison with alternating updates, developing analogous convex relaxation approaches for the problem of learning semidefinite regularizers may subsequently point the way to efficient global techniques that are different from alternating updates.

\begin{figure}
	\centering
	\begin{minipage}{.5\textwidth}
		\centering
		\includegraphics[width=0.7\textwidth]{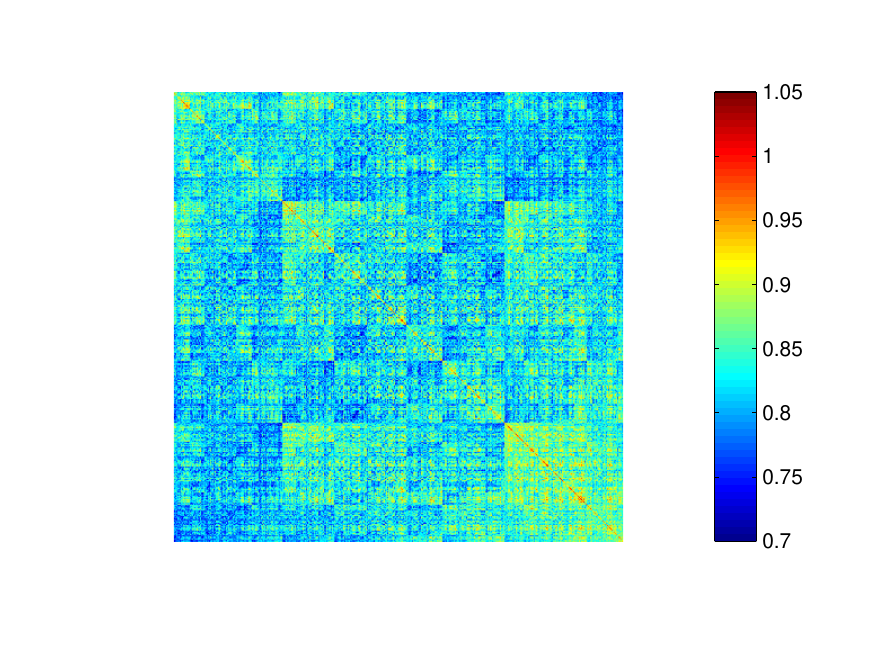}
		\label{fig:sparsegram}
	\end{minipage}%
	\begin{minipage}{.5\textwidth}
		\centering
		\includegraphics[width=0.7\textwidth]{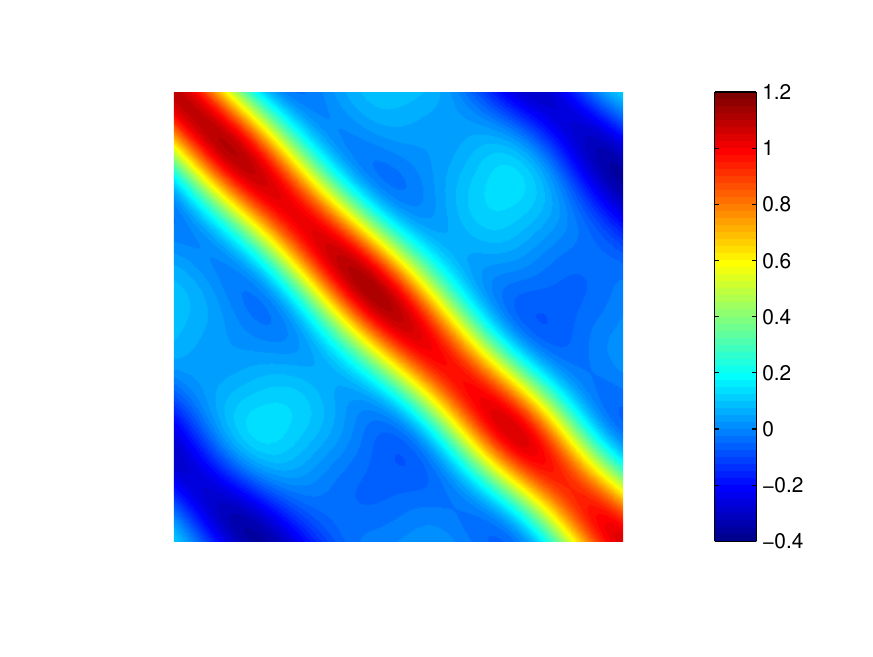}
		\label{fig:lowrankgram}
	\end{minipage}
	\caption{Gram matrices of images of sparse vectors (left) and low-rank matrices (right).}
	\label{fig:gram}
\end{figure}

\subsection{Approximation-theoretic questions}

The focus of our paper has been on the algorithmic aspects of learning semidefinite regularizers from data.  It is of interest to investigate the power of finite atomic sets in comparison with atomic sets specified as projections of determinantal varieties from a harmonic analysis perspective (for a fixed representation complexity; see Section \ref{sec:numexp_raw_rep} for a discussion on how these are defined).  For example, what types of data are better described using one representation framework versus the other?  As a simple preliminary illustration, we generate two sets of $400$ points in $\R^{500}$, with the first set being a random projection of sparse vectors in $\R^{900}$ and the second set being a random projection of rank-one matrices in $\R^{900}$ of the form $( \cdots ~ \cos(2 \pi \alpha_j t_i), ~ \sin(2 \pi \alpha_j t_i), ~ \cdots)' ~ ( \cdots ~ \cos(2 \pi \beta_j t_i), ~ \sin(2 \pi \beta_j t_i), ~ \cdots)$ for randomly chosen frequencies $\alpha_j, \beta_j$; the representation complexities of both these sets is the same.  Figure \ref{fig:gram} gives the Gram matrices associated with these data sets.  The data set of projections of sparse vectors appears to consist of `clusters' of `block' structure, while the data set of projections of low-rank matrices appears to consist of smoother `toroidal' structure.  We seek a better understanding of this phenomenon by analyzing the relative strengths of representations based on finite atomic sets versus projections of low-rank matrices.  In a different direction, it is also of interest to explore other families of infinite atomic sets that yield tractable regularizers in other conic programming frameworks.  Specifically, dictionary learning and our approach provide linear and semidefinite programming regularizers, but there are other families of computationally efficient convex cones such as the power cone and the exponential cone; learning atomic sets that are amenable to optimization in these frameworks would lead to a broader suite of data-driven approaches for identifying regularizers.

%% file: appendix_riptrivials.tex
\section{Proofs of Lemma \ref{thm:weakincoherencebound} and Lemma \ref{thm:boundondelta}} \label{apx:prelims}

\begin{proof}[Lemma \ref{thm:weakincoherencebound}]
	Note that if $Z \in \ct$ then $Z$ has rank at most $2r$. As a consequence of the restricted isometry property we have $(1-\delta_{2r}) \|Z \|_{\fro}^2 \leq \| [\L \circ \cp_{\ct}] (Z)\|_{\ell_2}^2 \leq (1+\delta_{2r}) \|Z \|_{\fro}^2$. Since $Z \in \ct$ is arbitrary we have $1 - \delta_{2r} \leq \lambda(\L_{\ct}^{\prime} \L_{\ct}) \leq 1+\delta_{2r}$, which proves (i).  This immediately implies the bound in (ii).  Moreover since $\|\L \circ \cp_{\ct}\|_{\spec} = \|\cp_{\ct} \circ \L^{\prime} \L \circ \cp_{\ct} \|_{\spec}^{1/2} \leq \sqrt{1+\delta_{2r}}$, we have $\|\cp_{\ct} \circ \L^{\prime} \L\|_{\spec} \leq \sqrt{1+\delta_{2r}}\|\L\|_{\spec} $, which is (iii). Last we have $\| [(\L_{\ct}^{\prime} \L_{\ct})^{-1}]_{\R^{q \times q}} \circ \L^{\prime} \L\|_{\spec} \leq \| [(\L_{\ct}^{\prime} \L_{\ct})^{-1}]_{\R^{q \times q}} \|_{\spec} \| \cp_{\ct} \circ \L^{\prime} \L\|_{\spec} \leq \frac{\sqrt{1+\delta_{2r}}}{1-\delta_{2r}}\|\L\|_{\spec}$, which proves (iv). \qed
\end{proof}

\begin{proof}[Lemma \ref{thm:boundondelta}]
	To simplify notation we omit $(\ensa)$.  Since $\mathrm{trace}(\mathsf{\Sigma}) = \frac{1}{n}\sum_{j=1}^{n}\|X^{(j)}\|_{\fro}^2$, we have $s_{\min} \leq \mathrm{trace}(\mathsf{\Sigma}) \leq s_{\max}$.  Next we have the inequalities $(\coveig  - \covsup ) \sfi \preceq \mathsf{\Sigma} \preceq (\coveig + \covsup ) \sfi$.  The result follows by applying trace. \qed
\end{proof}

%% file: appendix_randesmb.tex
\section{Proof of Proposition \ref{thm:randensemble}} \label{apx:randensemble}

In this section we prove that the ensemble of random matrices $\ensa$ described in Proposition \ref{thm:randensemble} satisfy the deterministic conditions in Theorem \ref{thm:localconvergence} with high probability.  We begin with computing $\mathbb{E}_{\mathcal{D}} [X^{(j)} \boxtimes X^{(j)}]$, and $\mathbb{E}_{\mathcal{D}} [(X^{(j)} \boxtimes X^{(j)}) \botimes \cp_{\ct(X^{(j)})}]$.  Note that the random matrices $\{X^{(j)} \boxtimes X^{(j)}\}_{j=1}^{n}$ and the random operators $\{(X^{(j)} \boxtimes X^{(j)}) \botimes \cp_{\ct(X^{(j)})}\}_{j=1}^{n}$ are almost surely bounded above in spectral norm by construction.  This allows us to conclude Proposition \ref{thm:randensemble} with an application of the Matrix Hoeffding Inequality \cite{Tro:12}.

To simplify notation we adopt the following.  In the first two results we omit the superscript $j$ from $X^{(j)}$.  In the remainder of the section we let $\mathbb{E}=\mathbb{E}_{\mathcal{D}}$, $\bar{s}^2 := \mathbb{E} [s^2]$, $\{\be_i\}_{i=1}^{q} \subset \R^q$ be the set of standard basis vectors, and $\{E_{ij}\}_{i,j=1}^{q} \subset \mathbb{R}^{q\times q}$ be the set of matrices whose $(i,j)$-th entry is $1$ and is $0$ everywhere else.

\begin{proposition}\label{thm:xxmean}
	Suppose $X\sim\mathcal{D}$ as described in Proposition \ref{thm:randensemble}.  Then $\mathbb{E}[X \boxtimes X] = \bar{s}^2 (r/q^2) \sfi$.
\end{proposition}

\begin{proof}
	It suffices to show that $\mathbb{E} \langle X \boxtimes X,\be_w \be_x^{\prime} \boxtimes \be_y \be_z^{\prime} \rangle = \mathbb{E} \langle X, \be_w \be_x^{\prime}\rangle \langle X, \be_y \be_z^{\prime}\rangle = \delta_{wy} \delta_{xz} \bar{s}^2 (r/q^2)$. Let $X = \sum_{i=1}^{r} s_{i} \bu_{i} \bv_{i}^{\prime}$ as described in the statement of Proposition \ref{thm:randensemble}. Suppose we denote $\bu_{i} = (u_{i1},\ldots,u_{iq})^{\prime}$, and $\bv_{i} = (v_{i1},\ldots,v_{iq})^{\prime}$. By applying independence we have $\mathbb{E} \langle X, \be_w \be_x^{\prime}\rangle \allowbreak \langle X, \be_y \be_z^{\prime}\rangle = \mathbb{E}[(\sum_{i=1}^{r} s_i u_{iw} v_{ix})(\sum_{k=1}^{r} s_k u_{ky} v_{kz})] = \sum_{i,k=1}^{r} \mathbb{E}[s_{i}s_{k}] \mathbb{E}[u_{iw}u_{ky}] \mathbb{E}[v_{ix}v_{kz}]$.  There are two cases we need to consider.
	
	\noindent [Case $w \neq y$ or $x \neq z$]: Without loss of generality suppose that $w \neq y$. Then $\mathbb{E}[u_{iw}u_{ky}]=0$ for all $1\leq i,k \leq q$, and hence $\mathbb{E} \langle X \boxtimes X,\be_w \be_x^{\prime} \boxtimes \be_y \be_z^{\prime} \rangle =0$.
	
	\noindent [Case $w=y$ and $x=z$]: Note that if $i \neq k$ then $\mathbb{E}[u_{iw}u_{ky}]=\mathbb{E}[u_{iw}] \mathbb{E}[u_{ky}]=0$.  Since $\bu_i$ is a unit-norm vector distributed u.a.r., we have $\mathbb{E}[u_{ix}^2] = 1/q$. Hence $\mathbb{E} \langle X \boxtimes X,\be_w \be_x^{\prime} \boxtimes \be_y \be_z^{\prime} \rangle = \sum_{i=1}^{r} \mathbb{E}[s_{i}^2] \mathbb{E}[u_{iw}^2] \mathbb{E}[v_{ix}^2] = \bar{s}^2 r /q^2$. \qed
\end{proof}

Our next result requires the definition of certain subspaces of $\R^{q\times q}$ and $\aut(\R^{q\times q})$.

We define the following subspaces in $\mathbb{R}^{q\times q}$: Let $\ssy:=\{W : W = W^{\prime}, W \in I^{\perp} \}$ be the subspace of symmetric matrices that are orthogonal to the identity, $\sas:=\{W : W = -W^{\prime}\}$ be the subspace of skew-symmetric matrices, and $\spi = \mathrm{Span}(I)$. It is clear that $\mathbb{R}^{q\times q} = \ssy \oplus \sas \oplus \spi$.

In addition to the subspace $\W$ defined in \eqref{eq:tangentspaceati}, we define the following subspaces in $\aut(\R^{q\times q})$:
\begin{enumerate}
	\item $\W_{SS} := \mathrm{Span}(\{A \botimes B : A,B \in \ssy \})$,
	\item $\W_{AA} := \mathrm{Span}(\{A \botimes B : A,B \in \sas \})$,
	\item $\W_{SA} := \mathrm{Span}(\{A \botimes B : A \in \ssy, B \in \sas \})$,
	\item $\W_{AS} := \mathrm{Span}(\{A \botimes B : A \in \sas, B \in \ssy \})$.
\end{enumerate}
Note that $\aut(\R^{q\times q}) = \W \oplus \W_{SS} \oplus \W_{AA} \oplus \W_{SA} \oplus \W_{AS}$.  To verify this, first express an arbitrary linear map $\sfe \in \aut(\R^{q\times q})$ as a sum of Kronecker products $\sfe = \sum_{i=1} A_i \botimes B_i$, second decompose each matrix $A_i,B_i$ into components in the subspaces $\{ \ssy, \sas, \spi \}$, and third expand the expression.  The orthogonality between subspaces is immediate from the identity $\langle A_i \botimes B_i , A_j \botimes B_j \rangle  = \langle A_i , A_j \rangle \langle B_i, B_j \rangle$.

\begin{proposition}\label{thm:expectationofgiantoperator}
	Suppose $X\sim\mathcal{D}$ as described in Proposition \ref{thm:randensemble}.  Then
	\begin{equation*}
	\mathbb{E} [(X \boxtimes X) \botimes \cp_{\ct(X)}] = c_{\W} \sfi_{\W} + c_{\W_{SS}} \sfi_{\W_{SS}} + c_{\W_{AA}} \sfi_{\W_{AA}} + c_{\W_{SA}} \sfi_{\W_{SA}} + c_{\W_{AS}} \sfi_{\W_{AS}},
	\end{equation*} 
	where (i) $c_{\W} = \bar{s}^2 r(\frac{1}{q^2})$, (ii) $c_{\W_{SS}} = \bar{s}^2 r(\frac{1}{q^2} - \frac{(q-r)^2}{(q-1)^2(q+2)^2}) $, (iii) $c_{\W_{AA}} = \bar{s}^2 r(\frac{1}{q^2}- \frac{(q-r)^2}{q^2 (q-1)^2} )$,	and (iv) $c_{\W_{SA}} = c_{\W_{AS}} = \bar{s}^2 r(\frac{1}{q^2} - \frac{(q-r)^2}{q(q-1)^2(q+2)})$.
\end{proposition}

\begin{proof}
	The proof consists of two parts, namely (i) to prove that the mean, when restricted to the respective subspaces described above, has diagonal entries as specified, and (ii) to prove that the off-diagonal elements are zero with respect to any basis that obeys the specified decomposition of $\aut(\R^{q\times q})$.  In addition, it suffices to only consider linear maps that are Kronecker products since these maps generate the respective subspaces.  The following identity for all matrices $A_i,B_i,A_j,B_j$ is particularly useful
	\begin{equation}\label{eq:giantopinnp}
	\langle (A^{\prime}_i \botimes B_i) \boxtimes (A_j^{\prime} \botimes B_j) ,  \mathbb{E} [(X \boxtimes X) \botimes \cp_{\ct(X)}]\rangle = \mathbb{E}  \langle \cp_{\ct(X)} (B_j X A_j), \cp_{\ct(X)}(B_i X A_i) \rangle .
	\end{equation}
	One may equivalently describe the distribution of $X$ as follows -- let $X = U \Sigma_R V^{\prime}$, where $U,V$ are $q\times q$ matrices drawn from the Haar measure, and $\Sigma_R$ is a diagonal matrix whose first $r$ entries are drawn from $\mathcal{D}$, and the remaining entries are $0$ (to simplify notation we omit the dependence on $X$ in the matrices $U,V$).  Let $I_N = \mathrm{diag}(0,\ldots, 0, 1,\ldots ,1)$ be a diagonal matrix consisting of $q-r$ ones. Under this notation, the projector is simply the map $\cp_{\ct(X)} (Z) = Z - U I_N U^{\prime}ZV I_N V^{\prime} $.  The remainder of the proof is divided into the two parts outlined above.
	
	\noindent [Part (i)]: The restriction to diagonal entries correspond to the case $i=j$, and hence equation \eqref{eq:giantopinnp} simplifies to $\mathbb{E}  [\| \cp_{\ct(X)} (B X A) \|^2_{\fro} ]$.  Consequently we have
	\begin{equation*}
	\mathbb{E}  [\| \cp_{\ct(X)} (B X A) \|^2_{\fro} ] = \mathbb{E} [\|B U \Sigma_R V^{\prime} A \|^2_{\fro} ] - \mathbb{E} [ \| I_N U^{\prime} A U \Sigma_R V^{\prime} B V I_N \|_{\fro}^2  ].
	\end{equation*}
	
	First we compute $\mathbb{E} [ \| I_N U^{\prime} A U \Sigma_R V^{\prime} B V I_N \|_{\fro}^2  ]$. By the cyclicity of trace and iterated expectations we have
	\begin{eqnarray*}
		\mathbb{E} [ \| I_N U^{\prime} A U \Sigma_R V^{\prime} B V I_N \|_{\fro}^2  ] & = & \mathbb{E}[ \mathrm{trace} ( \Sigma_R^{1/2} U^{\prime} A^{\prime} U I_N U^{\prime} A U \Sigma_R V^{\prime} B V I_N  V^{\prime} B^{\prime} V \Sigma_R^{1/2})  ] \\
		& = & \mathbb{E}_{U}[ \mathbb{E}_{V} [\mathrm{trace} ( \Sigma_R^{1/2} U^{\prime} A^{\prime} U I_N U^{\prime} A U \Sigma_R V^{\prime} B V I_N  V^{\prime} B^{\prime} V \Sigma_R^{1/2})]].
	\end{eqnarray*}
	It suffices to compute $\mathbb{E} [\Sigma_R^{1/2} V^{\prime} B V I_N  V^{\prime} B^{\prime} V \Sigma_R^{1/2}]=\Sigma_R^{1/2} \mathbb{E} [ V^{\prime} B V I_N  V^{\prime} B^{\prime} V] \Sigma_R^{1/2}$ in the three cases corresponding to $B \in \{ \ssy, \sas, \spi \}$ respectively.  Using linearity and symmetry, it suffices to compute $\mathbb{E} [V^{\prime} B V E_{11}  V^{\prime} B^{\prime} V]$.  We split this computation into the following three separate cases.
	
	[Case $B \in \spi$]:
	We have $I_N \Sigma_R^{1/2} = 0$, and hence the mean is the zero-matrix.  
	
	[Case $B \in \sas$]:	
	Claim: If $B \in \sas$, and $\|B\|_{\fro} = 1$, then $\mathbb{E} [V^{\prime} B V E_{11}  V^{\prime} B^{\prime} V] = (I- E_{11})/(q(q-1))$.
	
	Proof:  Denote $V=[\bv_1 | \ldots | \bv_q]$.  The off-diagonal entries vanish as $\mathbb{E} \langle E_{ij}, V^{\prime} B V E_{11}  V^{\prime} B^{\prime} V\rangle \allowbreak = \mathbb{E} (\bv_1^{\prime} B \bv_i)(\bv_1^{\prime} B \bv_j) = 0$ whenever $i \neq j$, as one of the indices $i,j$ appears exactly once. By a symmetry argument we have $\mathbb{E} [V^{\prime} B V E_{11}  V^{\prime} B^{\prime} V]  = \alpha I + \beta E_{11}$ for some $\alpha,\beta$. First $\mathbb{E} [ \mathrm{trace}(V^{\prime} B V E_{11}  V^{\prime} B^{\prime} V )] = \mathbb{E} [ \mathrm{trace}(B V E_{11} V^{\prime} B^{\prime})] = \mathrm{trace}(B \mathbb{E} [ V E_{11} V^{\prime}] B^{\prime}) = \mathrm{trace} (B (I/q) \allowbreak B^{\prime}) = 1/q$, which gives $\alpha q + \beta = 1/q$. Second since $B$ is asymmetric, $V^{\prime} B V$ is also asymmetric and hence is $0$ on the diagonals. Thus $\langle V^{\prime} B V E_{11}  V^{\prime} B^{\prime} V, E_{11} \rangle = 0$, which gives $\alpha + \beta = 0$. The two equations yield the values of $\alpha$ and $\beta$.
	
	[Case: $B \in \ssy$]:
	Claim: If $B \in \ssy$, and $\|B\|_{\fro} = 1$, then $\mathbb{E} [V^{\prime} B V E_{11}  V^{\prime} B^{\prime} V] = (I + (1-2/q) E_{11})/((q-1)(q+2))$.		
	
	Proof: With an identical argument as the previous claim one has $\mathbb{E} [V^{\prime} B V E_{11}  V^{\prime} B^{\prime} V] = \alpha I + \beta E_{11}$, where $\alpha q + \beta = 1/q$. Next $\mathbb{E} [ \langle V^{\prime} B V E_{11}  V^{\prime} B^{\prime} V, E_{11} \rangle ] = \mathbb{E} [(\bv_1^{\prime} B \bv_1)^2]$, where $\bv_1$ is a unit-norm vector distributed u.a.r. Since conjugation by orthogonal matrices preserves trace, and $\bv_1$ has the same distribution as $Q\bv_1$ for any orthogonal $Q$, we may assume that $B = \mathrm{diag}(b_{11},\ldots, b_{qq})$ is diagonal without loss of generality. Suppose we let $\bv_1 = (v_1,\ldots,v_q)^{\prime}$. Then $\mathbb{E} [(\bv_1^{\prime} B \bv_1)^2] = \mathbb{E} [ \sum b^2_{ii} v_{i}^4 + \sum_{i\neq j} b_{ii} b_{jj} v_{i}^2 v_{j}^2 ] = \mu_1 (\sum b_{ii}^2) + \mu_2 (\sum_{i \neq j} b_{ii} b_{jj} ) $, where $\mu_1 = \mathbb{E} [v_1^4]$, and $\mu_2 = \mathbb{E} [v_1^2 v_2^2]$. Since $\mathrm{trace}(B) = 0$, we have $\sum b_{ii}^2 = - \sum_{i \neq j} b_{ii} b_{jj}$. Last from Theorem 2 of \cite{Cho:13} we have $\mu_1 = 3 / (q (q+2))$, and $\mu_2 = 1/(q(q+2))$, which gives $\mathbb{E} [(\bv_1^{\prime} B \bv_1)^2] = 2/ (q(q+2))$, and hence $\alpha + \beta = 2/(q(q+2))$. The two equations yield the values of $\alpha$ and $\beta$.
	
	With a similar set of computations one can show that $\mathbb{E} [\|B U \Sigma_R V^{\prime} A \|_{\fro}^2 ] = \bar{s}^2 r/q^2$ for arbitrary unit-norm $A,B$.  An additional set of computations yields the diagonal entries, which completes the proof.  We omit these computations.
	
	\noindent [Part (ii)]: We claim that it suffices to show that $\mathbb{E} [ V^{\prime} A_i V E_{11} V^{\prime} A_j^{\prime} V] $ is the zero-matrix whenever $A_i,A_j \in \{ \ssy, \sas, \spi \}$, and satisfy $\langle A_i, A_j \rangle = 0$.  We show how this proves the result.  Suppose $A_i \botimes B_i, A_j \botimes B_j$ satisfy $\langle A_i \botimes B_i , A_j \botimes B_j \rangle = \langle A_i, A_j \rangle \langle B_i, B_j \rangle = 0$.  Without loss of generality we may assume that $\langle A_i, A_j \rangle = 0$.  From equation \eqref{eq:giantopinnp} we have
	\begin{eqnarray*}
		\mathbb{E}  \langle \cp_{\ct(X)} (B_j X A_j), \cp_{\ct(X)}(B_i X A_i) \rangle & = & \mathbb{E} [ \mathrm{trace}( A_j^{\prime} V \Sigma_R U^{\prime} B_j^{\prime} B_i U \Sigma_R V^{\prime} A_i ) ] \\
		& - & \mathbb{E}[ \mathrm{trace} ( A^{\prime}_j V \Sigma_R U^{\prime} B_j^{\prime} U I_N U^{\prime} B_i U \Sigma_R V^{\prime} A_i V I_N V^{\prime} )].
	\end{eqnarray*}
	By cyclicity of trace and iterated expectations we have
	\begin{eqnarray*}
		&& \mathbb{E}[ \mathrm{trace} ( A^{\prime}_j V \Sigma_R U^{\prime} B_j^{\prime} U I_N U^{\prime} B_i U \Sigma_R V^{\prime} A_i V I_N V^{\prime} )] \\
		& =& \mathbb{E}_{U}[ \mathrm{trace} (\Sigma_R^{1/2} U^{\prime} B_j^{\prime} U I_N U^{\prime} B_i U \Sigma_R^{1/2} (\mathbb{E}_{V} [ \Sigma_R^{1/2} V^{\prime} A_i V I_N V^{\prime} A^{\prime}_j V \Sigma_R^{1/2} ]))] = 0,
	\end{eqnarray*}
	which proves part (ii) of the proof.  It leaves to prove the claim.  We do so by verifying that the matrix $\mathbb{E} [ V^{\prime} A_i V E_{11} V^{\prime} A_j^{\prime} V] $ is $0$ in every coordinate, which is equivalent to showing that $\mathbb{E} (\bv^{\prime}_{m} A_i \bv_1) (\bv^{\prime}_{n} A_j \bv_1) =0$ for all $m,n$.  There are three cases.
	
	[Case $m \neq n$]: Without loss of generality suppose that $m \neq 1$.  Then $\mathbb{E} (\bv^{\prime}_{m} A_i \bv_1) (\bv^{\prime}_{n} A_j \bv_1) = \mathbb{E}[\mathbb{E} [ (\bv^{\prime}_{m} A_i \bv_1) (\bv^{\prime}_{n} A_j \bv_1)|\bv_1,\bv_n]] = 0$.
	
	[Case $m = n = 1$]: We divide into further sub-cases depending on the subspaces $A_i,A_j$ belong to. If $A_i \in \sas$ then $\bv_1^{\prime}A_i \bv_1 = 0$ since it is a scalar. Hence we eliminate the case where either matrix is in $\sas$. Since $\langle A_i,A_j \rangle = 0$ it cannot be that both $A_i,A_j \in \spi$. Suppose that $A_i = I / \sqrt{q}$ and $A_j \in \ssy$. Then $\mathbb{E} [ (\bv_1^{\prime}A_i \bv_1) (\bv_1^{\prime}A_j \bv_1)] = \mathbb{E} [ (\bv_1^{\prime}A_j \bv_1)]/\sqrt{q} = \mathbb{E} [ \mathrm{trace}(A_j \bv_1 \bv^{\prime}_1)]/\sqrt{q}=0$. Our remaining case is when $A_i,A_j \in \ssy$, and $\langle A_i,A_j \rangle = 0$. As before we let $\bv_1 = (v_1,\ldots,v_q)^{\prime}$. Then
	\begin{eqnarray*}
		& & \mathbb{E} [ (\bv^{\prime}_{1} A_i \bv_1) (\bv^{\prime}_1 A_j \bv_1)]  = \mathbb{E} [\sum_{pqrs} A_{i,pq} A_{j,rs} v_p v_q v_r v_s] \\
		& = & \sum_{p} A_{i,pp} A_{j,pp} \mathbb{E}[v_p^4] + \sum_{p\neq r} A_{i,pp} A_{j,rr} \mathbb{E}[v_p^2 v_r^2] + 2 \sum_{p\neq q} A_{i,pq} A_{j,pq}\mathbb{E}[v_p^2 v_q^2],
	\end{eqnarray*}
	where in the second equality we used the fact that $A_i,A_j$ are symmetric to obtain a factor of $2$ in the last term. Next we apply the relations $\mathbb{E}[v_p^4] = 3/(q(q+2))$, $ \mathbb{E}[v_p^2 v_r^2] = 1/(q(q+2))$, as well as the relations $0=\langle A_i, I \rangle \langle A_j, I\rangle=\sum_{p} A_{i,pp}A_{j,pp} + \sum_{p\neq r} A_{i,pp}A_{j,rr}$, and $0=\langle A_i, A_j \rangle = \sum_{p} A_{i,pp}A_{j,pp} + \sum_{p\neq q} A_{i,pq}A_{j,pq}$ to conclude that the mean is zero.
	
	[Case $m = n \neq 1$]: We have
	\begin{eqnarray*}
		\mathbb{E} [ (\bv^{\prime}_{m} A_i \bv_1) (\bv^{\prime}_m A_j \bv_1)] & = & \mathbb{E}[\mathbb{E} [ \mathrm{trace}(A_i \bv_1 \bv_1^{\prime} A_j^{\prime} \bv_m \bv^{\prime}_m)]|\bv_1] \\
		& = & \mathbb{E}[ \mathrm{trace}( A_i \bv_1 \bv_1^{\prime} A_j^{\prime} (I-\bv_1 \bv_1^{\prime})/(q-1))|\bv_1] \\
		& = & \mathbb{E} [ \mathrm{trace}( A_i \bv_1 \bv_1^{\prime} A_j^{\prime} /(q-1))] = \mathbb{E} [ \mathrm{trace}( A_i I A_j^{\prime} / (q(q-1)))] = 0,
	\end{eqnarray*}
	where the first equality applies the fact that, conditioned on $\bv_1$, $\mathbb{E}[\bv_m \bv^{\prime}_m]$ is the identity matrix in the subspace $\ct(\bv_1 \bv_1^{\prime})^{\perp}$ suitably scaled, and the second inequality applies the previous case. \qed
\end{proof}

\begin{proof}[Proposition \ref{thm:randensemble}]	
	First we have $\mathsf{0} \preceq X^{(j)} \boxtimes X^{(j)} \preceq s^2 r \sfi$.  By Proposition \ref{thm:xxmean} we have $\mathbb{E}[X^{(j)} \boxtimes X^{(j)}] = (\bar{s}^2 r/q^2) \sfi$.  Since $(X^{(j)} \boxtimes X^{(j)} - (\bar{s}^2 r/q^2) \sfi)^2 \preceq s^4 r^2 \sfi$, we have $ \mathbb{P} ( \|(1/n)\sum_{i=1}^{n} X^{(j)} \boxtimes X^{(j)} - (\bar{s}^2 r/q^2)\sfi \| > t r s^2 ) \leq 2q \exp(- t^2 n / 8)$ via an application of the Matrix Hoeffding inequality (Theorem 1.3 in \cite{Tro:12}).

	Second we have $\|X^{(j)} \boxtimes X^{(j)}\|_{\spec} \leq s^2 r$, and $\| \cp_{\ct(X^{(j)})}\|_{\spec} = 1$, and hence $(X^{(j)} \boxtimes X^{(j)}) \botimes \cp_{\ct(X^{(j)})} \preceq s^2 r \sfi \botimes \sfi =: s^2 r \mathtt{I}$. From Proposition \ref{thm:expectationofgiantoperator} we have
	\begin{equation*}
	\mathbb{E}[ (X^{(j)} \boxtimes X^{(j)}) \botimes \cp_{\ct(X^{(j)})} ] \preceq \frac{\bar{s}^2 r}{q^2} \mathtt{I}_{\W} + \frac{16 \bar{s}^2 r^2}{q^3} \mathtt{I}_{\W^{\perp}}.
	\end{equation*}
	Since $((X^{(j)} \boxtimes X^{(j)}) \botimes \cp_{\ct(X^{(j)})} - r \mathtt{I})^2 \preceq s^4 r^2 \mathtt{I}$ we have 
	\begin{eqnarray*}
	& & \mathbb{P} \left( \lambda_{\max} \left( \frac{1}{n} \sum_{i=1}^{n}  (X^{(j)} \boxtimes X^{(j)}) \botimes \cp_{\ct(X^{(j)})}- \mathbb{E} [(X^{(j)} \boxtimes X^{(j)}) \botimes \cp_{\ct(X^{(j)})}] \right) \geq t r s^2 \right) \\
	& \leq & q \exp( - t^2 n / 8) 
	\end{eqnarray*}
	by an application of the Matrix Hoeffding inequality.
	
	Let $t = t_1 /(5q^2)$ in the first concentration bound, and $t= t_2 / (5q^2)$ in the second concentration bound.  Then $\covsup( \ensa) \leq t_1 s^2 r / (5q^2)$, and $\roc( \ensa) \leq 16s^2 r^2/q^3 + t_2 s^2 r / (5q^2)$, with probability greater than $1- 2q\exp(-n t_1^2 / (200q^4)) - q \exp(- n t_2^2 / (200q^4))$.  We condition on the event that both inequalities hold.  Since $\covsup( \ensa) \leq t_1 s^2 r / (5q^2) \leq s^2 r/(20q^2)$, by Lemma \ref{thm:boundondelta} we have $\coveig( \ensa) \geq s^2 r/(5q^2)$, and hence $\covsup( \ensa) / \coveig( \ensa) \leq t_1$, and $\roc( \ensa) / \coveig( \ensa) \leq 80r/q + t_2$. \qed
\end{proof}

%% file: appendix_sinkhorn.tex
\section{Stability of Matrix and Operator Scaling} \label{apx:sinkhornstability}

In this section we prove a stability property of Sinkhorn scaling and Operator Sinkhorn scaling.  For Sinkhorn scaling, we show that if a matrix is close to being doubly stochastic and has entries that are suitably bounded away from $0$, then the resulting row and column scalings are close to $\ones:= (1,\ldots,1)^{\prime}$.  We also prove the operator analog of this result.  
These results are subsequently used to prove Propositions \ref{thm:normalizednearothogonal} and \ref{thm:gaussianmapsatisfy}.  We note that there is an extensive literature on the stability of matrix scaling, with results of a similar flavor to ours.  However, Proposition \ref{thm:sinkhornstability} in this section is stated in a manner that is directly suited to our analysis, and we include it for completeness.

\subsection{Main results}

\begin{proposition}[Local stability of Matrix Scaling] \label{thm:sinkhornstability}
	Let $T \in \mathbb{R}^{q\times q}$ be a matrix such that
	\begin{enumerate}
		\item $|\langle \be_i, T(\be_j)\rangle - 1/q| \leq 1/(2q)$ for all standard basis vectors $\be_i,\be_j$; and
		\item $\epsilon:= \max \{\|T \ones - \ones \|_{\infty},\| T^{\prime} \ones - \ones \|_{\infty} \} \leq 1/(48 \sqrt{q})$.
	\end{enumerate}
	Let $D_1, D_2$ be diagonal matrices such that $D_2 T D_1$ is doubly stochastic.  Then
	\begin{equation*}
	\| D_2 \botimes D_1 - \sfi \|_{\spec} \leq 96 \sqrt{q} \epsilon.
	\end{equation*}
\end{proposition}

\begin{proposition}[Local stability of Operator Scaling] \label{thm:oepratorsinkhornstability}
	Let $\cpt: \mathbb{S}^{q} \rightarrow \mathbb{S}^{q}$ be a rank-indecomposable linear operator such that
	\begin{enumerate}
		\item $|\langle \bv\bv^{\prime}, \cpt (\bu \bu^{\prime})\rangle -1/q| \leq 1/(2q)$ for all unit-norm vectors $\bu,\bv \in \mathbb{R}^{q}$; and
		\item $\epsilon:= \max \{ \| \cpt(I)-I \|_{\spec} , \| \cpt^{\prime}(I)-I \|_{\spec}\} \leq 1/(48\sqrt{q})$.
	\end{enumerate}
	Let $N_1, N_2 \in \mathbb{S}^{q}$ be positive definite matrices such that $ (N_2 \botimes N_2) \circ \cpt \circ (N_1 \botimes N_1) $ is doubly stochastic.  Then $\| N_2^2 \botimes N_1^2 - \sfi \|_{\spec} \leq 96 \sqrt{q} \epsilon$.  Furthermore we have $\| N_2 \botimes N_1 - \sfi \|_{\spec} \leq 96 \sqrt{q} \epsilon$.
\end{proposition}

\subsection{Proofs}

The proof of Proposition \ref{thm:sinkhornstability} relies on the fact that matrix scaling can be cast as the solution of a convex program; specifically, we utilize the correspondence between diagonal matrices $D_1, D_2$ such that $D_2 T D_1$ is doubly stochastic, and the vectors $\boldsymbol \varepsilon := (\varepsilon_1,\ldots, \varepsilon_q)^{\prime}, \boldsymbol \eta := (\eta_1,\ldots, \eta_q)^{\prime}$ that minimize the following convex function
\begin{equation*}
F(\boldsymbol \varepsilon, \boldsymbol \eta) = \sum_{ij} T_{ij} \exp( \varepsilon_i + \eta_j) - \sum_i \varepsilon_i - \sum \eta_j
\end{equation*}
via the maps $(D_2)_{ii} = \exp(\varepsilon_i)$ and $(D_1)_{jj} = \exp(\eta_j)$ \cite{Gor:63} (see also \cite{KhaKal:91}) -- this holds for all matrices $T$ with positive entries.  We remark that one can derive the above relationship from first order optimality.  In the following we prove bounds on the minima of $F$ (see Lemma \ref{thm:optimalf}).  

The proof of Proposition \ref{thm:oepratorsinkhornstability} relies on a reduction to the set-up in Proposition \ref{thm:sinkhornstability}.

We begin with a lower estimate of the sum of exponential functions.  We use the estimate to prove Proposition \ref{thm:sinkhornstability}.
\\

\begin{definition}
	Let $\alpha\geq 0$. Define the function $c_{\alpha}: \mathbb{R} \rightarrow \mathbb{R}$
\begin{equation*}
c_{\alpha} (x)= \begin{cases}
\frac{1}{2} \exp(-\alpha) x^2 \quad \text{ if } |x| \leq \alpha \\
\frac{1}{2} \exp(-\alpha) \alpha |x| \quad \text{ if } |x| \geq \alpha
\end{cases}
\end{equation*}
\end{definition}

\noindent \textbf{Remark.} Note that the function $c_{\alpha}(\cdot)$ is continuous. 

\begin{lemma}\label{thm:lowerboundonexp}
	For all $x$
	\begin{equation*}
	\exp(x) \geq 1+x+c_{\alpha} (x).
	\end{equation*}
\end{lemma}

\begin{proof}[Lemma \ref{thm:lowerboundonexp}]
	The second derivative of $\exp(x)$ is $\exp(x)$, and it is greater than $\exp(-\alpha)$ over all $x$ such that $|x| \leq \alpha$.  Hence, by strong convexity of $\exp(x)$, we have $\exp(x) \geq 1 + x + (1/2)\exp(-\alpha) x^2$ over the interval $[-\alpha,\alpha]$.
	
	It follows that $\exp(\alpha) \geq 1 + \alpha + c_{\alpha} (\alpha)$, and $\exp(-\alpha) \geq 1 - \alpha + c_{\alpha} (-\alpha)$.  Since the function $\exp(x)$ is convex, and $c_{\alpha}$ is linear in the intervals $(-\infty,-\alpha]$ and $[\alpha,\infty)$ respectively, it suffices to check that (i) the gradient of $\exp(x)$ at $x=\alpha$, which is $\exp(\alpha)$, exceeds that of $c_{\alpha}(\cdot)$, and (ii) the gradient of $c_{\alpha}(\cdot)$ exceeds that of $\exp(x)$ at $x = -\alpha$, which is $\exp(-\alpha)$.
	
	First we prove (i).  Since $\alpha \geq 0$ we have $1+2\alpha \geq \sqrt{1+2\alpha}$.  Hence $2 \exp(\alpha) \geq 2 + 2 \alpha \geq 1 + \sqrt{1 + 2\alpha}$.  By noting that the quadratic $2z^2- 2z - \alpha=0$ has roots $(1/2) \pm (1/2)\sqrt{1 + 2\alpha}$, we have the inequality $\exp(\alpha) \geq 1 + (1/2)\exp(-\alpha) \alpha$, from which (i) follows.  
	
	Next we prove (ii).  Since $\alpha \geq 0$, we have $\exp(\alpha) \geq 1+\alpha \geq 1 + \alpha/2$, and hence $1 - (1/2)\exp(-\alpha) \alpha \geq \exp(-\alpha)$ from which (ii) follows.\qed
\end{proof}

\begin{lemma}\label{thm:lbsumexps}
	Let $\{\varepsilon_i\}_{i=1}^{q}$ and $\{\eta_j\}_{j=1}^{q}$ be a collection of reals satisfying $(\sum_{i} \varepsilon_i) + (\sum_{j} \eta_j) \geq -2q$. Then there is a constant $d \in \mathbb{R}$ for which
	\begin{equation*}
	\frac{1}{q}\sum_{ij} \exp(\varepsilon_i + \eta_j ) \geq q + \left( \sum_i ( \varepsilon_i + c_{\alpha} (\varepsilon_i + d) ) \right) + \left( \sum_j ( \eta_j + c_{\alpha} (\eta_j - d) )  \right).
	\end{equation*}
\end{lemma}

\begin{proof}  Consider the function
	\begin{equation*}
		f(d) := \sum_{i} \left(\varepsilon_i + d + c_{\alpha} (\varepsilon_i + d) \right) - 
		\sum_j \left( \eta_j - d + c_{\alpha} (\eta_j-d ) \right).
	\end{equation*}
	Then $f(\cdot)$ is continuous in $d$, and  $f(d) \rightarrow \pm \infty$ as $d \rightarrow \pm \infty$.  By the Intermediate Value Theorem, there is a $d^{\star}$ for which $f(d^{\star}) = 0$.  Then
	\begin{equation*}
		\sum_{i} \left(1 + \varepsilon_i + d^{\star} + c_{\alpha} (\varepsilon_i + d^{\star}) \right) = \sum_j \left( 1 + \eta_j - d^{\star} + c_{\alpha} (\eta_j-d^{\star} ) \right).
	\end{equation*}
	By summing both sides and noting that $c_{\alpha}(\cdot) \geq 0$, we have that each side of the above equation is nonnegative.  It follows that
	\begin{eqnarray*}
		\frac{1}{q}\sum_{ij} \exp(\varepsilon_i + \eta_j ) & = & \frac{1}{q} \left(\sum_{i} \exp(\varepsilon_i+d^{\star}) \right) \left( \sum_j \exp(\eta_j -d^{\star}) \right) \\
		& \geq & \frac{1}{q} \left(\sum_i \left( 1+\varepsilon_i+d^{\star} + c_{\alpha} (\varepsilon_i+d^{\star}) \right)\right) \left( \sum_j \left( 1 + \eta_j-d^{\star} + c_{\alpha} (\eta_j-d^{\star}) \right)\right) \\
		& \geq &  q + \left( \sum_i ( \varepsilon_i + c_{\alpha} (\varepsilon_i + d) ) \right) + \left( \sum_j ( \eta_j + c_{\alpha} (\eta_j - d) )  \right).
	\end{eqnarray*}
	\qed
\end{proof}

\begin{lemma} \label{thm:optimalf}  Given vectors $\boldsymbol \varepsilon := (\varepsilon_1,\ldots,\varepsilon_q)$ and $\boldsymbol \eta := (\eta_1,\ldots, \eta_q)$ define
	\begin{equation}
	F ( \boldsymbol{ \varepsilon}, \boldsymbol \eta) = \sum_{ij} T_{ij} \exp(\varepsilon_i + \eta_j) - \sum_{i} \varepsilon_i - \sum_{j} \eta_j,
	\end{equation}
	and $\epsilon_{ij} := T_{ij}- 1/q $. Suppose (i) $|\epsilon_{ij} | \leq 1/2q$, and (ii) $ \epsilon:= \max \{ |\sum_i \epsilon_{ij} |, |\sum_j \epsilon_{ij}| \}\leq 1/(24 \sqrt{q})$.  Let $\boldsymbol \varepsilon^{\star}, \boldsymbol \eta^{\star}$ be a minimizer of $F$.  Then $| \varepsilon_i^{\star} + \eta_j^{\star} | \leq 48 \sqrt{q} \epsilon$, for all $i,j$.
\end{lemma}

\begin{proof}  Suppose $| \varepsilon_i + \eta_j | > 48 \sqrt{q} \epsilon$ for some $(i,j)$.  We show that $\boldsymbol \varepsilon,\boldsymbol \eta$ cannot be a minimum.  We split the analysis to two cases.
	
\noindent [$(\sum_{i} \varepsilon_i) + (\sum_{j} \eta_j) < -2q$]:  Since $T_{ij} > 0$ we have $F(\boldsymbol \varepsilon,\boldsymbol \eta) > - (\sum_{i} \varepsilon_i) - (\sum_{j} \eta_j) \geq 2q$.  Then $F(\boldsymbol 0,\boldsymbol 0) = \sum_{i} (\sum_{j} T_{ij} ) = \sum_{i} (1+ \sum_{j} \epsilon_{ij}) \leq q(1+ 1/(24\sqrt{q})) \leq 2q < F(\boldsymbol \varepsilon,\boldsymbol \eta)$.
	
\noindent [$(\sum_{i} \varepsilon_i) + (\sum_{j} \eta_j) \geq -2q$]:  Let $\alpha = 24 \sqrt{q}\epsilon $, and define the sets
	\begin{enumerate}
		\item $\mathfrak{S}(\boldsymbol \varepsilon) = \{i: | \varepsilon_i | \geq \alpha \}$;
		\item $\mathfrak{T}(\boldsymbol \varepsilon)  = \{i: \alpha > | \varepsilon_i | \geq 4 \epsilon \exp(\alpha) \}$; and
		\item $\mathfrak{U}(\boldsymbol \varepsilon)  = \{i: 4 \epsilon \exp(\alpha) > | \varepsilon_i | \}$.
	\end{enumerate}
	Similarly define the sets $\mathfrak{S}(\boldsymbol \eta), \mathfrak{T}(\boldsymbol \eta), \mathfrak{U}(\boldsymbol \eta)$.
	
	First since $\alpha \leq 1$, we have $\alpha \geq \alpha\exp(\alpha)/3 \geq 8 \sqrt{q} \epsilon \exp(\alpha) \geq 8 \epsilon \exp(\alpha)$, and hence
	\begin{equation*}
	\frac{1}{4} \biggl( \sum_{i\in\mathfrak{S}(\boldsymbol \varepsilon)} c_{\alpha}(\varepsilon_i) + \sum_{j\in\mathfrak{S}(\boldsymbol \eta)} c_{\alpha}(\eta_j) \biggr) \geq \epsilon \biggl( \sum_{i\in\mathfrak{S}(\boldsymbol \varepsilon)} |\varepsilon_i| + \sum_{j\in\mathfrak{S}(\boldsymbol \eta)} |\eta_j| \biggr).
	\end{equation*}
	
	Second
	\begin{align*}
	\frac{1}{2} \biggl(\sum_{i\in\mathfrak{T}(\boldsymbol \varepsilon)} c_{\alpha}(\varepsilon_i) + \sum_{j\in\mathfrak{T}(\boldsymbol \eta)} c_{\alpha}(\eta_j) \biggr) & = \sum_{i\in\mathfrak{T}(\boldsymbol \varepsilon)} \frac{1}{4} \exp(-\alpha) \varepsilon_i^2 + \sum_{j\in\mathfrak{T}(\boldsymbol \eta)} \frac{1}{4} \exp(-\alpha) \eta_j^2  \\
	& \geq \epsilon \biggl( \sum_{i\in\mathfrak{T}(\boldsymbol \varepsilon)} |\varepsilon_i| + \sum_{j\in\mathfrak{T}(\boldsymbol \eta)} |\eta_j| \biggr).
	\end{align*}
	
	Third since there is an index $(i,j)$ such that $| \varepsilon_i + \eta_j | > 48 \sqrt{q} \epsilon$, one of the sets $\mathfrak{S}(\boldsymbol \varepsilon), \mathfrak{S}(\boldsymbol \eta)$ is nonempty. By noting that $\alpha \exp(-\alpha) \geq 8 \sqrt{q} \epsilon$, we have
	\begin{equation*}
	\frac{1}{4} \biggl( \sum_{i\in\mathfrak{S}(\boldsymbol \varepsilon)} c_{\alpha}(\varepsilon_i) + \sum_{j\in\mathfrak{S}(\boldsymbol \eta)} c_{\alpha}(\eta_j) \biggr) > \epsilon \times 2q \times 4 \epsilon\exp(\alpha) \geq \epsilon \biggl( \sum_{i\in\mathfrak{U}(\boldsymbol \varepsilon)} |\varepsilon_i| + \sum_{j\in\mathfrak{U}(\boldsymbol \eta)} |\eta_j| \biggr).
	\end{equation*}
	
	We have $\epsilon(\sum_{i}|\varepsilon_i| + \sum_{j}|\eta_j|) \geq \sum_i (\varepsilon_i (\sum_j \epsilon_{ij})) + \sum_j (\eta_j (\sum_i \epsilon_{ij})) = \sum_{ij} \epsilon_{ij}(\varepsilon_i+\eta_j)$.  By combining the above inequalities with Lemma \ref{thm:lbsumexps} we have
	\begin{equation} \label{eq:flb_bound2}
	\frac{1}{2q} \sum \biggl(\exp(\varepsilon_i + \eta_j) - (\varepsilon_i + \eta_j ) - 1 \biggr) \geq \frac{1}{2} \biggl( \sum_{i} c_{\alpha}(\varepsilon_i) + \sum_{j} c_{\alpha}(\eta_j) \biggr) > \sum_{ij} \epsilon_{ij} (\varepsilon_i + \eta_j).
	\end{equation}
	
	Also, since $ \exp(\varepsilon_i + \eta_j ) -(\varepsilon_i + \eta_j) - 1 \geq 0$ for all $i,j$, and $|\epsilon_{ij}| \leq 1/(2q)$, we have
	\begin{eqnarray}\label{eq:flb_bound1}
	\frac{1}{2q} \sum_{ij}   (\exp(\varepsilon_i + \eta_j ) -(\varepsilon_i + \eta_j) - 1) & \geq & \max_{ij}|\epsilon_{ij}| \times \sum_{ij}  \biggl|\exp(\varepsilon_i + \eta_j ) -(\varepsilon_i + \eta_j) - 1 \biggr| \nonumber \\
	& \geq & \sum_{ij} \epsilon_{ij} (\exp(\varepsilon_i + \eta_j ) -(\varepsilon_i + \eta_j) - 1).
	\end{eqnarray}
	
	By combining equations \eqref{eq:flb_bound2} and \eqref{eq:flb_bound1} we have
	\begin{equation*}
	\frac{1}{q} \sum_{ij}   (\exp(\varepsilon_i + \eta_j ) -(\varepsilon_i + \eta_j) - 1) > -\sum_{ij} \epsilon_{ij} (\exp(\varepsilon_i + \eta_j )  - 1),
	\end{equation*}
	which implies $F(\boldsymbol \varepsilon, \boldsymbol \eta) > F(\boldsymbol 0, \boldsymbol 0)$. \qed
\end{proof}

\begin{proof}[Proposition \ref{thm:sinkhornstability}]
	By Lemma \ref{thm:optimalf} any minimum $\boldsymbol \varepsilon^{\star}, \boldsymbol \eta^{\star}$ satisfies $| \varepsilon_i^{\star} + \eta_j^{\star} | \leq 48 \sqrt{q} \epsilon$.  Hence by the one-to-one correspondence between the minima of $F$ and the diagonal scalings $D_1,D_2$ \cite{Gor:63}, we have $\| D_2 \botimes D_1 - \sfi \|_{\spec} \leq \exp(48 \sqrt{q} \epsilon) -1 \leq 96 \sqrt{q} \epsilon$. \qed
\end{proof}

\begin{proof}[Proposition \ref{thm:oepratorsinkhornstability}]
	Without loss of generality we may assume that $N_1,N_2$ are diagonal matrices, say $D_1,D_2$ respectively.  Define the matrix $T_{ij} = \langle \be_i \be_i^{\prime}, \cpt (\be_j \be_j^{\prime}) \rangle$.  It is straightforward to check that $T$ satisfies the conditions of Proposition \ref{thm:sinkhornstability}; moreover, the condition that $ (N_2 \botimes N_2) \circ \cpt \circ (N_1 \botimes N_1) $ is a doubly stochastic operator implies that $D_2^2 T D_1^2$ is a doubly stochastic matrix.  By Proposition \ref{thm:sinkhornstability} we have $\| D_1^2 \botimes D_2^2 - \sfi \|_{\spec} \leq 96 \sqrt{q} \epsilon$, and hence $\| N_1^2 \botimes N_2^2 - \sfi \|_{\spec} \leq 96 \sqrt{q} \epsilon$.  Since $N_1,N_2$ are self-adjoint, we also have $\| N_1 \botimes N_2 - \sfi \|_{\spec} \leq 96 \sqrt{q} \epsilon$. \qed
\end{proof}

%% file: appendix_randlinmap.tex
\section{Proof of Proposition \ref{thm:gaussianmapsatisfy}} \label{apx:randlinmaps}

In this section we prove that Gaussian linear maps that are subsequently normalized satisfy the deterministic conditions in Theorem \ref{thm:localconvergence} concerning the linear map $\L^{\star}$ with high probability.  There are two steps to our proof.  First we state sufficient conditions for linear maps such that, when normalized, satisfy the deterministic conditions.  Second we show that Gaussian maps satisfy these sufficient conditions with high probability. 

We introduce the following parameter that measures how close a linear map $\L$ is to being normalized.

\begin{definition}
	Let $\L \in \mathbb{R}^{q\times q} \rightarrow \mathbb{R}^d$ be a linear map. The \emph{nearly normalized parameter} of $\L$ is defined as
	\begin{equation*}
	\epsilon(\L) := \max \{ \|\cpt_{\L} (I) - I\|_{\spec}, \|\cpt_{\L}^{\prime} (I) - I\|_{\spec}\} .
	\end{equation*}
\end{definition}

\begin{proposition} \label{thm:ripnppequalsrip}
	Let $\L : \mathbb{R}^{q\times q} \rightarrow \mathbb{R}^d$ be a linear map that satisfies (i) the restricted isometry condition $\delta_r(\L)\leq 1/2$, and (ii) whose nearly normalized parameter satisfies $\epsilon (\L) \leq 1/ (650 \sqrt{q})$.  Let $\L \circ \sfin_{\L}$ be the normalized linear map where $\sfin_{\L}$ is a positive definite rank-preserver.  Then $\L \circ \sfin_{\L}$ satisfies the restricted isometry condition $\delta_r(\L \circ \sfin) \leq \bar{\delta_r}:= (1+\delta_r(\L)) (1+ 96\sqrt{q} \epsilon (\L))^2 - 1 < 1$. Moreover, $\|\L \circ \sfin_{\L} \|_{\spec} \leq (1+ 96\sqrt{q} \epsilon (\L)) \|\L\|_{\spec}$.
\end{proposition}

\begin{proof}[Proposition \ref{thm:ripnppequalsrip}]
	Since $\L$ satisfies the restricted isometry condition $\delta_1(\L) \leq 1/2$, we have $|\langle \bv\bv^{\prime}, \cpt_{\L} (\bu \bu^{\prime})\rangle -1/q| \leq 1/(2q)$ for all unit-norm vectors $\bu,\bv \in \mathbb{R}^{q}$.  In addition, the linear map $\L$ has nearly normalized parameter $\epsilon (\L) \leq 1/ (650 \sqrt{q})$.  Hence by applying Proposition \ref{thm:oepratorsinkhornstability} to the linear map $\cpt_{\L}$, any pair of positive definite matrices $Q_2,Q_1$ such that $Q_2 \botimes Q_2 \circ \cpt_{\L} \circ Q_1 \botimes Q_1$ is doubly stochastic satisfies $\|Q_2 \botimes Q_1 - \sfi \|_{\spec} \leq 96 \sqrt{q} \epsilon(\L)$.  By noting the correspondence between such matrices with the positive definite rank-preserver $\sfin_{\L}$ such that $\L \circ \sfin_{\L}$ is normalized via the relation $\sfin_{\L} = Q_2 \botimes Q_1$ (see Corollary \ref{thm:external_normalizablelinearmaps}), we have $\|\sfin_{\L}\|_{\spec} \leq 1+96 \sqrt{q} \epsilon(\L)$.
	
	Let $X$ be a matrix with rank at most $r$.  Then
	\begin{equation*}
		\| \L (\sfin_{\L} (X)) \|_{\fro} \leq \sqrt{1+\delta_r(\L)} \| \sfin_{\L} \|_{\spec} \|X\|_{\fro} \leq \sqrt{1+\delta_r(\L)}(1 + 96 \sqrt{q} \epsilon(\L) ) \| X \|_{\fro},
	\end{equation*}	
	and hence $\| \L (\sfin_{\L} (X)) \|_{\fro}^2 \leq (1+\bar{\delta_r}) \|X\|_{\fro}^2$.  A similar set of steps show that $\| \L (\sfin_{\L} (X)) \|_{\fro}^2 \geq (1-\bar{\delta_r}) \|X\|_{\fro}^2$.  Last $\|\L \circ \sfin_{\L} \|_{\spec} \leq \|\L\|_{\spec}\|\sfin_{\L} \|_{\spec} \leq (1+96\sqrt{q}\epsilon)\|\L\|_{\spec}$. \qed
\end{proof}

\begin{proposition} (\cite[Theorem II.13]{DavSza:01}) \label{thm:gaussianspec}
	Let $t>0$ be fixed.  Suppose $\L \sim \mathcal{N} (0,1/d)$.  Then with probability greater than $1 - \exp(- t^2d/2) $ we have $\|\L\|_{\spec} \leq \sqrt{q^2/d} + 1 + t$.
\end{proposition}

\begin{proposition}(\cite[Theorem 2.3]{CanPla:11}) \label{thm:gaussianrip}
	Let $0<\delta<1$ be fixed.  There exists constants $c_1,c_2$ such that for $d \geq c_1 qr$, if $\L \sim \mathcal{N} (0,1/d)$, then with probability greater than $1 - 2 \exp(- c_2 d)$ the linear map $\L$ satisfies the restricted isometry condition $\delta_r(\L) \leq \delta$.
\end{proposition}

\begin{proposition}[Gaussian linear maps are nearly normalized]\label{thm:rowcoldev}
	Suppose $3/\sqrt{d} \leq \epsilon \leq 3$.  Suppose $\L \sim \mathcal{N}(0, 1/d)$.  Then with probability greater than $1-4 \exp( - q(-1 + \sqrt{d} \epsilon / 3)^2/2)$ the nearly normalized parameter of $\L$ is smaller than $\epsilon$.
\end{proposition}


Bounding the nearly normalized parameter of a Gaussian linear map exactly corresponds to computing the deviation of the sum of independent Wishart matrices from its mean in spectral norm.  To do so we appeal to the following concentration bound.

\begin{proposition}[Concentration of sum of Wishart Matrices] \label{thm:concentrationwishart}
	Suppose $3/\sqrt{d} \leq t \leq 3$.  Let $\{X^{(j)} \}_{j=1}^{d}, X^{(j)} = G^{(j)} G^{(j)\prime}$, where $G^{(j)} \in \R^{q\times q}, G^{(j)} \sim \mathcal{N}(0,1/q)$, be a collection of independent Wishart matrices.  Then $\mathbb{P} ( \|\frac{1}{d}\sum_{j=1}^{d} X^{(j)} - I  \|_{\spec} \geq t) \leq 2 \exp( - q(-1 + \sqrt{d} t / 3)^2/2)$.
\end{proposition}

\begin{proof}[Proposition \ref{thm:concentrationwishart}]
	Consider the linear map $G = [G^{(1)}|\ldots|G^{(d)}]$.  Then $\sum_{j=1}^{d} X^{(j)} = G G^{\prime}$, and $\|\frac{1}{d}\sum_{j=1}^{d} X^{(j)} - I  \|_{\spec} \leq t$ if and only if $\sigma(G) \in [\sqrt{d(1-t)},\sqrt{d(1+t)}] $.  By \cite[Theorem II.13]{DavSza:01} we have $\sigma(G) \in [\sqrt{d} - 1 - \tilde{t},\sqrt{t} + 1 + \tilde{t}]$ with probability greater than $1-2\exp(-q\tilde{t}^2/2)$.  The result follows with the choice of $\tilde{t} = -1 + \sqrt{d} t / 3$.
	\qed
\end{proof}

\begin{proof}[Proposition \ref{thm:rowcoldev}]
	This is a direct application of Proposition \ref{thm:concentrationwishart} with $G^{(j)} =\sqrt{q/d} \L^{(j)}$ and $G^{(j)\prime} =\sqrt{q/d} \L^{(j)}$, followed by a union bound. \qed
\end{proof}

\begin{proof}[Proposition \ref{thm:gaussianmapsatisfy}]
	We choose $t = 1/50$ in Proposition \ref{thm:gaussianspec}, $\delta = \delta_{4r}/2$ in Proposition \ref{thm:gaussianrip}, and $\epsilon = \delta / (960\sqrt{q})$ in Proposition \ref{thm:rowcoldev}.  Then there are constants $c_1, c_2, c_3$ such that if $d \geq c_1 r q$, then (i) $\|\tilde{\L}\|_{\spec}\leq \sqrt{q^2/d}+ 51/50 \leq (101/50) \sqrt{q^2/d}$, (ii) $\tilde{\L}$ satisfies the restricted isometry condition $\delta_{4r}(\tilde{\L}) \leq \delta_{4r}/2$, and (iii) $\tilde{\L}$ is nearly normalized with parameter $\epsilon(\tilde{\L}) \leq \delta_{4r}/ 960 \sqrt{q}$, with probability greater than $1 - c_2 \exp( - c_3 d)$.
	
	By applying Proposition \ref{thm:ripnppequalsrip} we conclude that the linear map $\L$ satisfies the restricted isometry condition $\delta_{4r}(\L)\leq (1+\delta_{4r}/2) (1 + \delta_{4r}/10)^2 -1 \leq \delta_{4r}$, and $\|\L \|_{\spec} \leq \sqrt{5q^2/d}$. \qed
\end{proof}

%% file: appendix_normstab.tex
\section{Proof of Proposition \ref{thm:normalizednearothogonal}}

\begin{proof}[Proposition \ref{thm:normalizednearothogonal}]
	First we check that the linear map $\L^{\star} \circ (\sfi + \sfe)$ satisfies the restricted isometry condition $\delta_1 (\L^{\star} \circ (\sfi + \sfe)) \leq 1/2$.  For any rank-one unit-norm matrix $X$ we have $\| [\L^{\star} \circ (\sfi + \sfe)] (X) \|_{\ell_2} \leq \| \L^{\star} (X) \|_{\ell_2} + \| \L^{\star} (\sfe (X)) \|_{\ell_2} \leq \sqrt{1+ 1/10} + 1/150 \leq \sqrt{1+1/2}$.  A similar set of inequalities show that $\| [\L^{\star} \circ (\sfi + \sfe)] (X) \|_{\ell_2} \geq \sqrt{1 - 1/2}$.
	
	Second we check that the nearly normalized parameter of $\L^{\star} \circ (\sfi+\sfe)$ satisfies $\epsilon (\L^{\star} \circ (\sfi+\sfe)) \leq 1/ 48\sqrt{q}$.  Denote $\mathcal{E} := \L^{\star} \circ \sfe$.  For all unit-norm rank-one matrices $E$ we have $\|\mathcal{E}(E)\|_{\spec}^2\leq \|\L^{\star}\|_{\spec}^2\|\sfe\|_{\eu}^2$.  Hence for any unit-norm $\bu\in\mathbb{R}^{q}$ we have
	\begin{equation*}
	\frac{1}{q} \sum_{j=1}^{d} \langle \mathcal{E}_{j} \mathcal{E}^{\prime}_{j}, \bu \bu^{\prime}\rangle  = \frac{1}{q}  \sum_{j=1}^{d} \sum_{k=1}^{q} (\mathcal{E}^{\prime}_{j} \bu)_k^2 = \frac{1}{q} \sum_{k=1}^{q}  \| \mathcal{E} ( \bu \be_k^{\prime}) \|_{\ell_2}^2 \leq \|\L^{\star}\|_{\spec}^2\|\sfe\|_{\eu}^2.
	\end{equation*}
	Using the fact that $\L^{\star}$ is normalized we have 
	\begin{equation*}
	\frac{1}{q} \sum_{j=1}^{d} \langle \L^{\star}_{j} \L^{\star\prime}_{j}, \bu \bu^{\prime}\rangle  = 1.
	\end{equation*}
	By combining the previous inequalities with an application of Cauchy-Schwarz we have
	\begin{eqnarray*}
	& & \langle \cpt_{\L^{\star} \circ (\sfi+\sfe) } (I)  - I , \bu \bu^{\prime} \rangle   \\
	&= &\langle \cpt_{\L^{\star} + \mathcal{E}} (I)  - \cpt_{\L^{\star}}  (I) , \bu \bu^{\prime} \rangle \\
	& = & \frac{1}{q} \sum_{j=1}^{d} \langle \mathcal{E}_{j} \mathcal{E}^{\prime}_{j}, \bu \bu^{\prime}\rangle + \frac{1}{q} \sum_{j=1}^{d} \langle \L^{\star}_{j} \mathcal{E}^{\prime}_{j}, \bu \bu^{\prime}\rangle + \frac{1}{q} \sum_{j=1}^{d} \langle \mathcal{E}_{j} \L^{\star\prime}_{j}, \bu \bu^{\prime}\rangle \\
	& \leq & 3 \|\L^{\star}\|_{\spec} \|\sfe\|_{\eu},
	\end{eqnarray*}
	Further more since $\bu$ is arbitrary it follows that
	\begin{equation*}
	\| \cpt_{\L^{\star} + \mathcal{E}} (I)  - I \|_{\spec} \leq 3 \|\L^{\star}\|_{\spec} \|\sfe\|_{\eu}.
	\end{equation*}
	Using a similar sequence of steps one can show that $\| \cpt_{\L^{\star} + \mathcal{E}}^{\prime} (I)  - I \|_{\spec} \leq 3 \|\L^{\star}\|_{\spec} \|\sfe\|_{\eu}$.  Thus $\epsilon (\L^{\star} \circ (\sfi+\sfe)) \leq 3 \|\L^{\star}\|_{\spec} \|\sfe\|_{\eu} \leq 1/( 48\sqrt{q})$.

	The result follows by applying Proposition \ref{thm:oepratorsinkhornstability} to the linear map $\cpt_{\L^{\star} \circ (\sfi+\sfe)}$. \qed
\end{proof}

%% file: appendix_variety.tex
\section{Proof of Proposition \ref{thm:varietyconstrainederr}} \label{apx:varietyanalysis}

The proof of Proposition \ref{thm:varietyconstrainederr} is based on the following result concerning affine rank minimization, which may be of independent interest.

\begin{proposition} \label{thm:varietyconstrainedoptimization}
	Suppose $X^{\star}$ is a $q\times q$ rank-$r$ matrix satisfying $\sigma_r (X^{\star}) \geq 1/2$. Let $\by = \L(X^{\star}) + \bz$, where the linear map $\L$ satisfies the restricted isometry condition $\delta_{4r}(\L) \leq 1/10$, and $\| \L^{\prime} \bz\|_{\spec} =: \epsilon \leq 1 / (80 r^{3/2})$. Let $\hat{X}$ be the optimal solution to
	\begin{equation*}
	\hat{X}~=~\underset{X}{\mathrm{argmin}}~\| \by - \L(X) \|^2_{\ell_2} \quad \quad \mathrm{s.t.} \quad \quad \mathrm{rank}(X) \leq r.
	\end{equation*}
	Then (i) $\| \hat{X} - X^{\star} \|_{\spec} \leq 4 \sqrt{r} \epsilon$, and (ii) $ \hat{X} - X^{\star} = [(\L^{\prime}_{\ct(X^{\star})}\L_{\ct(X^{\star})})^{-1}]_{\R^{q \times q}} (\L_{\ct(X^{\star})}^{\prime} \bz) + G$, where $\|G\|_{\fro} \leq 340 r^{3/2} \epsilon^2 $. 
\end{proposition}

The proof of Proposition \ref{thm:varietyconstrainedoptimization} requires two preliminary results which we state and prove first.  Our development relies on results from matrix perturbation theory; we refer the reader to \cite{Kat:66,SteSun:1990} for detailed expositions.  Several of our results are minor modifications of analogous results in \cite{CPW:12}.

The following result and the accompanying proof is a minor modification of Proposition 2.2 in the supplementary material (s.m.) of \cite{CPW:12},  and its proof.  The modification allows us to provide a bound that does not scale with the ambient dimension.

\begin{proposition} \label{thm:lgmmod}
	Let $X_1,X_2 \in \mathbb{R}^{q\times q}$ be rank-$r$ matrices. Let $\sigma$ be the smallest nonzero singular value of $X_1$, and suppose that $\| X_1 - X_2 \|_{\spec} \leq \sigma / 8$. Then $\| \cp_{\ct(X_1)^{\perp}} (X_2) \|_{\fro} \leq \sqrt{r} \| X_1 - X_2 \|_{\spec}^2 / (3\sigma)$, and $\| \cp_{\ct(X_1)^{\perp}} (X_2) \|_{\spec} \leq \| X_1 - X_2 \|_{\spec}^2 / (5\sigma)$.
\end{proposition}

In the following proof, given a matrix $X \in \R^{q \times q}$, we denote $\tilde{X} := \left( \begin{array}{cc}
0 & X^{\prime} \\ X & 0 
\end{array} \right)$.

\begin{proof}[Proposition \ref{thm:lgmmod}]
	Let $\tilde{\Delta} = \tilde{X_2} - \tilde{X_1}$, and let $\kappa = \sigma /4$.  By combining equation (1.5) in the s.m. of \cite{CPW:12} with the proofs of Propositions 1.2 and 2.2 in the s.m. of \cite{CPW:12} it can be shown that $ \cp_{\ct(\tilde{X_1})^{\perp}} (\tilde{X_2})  = (1/(2\pi i))\oint_{\mathcal{C}_{\kappa}} \zeta [\tilde{X_1}-\zeta I]^{-1} \tilde{\Delta} [\tilde{X_1} - \zeta I]^{-1} \tilde{\Delta} [\tilde{X_2}-\zeta I]^{-1} d \zeta $, where the contour integral is taken along $\mathcal{C}_{\kappa}$ defined as the circle centered at the origin with radius $\kappa$.
	
	By a careful use of the inequality $\|AB\|_{\fro} \leq \|A\|_{\spec} \|B\|_{\fro}$, we have $\|[\tilde{X_1}-\zeta I]^{-1} \tilde{\Delta} [\tilde{X_1}-\zeta I]^{-1} \tilde{\Delta} [\tilde{X_2}-\zeta I]^{-1}\|_{\fro} \leq \|[\tilde{X_1}-\zeta I]^{-1}\|_{\spec} \|\tilde{\Delta} \|_{\fro} \|[\tilde{X_1}-\zeta I]^{-1}\|_{\spec} \|\tilde{\Delta}\|_{\spec} \|[\tilde{X_2}-\zeta I]^{-1}\|_{\spec}$. Since $\tilde{\Delta}$ is a matrix with rank at most $4r$, we have $\|\tilde{\Delta}\|_{\fro} \leq \sqrt{4r}\|\tilde{\Delta}\|_{\spec}$. We proceed to apply the same bounds as those used in the proof of Proposition 1.2 in the s.m. of \cite{CPW:12} to obtain $\|\cp_{\ct(\tilde{X_1})^{\perp}}(\tilde{X_2})\|_{\fro} \leq 2 \sqrt{r}\kappa^2 \|\tilde{\Delta}\|_{\spec}^2 / ((\sigma-\kappa)^2(\sigma - 3\kappa /2)) \leq \sqrt{2r} \| \tilde{X_1} - \tilde{X_2} \|_{\spec}^2 / (3\sigma)$.  The first inequality follows by noting that $\sqrt{2} \|\cp_{\ct(X_1)^{\perp}}(X_2)\|_{\fro} = \|\cp_{\ct(\tilde{X_1})^{\perp}}(\tilde{X_2})\|_{\fro}$, and that $\|X_1-X_2\|_{\spec} = \|\tilde{X_1} - \tilde{X_2} \|_{\spec}$.
	
	The proof of the second inequality follows from a similar argument. 
	\qed
\end{proof}

We define the following distance measure between two subspaces $\ct_1$ and $\ct_2$ \cite{CPW:12}
\begin{equation*}
\rho(\ct_1,\ct_2) := \underset{\|N\|_{\spec} \leq 1}{\sup} \| \cp_{\ct_1}-\cp_{\ct_2}(N) \|_{\spec}.
\end{equation*}
This definition is useful for quantifying the distance between tangent spaces with respect to the variety of low-rank matrices for pairs of nearby matrices.

\begin{lemma}\label{thm:pertubationrestrictedlinearmap}
	Let $X_1, X_2 \in \R^{q \times q}$ be matrices with rank at most $r$, and satisfy $\| X_1 - X_2 \|_{\spec} \leq \sigma / 8$, where $\sigma$ is the smallest nonzero singular value value of $X_2$. Let $\ct_1:= \ct(X_1)$ and $\ct_2: = \ct(X_2)$ be tangent spaces on the variety of matrices with rank at most $r$ at the points $X_1$ and $X_2$ respectively. Let $\L$ be a linear map satisfying the restricted isometry condition $\delta_{4r}(\L) \leq 1/10$.  If $Z_i \in \ct_i$, $i \in \{1,2\}$, then $\| [(\L^{\prime}_{\ct_1}\L_{\ct_1})^{-1}]_{\R^{q \times q}} (Z_1) - [(\L^{\prime}_{\ct_2}\L_{\ct_2})^{-1}]_{\R^{q \times q}}  (Z_2) \|_{\fro} \leq (43/10) \sqrt{r} \| Z_1 - Z_2\|_{\spec} + 16 r\|X_1 - X_2 \|_{\spec} \|Z_2 \|_{\spec} / \sigma$.
\end{lemma}

\begin{proof}[Lemma \ref{thm:pertubationrestrictedlinearmap}]
	To simplify notation we denote $Y_i = [(\L^{\prime}_{\ct_i}\L_{\ct_i})^{-1}]_{\R^{q \times q}} (Z_i)$, $i \in \{1,2\}$.  From the triangle inequality we have $\| Y_1 - Y_2 \|_{\fro} \leq \| \cp_{\ct_1^{\perp}} (Y_1 - Y_2) \|_{\fro} + \| \cp_{\ct_1} (Y_1 - Y_2) \|_{\fro} $. We bound both components separately.
	
	[$ \| \cp_{\ct_1^{\perp}} (Y_1 - Y_2) \|_{\fro} $]: From Proposition 2.1 of the s.m. of \cite{CPW:12} we have $\rho(\ct_1,\ct_2) \leq \frac{2}{\sigma}\|X_1-X_2\|_{\spec}$.  From Lemma \ref{thm:weakincoherencebound} we have $\| Y_2 - Z_2 \|_{\fro} \leq \delta_{4r} \| Y_2 \|_{\fro} \leq \frac{\delta_{4r}}{1-\delta_{4r}} \| Z_2 \|_{\fro} \leq \frac{\sqrt{2r}\delta_{4r}}{1-\delta_{4r}} \| Z_2 \|_{\spec}$. Hence
	\begin{eqnarray*}
		\| \cp_{\ct_1^{\perp}}(Y_2 - Z_2)\|_{\fro} & = & \| [\sfi-\cp_{\ct_1}]([\cp_{\ct_1} - \cp_{\ct_2}](Y_2 - Z_2))\|_{\fro} \\
		& \leq & 2 \sqrt{r} \| [\cp_{\ct_1} - \cp_{\ct_2}](Y_2 - Z_2) \|_{\spec} \\
		& \leq & 2 \sqrt{r} \rho(\ct_1,\ct_2)\|Y_2 - Z_2\|_{\spec} \\
		& \leq & \frac{4\sqrt{2}r}{\sigma} \frac{\delta_{4r}}{1-\delta_{4r}} \|X_1-X_2\|_{\spec} \| Z_2 \|_{\spec} .
	\end{eqnarray*}
Here the first inequality follows by noting that $[\sfi-\cp_{\ct_1}]([\cp_{\ct_1} - \cp_{\ct_2}](Y_2 - Z_2))$ has rank at most $4r$.  Next
	\begin{equation*}
	\| \cp_{\ct_1^{\perp}}( Z_2)\|_{\fro}  = \| \cp_{\ct_1^{\perp}}( Z_1 - Z_2)\|_{\fro} \leq \| Z_1 - Z_2 \|_{\fro} \leq 2 \sqrt{r} \| Z_1 - Z_2 \|_{\spec}.
	\end{equation*}	
	By combining both bounds with the triangle inequality we obtain
	\begin{eqnarray*}
		\| \cp_{\ct_1^{\perp}}(Y_1 - Y_2)\|_{\fro} = \| \cp_{\ct_1^{\perp}}( Y_2)\|_{\fro} & \leq & \| \cp_{\ct_1^{\perp}}( Z_2)\|_{\fro} + \| \cp_{\ct_1^{\perp}}( Y_2 - Z_2 )\|_{\fro} \\
		& \leq & 2\sqrt{r} \| Z_1 - Z_2 \|_{\spec} + \frac{4\sqrt{2}r}{\sigma} \frac{\delta_{4r}}{1-\delta_{4r}} \|X_1-X_2\|_{\spec} \| Z_2 \|_{\spec}.
	\end{eqnarray*}
	
	[$ \| \cp_{\ct_1} (Y_1 - Y_2) \|_{\fro} $]: Define the linear map $\mathsf{G} =  \L^{\prime}_{\ct_1 \cup \ct_2} \L_{\ct_1 \cup \ct_2} $.  First $ \|[\cp_{\ct_2} \circ \mathsf{G} \circ \cp_{\ct_2}] (Y_2) - [\cp_{\ct_1} \circ \mathsf{G} \circ \cp_{\ct_2}] (Y_2) \|_{\fro} \leq 2 \sqrt{r} \| [\cp_{\ct_2} \circ \mathsf{G} \circ \cp_{\ct_2}] (Y_2) - [\cp_{\ct_1} \circ \mathsf{G} \circ \cp_{\ct_2}] (Y_2) \|_{\spec} \leq 2 \sqrt{r} \rho(\ct_1,\ct_2) \| \mathsf{G} (Y_2)\|_{\spec} $, where $ \|\mathsf{G}(Y_2) \|_{\spec} \leq \|\mathsf{G}(Y_2) \|_{\fro} \leq (1+\delta_{4r}) \| Y_2 \|_{\fro} \leq \frac{1+ \delta_{4r}}{1 - \delta_{4r}} \|Z_2\|_{\fro} \leq \sqrt{2r}\frac{1+ \delta_{4r}}{1 - \delta_{4r}} \| Z_2 \|_{\spec}$.  Second $\|[\cp_{\ct_1} \circ \mathsf{G} \circ \cp_{\ct_2}] (Y_2) - [\cp_{\ct_1} \circ \mathsf{G} \circ \cp_{\ct_1}] (Y_2) \|_{\fro} = \|[\cp_{\ct_1} \circ \mathsf{G} \circ (\cp_{\ct_1} - \cp_{\ct_2})](Y_2)\|_{\fro} \leq \| [\mathsf{G} \circ (\cp_{\ct_1} - \cp_{\ct_2})] (Y_2)\|_{\fro} \leq (1+\delta_{4r})\| [\cp_{\ct_1} - \cp_{\ct_2}](Y_2)\|_{\fro} \leq 2 \sqrt{r} (1+\delta_{4r})\| [\cp_{\ct_1} - \cp_{\ct_2}] (Y_2)\|_{\spec} \leq 2 \sqrt{r} (1+\delta_{4r}) \rho(\ct_1,\ct_2) \|Y_2\|_{\spec}$, where $\| Y_2 \|_{\spec} \leq \|Y_2 \|_{\fro} \leq \frac{\sqrt{2r}}{1 - \delta_{4r}} \| Z \|_{\spec}$.  Third by combining these bounds with an application of Lemma \ref{thm:weakincoherencebound} and the triangle inequality we obtain
	\begin{align*}
		& ~~ \| \cp_{\ct_1} (Y_1 - Y_2) \|_{\fro} \\
		\leq & ~~  \frac{1}{1- \delta_{4r}} \| [\cp_{\ct_1} \circ \mathsf{G} \circ \cp_{\ct_1}] (Y_1 - Y_2) \|_{\fro} \\
		\leq & ~~  \frac{1}{1- \delta_{4r}}  ( \| [\cp_{\ct_1} \circ \mathsf{G} \circ \cp_{\ct_1}] (Y_1) - [\cp_{\ct_2} \circ \mathsf{G} \circ \cp_{\ct_2}](Y_2) \|_{\fro}  \\
		& \quad + ~ \|[\cp_{\ct_2} \circ \mathsf{G} \circ \cp_{\ct_2}] (Y_2) - [\cp_{\ct_1} \circ \mathsf{G} \circ \cp_{\ct_2}] (Y_2) \|_{\fro} \\
		& \quad + ~ \|[\cp_{\ct_1} \circ \mathsf{G} \circ \cp_{\ct_2}] (Y_2) - [\cp_{\ct_1} \circ \mathsf{G} \circ \cp_{\ct_1}] (Y_2) \|_{\fro}) \\
		\leq & ~~  \frac{1}{1- \delta_{4r}} (2 \sqrt{r} \|Z_1 - Z_2 \|_{\spec} + 4\sqrt{2} r \rho(\ct_1,\ct_2) \frac{1+\delta_{4r}}{1-\delta_{4r}} \|Z\|_{\spec} ).
	\end{align*}
	\qed
\end{proof}

\begin{proof}[Proposition \ref{thm:varietyconstrainedoptimization}] We prove (i) and (ii) in sequence.
	
	[(i)]: Let $\hat{X}_o$ be the optimal solution to the following
	\begin{equation*}
	\hat{X}_o ~=~ \underset{X}{\mathrm{argmin}}~\| \by - \L(X) \|^2_{\ell_2} \quad\quad
	\mathrm{s.t.} \quad\quad \mathrm{rank}(X) \leq r, \quad\quad \| X - X^{\star} \|_{\spec} \leq 4 \sqrt{r} \epsilon.
	\end{equation*}
	Since $4 \sqrt{r} \epsilon < 1/2 \leq \sigma_r (X^{\star})$, $\hat{X}_o$ has rank exactly $r$, and hence is a smooth point with respect to the variety of matrices with rank at most $r$. Define the tangent space $\hat{\ct}:= \ct (\hat{X}_o)$, and the matrix $\hat{X}_c$ as the solution to the following optimization instance
	\begin{equation*}
	\hat{X}_{c}~=~\underset{X}{\mathrm{argmin}}~\| \by - \L(X) \|^2_{\spec}  \quad\quad \mathrm{s.t.} \quad\quad X \in \hat{\ct}, \quad\quad \| X - X^{\star} \|_{\spec} \leq 4 \sqrt{r} \epsilon.
	\end{equation*}
	Here $\hat{X}_c$ is the solution to the optimization instance where the constraint $X \in \hat{\ct}$, which is convex, replaces the only non-convex constraint in the previous optimization instance. Hence $\hat{X}_{c} = \hat{X}_o$.  Define $\hat{X}_{\hat{\ct}}$ as the solution to the following optimization instance
	\begin{equation*}
	\hat{X}_{\hat{\ct}}~=~\underset{X}{\mathrm{argmin}} ~\| \by - \L(X) \|^2_{\ell_2} \quad \quad
	\mathrm{s.t.} \quad \quad X \in \hat{\ct}.
	\end{equation*}
	The first order condition is given by $\L^{\prime} \L ( \hat{X}_{\hat{\ct}} - X^{\star}) - \L^{\prime} \bz + Q_{\hat{\ct}^{\perp}} = 0$,	where $Q_{\hat{\ct}^{\perp}} \in \hat{\ct}^{\perp}$ is the Lagrange multiplier associated to the constraint $X \in \hat{\ct}$. Project the above equation onto the subspace $\hat{\ct}$ to obtain $[\cp_{\hat{\ct}} \circ \L^{\prime} \L \circ \cp_{\hat{\ct}}] ( \hat{X}_{\hat{\ct}} - X^{\star}) = [\cp_{\hat{\ct}} \circ \L^{\prime} \L \circ \cp_{\hat{\ct}^{\perp}}] (X^{\star}) + \cp_{\hat{\ct}} (\L^{\prime}\bz)$, and hence
	\begin{equation*}
	\hat{X}_{\hat{\ct}} - X^{\star} =  [(\L^{\prime}_{\hat{\ct}}\L_{\hat{\ct}})^{-1}]_{\R^{q \times q}} \circ \left( [\L^{\prime} \L \circ \cp_{\hat{\ct}^{\perp}}] (X^{\star}) + \L^{\prime}\bz \right) - \cp_{\hat{\ct}^{\perp}} (X^{\star}).
	\end{equation*}
	
	We proceed to bound $\| \hat{X}_{\hat{\ct}} - X^{\star}\|_{\spec}$.
	First we have $\| \hat{X}_{c} - X^{\star}\|_{\spec} \leq 4 \sqrt{r} \epsilon \leq 1/20$, and hence $\sigma_r (\hat{X}_c) \geq 9/20$. Second by applying Proposition \ref{thm:lgmmod}, we have $\| \cp_{\hat{\ct}^{\perp}} (X^{\star}) \|_{\spec} = \| \cp_{\hat{\ct}^{\perp}} (\hat{X}_c - X^{\star}) \|_{\spec} \leq (4 \sqrt{r} \epsilon)^2 / (5 \sigma_r (\hat{X}_c)) \leq (64/9) r \epsilon^2$, and $\|\cp_{\hat{\ct}^{\perp}} (X^{\star})\|_{\fro} \leq (320/27) r^{3/2} \epsilon^2 $. Third by Lemma \ref{thm:weakincoherencebound} and noting the inequality $\|\cdot\|_{\spec} \leq \|\cdot\|_{\fro}$ we have 
	\begin{eqnarray*}
	\| [(\L^{\prime}_{\hat{\ct}}\L_{\hat{\ct}})^{-1}]_{\R^{q \times q}} (\L^{\prime} \bz) \|_{\spec} & \leq & \| [(\L^{\prime}_{\hat{\ct}}\L_{\hat{\ct}})^{-1}]_{\R^{q \times q}} \|_{\spec} \| \cp_{\hat{\ct}}(\L^{\prime} \bz) \|_{\fro} \\
	& \leq & 2\sqrt{2r}\| \L^{\prime} \bz \|_{\spec} /(1 - \delta_{4r}) \leq (16/5)\sqrt{r} \epsilon.
	\end{eqnarray*}
	Fourth by Proposition 2.7 in \cite{GM:11} we have 
	\begin{eqnarray*}
	\| [ [(\L^{\prime}_{\hat{\ct}}\L_{\hat{\ct}})^{-1}]_{\R^{q \times q}} \circ \L^{\prime} \L \circ \cp_{\hat{\ct}^{\perp}}]  (X^{\star}) \|_{\spec} & \leq &  \| [(\L^{\prime}_{\hat{\ct}}\L_{\hat{\ct}})^{-1}]_{\R^{q \times q}} \|_{\spec}  \| [\cp_{\hat{\ct}} \circ \L^{\prime} \L \circ \cp_{\hat{\ct}^{\perp}}] (X^{\star}_{\hat{\ct}^{\perp}}) \|_{\fro} \\
	& \leq & \delta_{4r} \|\cp_{\hat{\ct}^{\perp}}(X^{\star})\|_{\fro} /(1 - \delta_{4r}) \leq (320/243) r^{3/2} \epsilon^2.
	\end{eqnarray*}
		
	Last, we combine the bounds to obtain $\| \hat{X}_{\hat{\ct}} - X^{\star} \|_{\spec} \leq 8 r \epsilon^2 + (16/5) \sqrt{r} \epsilon + 2 r^{3/2} \epsilon^2 < 4 \sqrt{r} \epsilon $. This implies that the constraint $\| X - X^{\star} \|_{\spec} \leq 4 \sqrt{r} \epsilon$ for $\hat{X}_c$ and $\hat{X}_o$ are inactive, and hence $\hat{X} = \hat{X}_o = \hat{X}_c = \hat{X}_{\hat{\ct}}$.
	
	[(ii)]: We have 
\begin{eqnarray*}
	G & = & [(\L^{\prime}_{\hat{\ct}}\L_{\hat{\ct}})^{-1}]_{\R^{q \times q}} (\L^{\prime} \bz) -  [(\L^{\prime}_{\ct^{\star}}\L_{\ct^{\star}})^{-1}]_{\R^{q \times q}} ( \L^{\prime} \bz) \\
	& + &  [ [(\L^{\prime}_{\hat{\ct}}\L_{\hat{\ct}})^{-1}]_{\R^{q \times q}} \circ \L^{\prime} \L \circ \cp_{\hat{\ct}^{\perp}}] (X^{\star})  - \cp_{\hat{\ct}^{\perp}} (X^{\star}).
\end{eqnarray*}
	We deal with the contributions of each term separately. 
	
	First $\| [\cp_{\ct^{\star}} - \cp_{\hat{\ct}}] (\L^{\prime} \bz) \|_{\spec} \leq \rho(\hat{\ct},\ct^{\star}) \| \L^{\prime} \bz\|_{\spec} \leq (2\epsilon/\sigma_{r}(X^{\star})) \| \hat{X} - X^{\star} \|_{\spec} \leq 16 \sqrt{r} \epsilon^2$, where the second inequality applies Proposition 2.1 of the s.m. of \cite{CPW:12}.  Second $\|\cp_{\ct^{\star}} (\L^{\prime} \bz)\|_{\spec} \leq 2\|\L^{\prime} \bz\|_{\spec} = 2 \epsilon$. Hence by applying Lemma \ref{thm:pertubationrestrictedlinearmap} with the choice of $Z_1 = \cp_{\hat{\ct}}(\L^{\prime} \bz)$ and $Z_2 = \cp_{\ct^{\star}} (\L^{\prime} \bz)$ we obtain $\| [(\L^{\prime}_{\ct^{\star}}\L_{\ct^{\star}})^{-1}]_{\R^{q \times q}} (\L^{\prime} \bz) - [(\L^{\prime}_{\hat{\ct}}\L_{\hat{\ct}})^{-1}]_{\R^{q \times q}} (\L^{\prime} \bz )\|_{\fro} \leq 70 r \epsilon^2 + 256 r^{3/2} \epsilon^2$.  Third we have  $\| [(\L^{\prime}_{\hat{\ct}}\L_{\hat{\ct}})^{-1}]_{\R^{q \times q}} \circ \L^{\prime} \L \circ \cp_{\hat{\ct}^{\perp}}] (X^{\star}) \|_{\fro} \leq (320/243) r^{3/2} \epsilon^2$, and $\|\cp_{\hat{\ct}^{\perp}}(X^{\star})\|_{\fro} \leq  (320/27) r^{3/2} \epsilon^2$.
	
	The bound follows by summing up these bounds.
	\qed
\end{proof}

The proof of Proposition \ref{thm:varietyconstrainederr} requires two additional preliminary results; in particular, the first establishes the restricted isometry condition for linear maps that are near linear maps that already satisfy the restricted isometry condition.


\begin{proposition} \label{thm:ripfornearbylinearmaps}
	Suppose $\L^{\star}$ is a linear map that satisfies the restricted isometry condition $\delta_{r}(\L^{\star}) \leq 1/ 20$.  Let $\sfe$ be a linear operator such that $\|\sfe\|_{\spec} \leq 1/(50\|\L^{\star}\|_{\spec})$.  Then $\L = \L^{\star} \circ (\sfi + \sfe)$ satisfies the restricted isometry condition $\delta_{r}(\L) \leq 1/10$.
\end{proposition}

\begin{proof}[Proposition \ref{thm:ripfornearbylinearmaps}]  Let $X$ be a matrix with rank at most $r$.  Then
	\begin{equation*}
		\| \L (X) \|_{\ell_2} \leq \| \L^{\star} (X) \|_{\ell_2} + \| \L^{\star} ( \sfe (X)) \|_{\ell_2} \leq (\sqrt{1+ \delta_{r}(\L^{\star})} + 1/50) \|X\|_{\fro} \leq \sqrt{1 + 1/10} \|X \|_{\fro}.
	\end{equation*}
	A similar argument also proves the lower bound $\| \L (X) \|_{\ell_2} \geq \sqrt{1 - 1/10} \|X\|_{\fro}$.
	\qed
\end{proof}

\begin{lemma}\label{thm:operatorf2to2} Suppose $\L$ satisfies the restricted isometry condition $\delta_1(\L) < 1$.  Then $\|\L^{\prime} \L\|_{\fro,\spec}\leq \sqrt{2(1+\delta_1(\L))} \| \L \|_{\spec} $.
\end{lemma}

\begin{proof}
	Let $Z \in \mathrm{argmax}_{X:\|X\|_{\fro}\leq 1} \| \L^{\prime} \L (X)\|_{\spec}$, and let $\ct$ be the tangent space of the rank-one matrix corresponding to the largest singular value of $Z$. Then $\sup_{X:\|X\|_{\fro} \leq 1} \| \L^{\prime} \L (X)\|_{\spec} \leq \sup_{X:\|X\|_{\fro} \leq 1} \| [\cp_{\ct} \circ \L^{\prime} \L] (X)\|_{\spec} \leq \sqrt{2} \sup_{X:\|X\|_{\fro} \leq 1} \| [\cp_{\ct} \circ \L^{\prime} \L] (X)\|_{\fro} \leq \sqrt{2} \| \cp_{\ct} \circ \L^{\prime} \L \|_{\spec} $.  By Lemma \ref{thm:weakincoherencebound} we have $\sqrt{2} \| \cp_{\ct} \circ \L^{\prime} \L \|_{\spec} \leq \sqrt{2(1+\delta_1(\L))} \|\L\|_{\spec}$.
	\qed
\end{proof}

\begin{proof}[Proposition \ref{thm:varietyconstrainederr}]
	To simplify notation we denote $\ct:=\ct(X^{\star})$.  Without loss of generality we may assume that $\|X^{\star}\|_{\spec}=1$.  By the triangle inequality we have
	\begin{align*}
		& ~~ \| (X^{\star} - \mathcal{M}(\hat{X})) - [(\L^{\star\prime}_{\ct} \L^{\star}_{\ct})^{-1}]_{\R^{q \times q}} \circ \L^{\star\prime} \L^{\star} \circ \sfe (X^{\star}) \|_{\fro} \\
		\leq & ~~ \| (X^{\star} - \mathcal{M}(\hat{X})) - [( (\sfi+\sfe^{\prime}) \circ \L^{\star\prime} \L^{\star} \circ (\sfi+\sfe) |_{\ct} )^{-1}]_{\R^{q \times q}} \circ (\sfi+\sfe^{\prime}) \circ \L^{\star\prime} \L^{\star} \circ \sfe (X^{\star}) \|_{\fro} \\
		+ & ~~ \| ( [((\sfi+\sfe^{\prime}) \circ \L^{\star\prime} \L^{\star} \circ (\sfi+\sfe)|_{\ct} )^{-1}]_{\R^{q \times q}} - [(\L^{\star\prime}_{\ct} \L^{\star}_{\ct} )^{-1}]_{\R^{q \times q}} ) \circ (\sfi+\sfe^{\prime}) \circ \L^{\star\prime} \L^{\star} \circ \sfe (X^{\star}) \|_{\fro} \\
		+ & ~~ \| [(\L^{\star\prime}_{\ct} \L^{\star}_{\ct} )^{-1}]_{\R^{q \times q}} \circ (\sfi+\sfe^{\prime}) \circ \L^{\star\prime} \L^{\star} \circ \sfe (X^{\star}) - [(\L^{\star\prime}_{\ct} \L^{\star}_{\ct} )^{-1}]_{\R^{q \times q}} \circ \L^{\star\prime} \L^{\star} \circ \sfe (X^{\star}) \|_{\fro}
	\end{align*}
	We bound each term separately.
	
	[First term]:  Let $\tilde{\bz}:=[\L^{\star} \circ \sfe] (X^{\star})$.  First by Proposition \ref{thm:ripfornearbylinearmaps} the linear map $\L^{\star} \circ (\sfi+\sfe)$ satisfies the restricted isometry condition $\delta_{4r}(\L^{\star} \circ (\sfi+\sfe)) \leq 1/10$.  Second we have $\|\sfi+\sfe\|_{\spec,\spec}\leq 1+\sqrt{q} \|\sfe\|_{\eu} \leq 51/50$.  Third from Lemma \ref{thm:operatorf2to2} we have $\|\L^{\star\prime}\L^{\star}\|_{\fro,\spec} \leq \sqrt{2(1+\delta_{4r}(\L^{\star}))} \|\L^{\star}\|_{\spec}$.  Fourth $\|\sfe(X^{\star})\|_{\fro} \leq \sqrt{r} \|\sfe\|_{\eu}$.  Hence
	\begin{equation*}
		\| (\sfi+\sfe^{\prime}) \circ \L^{\star\prime} \tilde{\bz} \|_{\spec} \leq \|\sfi+\sfe\|_{2,\spec} \|\L^{\star\prime}\L^{\star}\|_{\fro,\spec} \|\sfe\|_{\eu}\|X^{\star}\|_{\fro} \leq (3/2) \sqrt{r} \|\L^{\star}\|_{\spec} \|\sfe\|_{\eu}.
	\end{equation*}
	By the initial conditions we have that the above quantity is at most $1/(80r^{3/2})$.
	Consequently, by applying Proposition \ref{thm:varietyconstrainedoptimization} to the optimization instance \eqref{eq:varietyoptimization} with the choice of linear map $\L^{\star}\circ(\sfi+\sfe)$ and error term $\tilde{\bz}$ we have
	\begin{eqnarray*}
		& & \| (X^{\star} - \mathcal{M}(\hat{X})) - [( (\sfi+\sfe^{\prime}) \circ \L^{\star\prime} \L^{\star} \circ (\sfi+\sfe) |_{\ct} )^{-1}]_{\R^{q \times q}} \circ (\sfi+\sfe^{\prime}) \circ \L^{\star\prime} \tilde{\bz} \|_{\fro} \\
		& \leq & 765 r^{5/2} \|\L^{\star}\|_{\spec}^2 \|\sfe\|_{\eu}^2.
	\end{eqnarray*}
	
	[Second term]:  First by Lemma \ref{thm:weakincoherencebound} we have $\| [(\L^{\star\prime}_{\ct} \L^{\star}_{\ct} )^{-1}]_{\R^{q \times q}} \|_{\spec} \leq 20/19$.  Second by the triangle inequality we have $\| \cp_{\ct} \circ (\sfi+\sfe^{\prime}) \circ \L^{\star\prime} \L^{\star} \circ (\sfi+\sfe) \circ \cp_{\ct} - \cp_{\ct} \circ \L^{\star\prime} \L^{\star} \circ \cp_{\ct} \|_{\spec} \leq 3 \|\L^{\star}\|_{\spec} \|\sfe\|_{\eu}$.  Third by utilizing the identity $(\mathsf{A} + \mathsf{B})^{-1} = \mathsf{A}^{-1} - \mathsf{A}^{-1} \circ \mathsf{B} \circ \mathsf{A}^{-1} + \mathsf{A}^{-1} \circ \mathsf{B} \circ \mathsf{A}^{-1} \circ \mathsf{B} \circ \mathsf{A}^{-1} - \ldots $ with the choice of $\mathsf{A} = \L^{\star\prime}_{\ct} \L^{\star}_{\ct}$ and $\mathsf{B} = \cp_{\ct} \circ (\sfi+\sfe)^{\prime} \circ \L^{\star\prime} \L^{\star} \circ (\sfi+\sfe) \circ \cp_{\ct} - \mathsf{A}$ we obtain
	\begin{equation*}
	\| [((\sfi+\sfe^{\prime}) \circ \L^{\star\prime} \L^{\star} \circ (\sfi+\sfe)|_{\ct} )^{-1}]_{\R^{q \times q}} - [(\L^{\star\prime}_{\ct} \L^{\star}_{\ct} )^{-1}]_{\R^{q \times q}} \|_{\spec} \leq 4 \|\L^{\star}\|_{\spec} \|\sfe\|_{\eu}.
	\end{equation*}
	Fourth $\|\cp_{\ct} \circ (\sfi+\sfe^{\prime}) \circ \L^{\star\prime} \L^{\star} \circ \sfe (X^{\star}) \|_{\fro} \leq (11/10) \sqrt{r} \|\L^{\star}\|_{\spec} \|\sfe\|_{\eu}$.  Hence
	\begin{eqnarray*}
	& & \| ( [((\sfi+\sfe^{\prime}) \circ \L^{\star\prime} \L^{\star} \circ (\sfi+\sfe)|_{\ct} )^{-1}]_{\R^{q \times q}} - [(\L^{\star\prime}_{\ct} \L^{\star}_{\ct} )^{-1}]_{\R^{q \times q}} ) \circ (\sfi+\sfe^{\prime}) \circ \L^{\star\prime} \L^{\star} \circ \sfe (X^{\star}) \|_{\fro} \\
	& \leq & 5\sqrt{r} \|\L^{\star}\|_{\spec}^2 \|\sfe\|_{\eu}^2.
	\end{eqnarray*}
	
	[Third Term]:  We have
	\begin{eqnarray*}
	&& \| [(\L^{\star\prime}_{\ct} \L^{\star}_{\ct} )^{-1}]_{\R^{q \times q}} \circ (\sfi+\sfe^{\prime}) \circ \L^{\star\prime} \L^{\star} \circ \sfe (X^{\star}) - [(\L^{\star\prime}_{\ct} \L^{\star}_{\ct} )^{-1}]_{\R^{q \times q}} \circ \L^{\star\prime} \L^{\star} \circ \sfe (X^{\star}) \|_{\fro} \\
	& \leq & \| [(\L^{\star\prime}_{\ct} \L^{\star}_{\ct} )^{-1}]_{\R^{q \times q}} \|_{\spec} \|\sfe^{\prime}\|_{\spec} \|\L^{\star}\|_{\spec}^2 \|\sfe(X^{\star})\|_{\fro} \leq 2 \sqrt{r} \|\L^{\star}\|_{\spec}^2 \|\sfe\|_{\eu}^2.
	\end{eqnarray*}
	
	[Conclude]:  The result follows by summing each bound and applying Lemma \ref{thm:operatorf2to2}.
	\qed
\end{proof}

%% file: appendix_pseudoinverse.tex
\section{Proof of Proposition \ref{thm:structuralresult}} \label{apx:pseudoinverseexpand}

\begin{proof} [Proposition \ref{thm:structuralresult}]
	To simplify notation we let $\coveig := \coveig ( \{ {A^{(j)}}\}_{j=1}^{n} )$, $\covsup := \covsup ( \{ {A^{(j)}}\}_{j=1}^{n} ) $, and $\dd$ be the linear map defined as $\dd:\bz \mapsto \sum_{j=1}^{n} (\sfq(B^{(j)}) - A^{(j)})\bz_j $. In addition we define $\tau:= (1/\sqrt{n\coveig}) \|\dd\|_{\spec}$.  Note that by the Cauchy-Schwarz inequality we have $\tau \leq \omega / \sqrt{\coveig} \leq 1/20$.
	
	We begin by noting that since $\|(1/n\coveig)\xx^{\star} \circ \xx^{\star\prime} - \sfi \|_{\spec} \leq \covsup / \coveig \leq 1/6$, we have $\|((1/n\coveig)\xx^{\star} \circ \xx^{\star\prime})^{-1}\|_{\spec}$, and $\|(1/n\coveig)\xx^{\star}\circ\xx^{\star\prime}\|_{\spec} \leq 6/5$.
	
	Next we compute the following bounds. First $\| \dd\circ \dd^{\prime} \circ (\xx^{\star}\circ\xx^{\star \prime})^{-1} \|_{\spec} \leq \| \dd\|_{\spec}^2 \| (\xx^{\star}\circ\xx^{\star \prime})^{-1} \|_{\spec} \leq (6/5) \tau^2$.  Second $\| \dd \circ \xx^{\star\prime} \circ (\xx^{\star} \circ \xx^{\star\prime})^{-1} \|_{\spec} \leq \| \dd \|_{\spec} \|\xx^{\star\prime} \|_{\spec} \| (\xx^{\star}\circ \xx^{\star\prime})^{-1} \|_{\spec} \leq \tau (6/5)^{3/2}$. Third $\| \xx^{\star}\circ\dd^{\prime}\circ (\xx^{\star} \circ \xx^{\star\prime})^{-1} \|_{\spec} \leq \tau (6/5)^{3/2}$.  By applying these bounds to the following expansion we obtain
	\begin{eqnarray*}
		& & \bigl((\xx^{\star} + \dd)\circ(\xx^{\star} + \dd)^{\prime} \bigr)^{-1} \\
		&=& \bigl(\bigl( \sfi + \dd\circ\xx^{\star\prime}\circ (\xx^{\star}\circ \xx^{\star\prime})^{-1} + \xx^{\star}\circ \dd^{\prime}\circ(\xx^{\star}\circ \xx^{\star\prime})^{-1} + \sfe_1 \bigr)\circ \xx^{\star} \circ \xx^{\star\prime} \bigr)^{-1} \\
		&=& (\xx^{\star} \circ \xx^{\star\prime})^{-1} \bigl( \sfi - \dd \circ \xx^{\star\prime} \circ (\xx^{\star} \circ \xx^{\star\prime})^{-1} - \xx^{\star}\circ\dd^{\prime}\circ(\xx^{\star} \circ \xx^{\star\prime})^{-1} + \sfe_2 \bigr),
	\end{eqnarray*}
	where $\| \sfe_1 \|_{\spec} \leq (6/5) \tau^2$, and $\| \sfe_2 \|_{\spec} = \| -\sfe_1 + (\dd \circ \xx^{\star\prime} \circ (\xx^{\star} \circ \xx^{\star\prime})^{-1} + \xx^{\star} \circ \dd^{\prime} \circ (\xx^{\star} \circ \xx^{\star\prime})^{-1} + \sfe_1)^2  - (\ldots)^3 \|_{\spec} \leq ( \|\sfe_1 \|_{\spec} + \|\dd\circ \xx^{\star\prime} \circ (\xx^{\star} \circ \xx^{\star\prime})^{-1} + \xx^{\star} \circ \dd^{\prime} \circ (\xx^{\star} \circ \xx^{\star\prime})^{-1} + \sfe_1\|_{\spec}^2 + \ldots ) \leq (6/5) \tau^2 + (\tau (6/5) (\tau+2\sqrt{6/5}))^2 + \ldots \leq 1.2 \tau^2 + (3\tau)^2 + (3\tau)^3 + \ldots \leq 12\tau^2$.
	
	We apply the above expansion to derive the following approximation of $\xx^{\star} \circ (\xx^{\star} + \dd)^{+}$
	\begin{eqnarray*}
		& & \xx^{\star}\circ(\xx^{\star} + \dd)^{+} \\
		&=& \xx^{\star} \circ (\xx^{\star} + \dd)^{\prime} \circ \bigl((\xx^{\star} + \dd)\circ(\xx^{\star} + \dd)^{\prime} \bigr)^{-1} \\
		&=& (\xx^{\star}\circ\xx^{\star\prime} + \xx^{\star} \circ \dd^{\prime} )\circ(\xx^{\star}\circ \xx^{\star\prime})^{-1} \bigl( \sfi - \dd \circ \xx^{\star\prime} \circ (\xx^{\star} \circ \xx^{\star\prime})^{-1} - \xx^{\star} \circ \dd^{\prime}\circ(\xx^{\star} \circ \xx^{\star\prime})^{-1} + \sfe_2 \bigr) \\
		&=& (\sfi - \dd \circ \xx^{\star+} + \sfe_3 ),
	\end{eqnarray*}
	where $\sfe_3$ satisfies $\| \sfe_3\|_{\spec} \leq 2(\tau (6/5)^{3/2})(2(\tau (6/5)^{3/2}) + \| \sfe_2 \|_{\spec}) +\| \sfe_2 \|_{\spec} \leq 20 \tau^2$. 
	
	Next we write $((1/(n\coveig))\xx^{\star} \circ \xx^{\star\prime})^{-1} = \sfi + \sfe_4$, where $\|\sfe_4\|_{\spec} \leq (6/5)\covsup / \coveig$.  Then
	\begin{equation*}
		\xx^{\star} \circ (\xx^{\star} + \dd)^{+} = \sfi - \dd \circ \xx^{\star\prime} \circ (\xx^{\star} \circ \xx^{\star\prime})^{-1} + \sfe_3 =\sfi - (1/n\coveig) \dd \circ \xx^{\star \prime} + \sff,
	\end{equation*}
	where $\|\sff\|_{\spec} \leq  \| \sfe_3 \|_{\spec} + \| \dd \circ \xx^{\star\prime} \circ \sfe_4 \|_{\spec} /(n\coveig) \leq \| \sfe_3 \|_{\spec} + \tau (6/5)^{1/2} \| \sfe_4 \|_{\spec} \leq 20 \tau^2 + 2 \tau \covsup/\coveig$.  The result follows by noting that $\|\sff\|_{\eu} \leq q \|\sff\|_{\spec}$, $\tau \leq \omega / \sqrt{\coveig}$, and that $\xx^{\star} \circ \hat{\xx}^{+} = \xx^{\star} \circ (\xx^{\star} + \dd)^{+} \circ \sfq$.
	\qed
\end{proof}

%% file: appendix_induction.tex
\section{Proof of Proposition \ref{thm:induction}} \label{apx:linearizemanifold}

\begin{proposition}\label{thm:projectorwbound}
	Given an operator $\sfe : \R^{q\times q} \rightarrow \R^{q \times q}$, there exists matrices $E_L$, $E_R$ such that
	$\cp_{\W} (\sfe) = I \botimes E_L + E_R \botimes I$, and $\|E_L\|_{\fro},\|E_R\|_{\fro} \leq \|\sfe\|_{\eu}/\sqrt{q}$.
\end{proposition}

\begin{proof}[Proposition \ref{thm:projectorwbound}]
	Define the subspaces $\W_R := \{ S \botimes I : S \in \mathbb{R}^{q \times q} \}$ and $\W_L := \{ I \botimes S : S \in \mathbb{R}^{q \times q} \}$.  Note that $\W_R \cap \W_L = \mathrm{Span}(\sfi)$, and hence $\cp_{\W} = \cp_{\W_R \cap \mathrm{Span}(\sfi)^{\perp}} + \cp_{\W_L \cap \mathrm{Span}(\sfi)^{\perp}} + \cp_{\mathrm{Span}(\sfi)}$.
	
	Define $E_L$ and $E_R$ to be matrices such that $E_R \botimes I = \cp_{\W_R \cap \mathrm{Span}(\sfi)^{\perp}} (\sfe) + (1/2) \cp_{\mathrm{Span}(\sfi)} (\sfe)$, and $I \botimes E_L = \cp_{\W_L \cap \mathrm{Span}(\sfi)^{\perp}} (\sfe) + (1/2) \cp_{\mathrm{Span}(\sfi)} (\sfe)$.  For $i \in \{L,R\}$ we have the following.  Since $\cp_{\W_i \cap \mathrm{Span}(\sfi)^{\perp}}$ and $(1/2) \cp_{\mathrm{Span}(\sfi)}$ are projectors onto orthogonal subspaces with spectral norm $1$ and $1/2$ respectively, we have $\| E_i \botimes I \|_{\eu} \leq \| \sfe \|_{\eu}$. Moreover, since $\| E_i \botimes I \|_{\eu} = \|E_i\|_{\fro} \|I\|_{\fro}$, we have $\|E_i\|_{\fro} \leq \|\sfe\|_{\eu}/\sqrt{q}$. 
	\qed
\end{proof}

\begin{proof}[Proposition \ref{thm:induction}]
	By applying Proposition \ref{thm:projectorwbound} to the operator $\sfd$ we have $\cp_{\W}(\sfd) = I\botimes E_L + E_R\botimes I $ for a pair of matrices $E_L,E_R \in \R^{q\times q}$ satisfying $\|E_L\|_{\fro}, \|E_R\|_{\fro} \leq \|\sfd\|_{\eu} / \sqrt{q}$.  Moreover since $\|E_L\|_{\spec}, \|E_R\|_{\spec} <1$, it follows that the matrices $I+E_R$ and $I+E_L$ are invertible.  Consider the following identity
	\begin{equation*}
\sfi + \sfd = \left(\sfi + \left(\cp_{\W^{\perp}}(\sfd) - E_R \botimes E_L \right) \circ (I+E_R)^{-1} \botimes (I+E_L)^{-1} \right) \circ (I + E_R) \botimes (I + E_L).
	\end{equation*}
	We define $\sfg = (\cp_{\W^{\perp}}(\sfd) - E_R \botimes E_L) \circ (I+E_R)^{-1} \botimes (I+E_L)^{-1} - \cp_{\W^{\perp}}(\sfd)$, and we define $\sfiw = (I+E_R)\botimes (I+E_L)$.  By the triangle inequality we have $\|\sfiw - \sfi\|_{\spec} \leq 3\|\sfd\|_{\eu}/\sqrt{q}$.  
	
	Next we note that $\|(I+E_i)^{-1}\|_{\spec} \leq 10/9$, $i\in\{L,R\}$, and that $\| (I+E_R)^{-1} \botimes (I+E_L)^{-1} \|_{\spec} \leq 100/81$.  We also have $\|E_R\botimes E_L\|_{\eu} = \|E_R \|_{\fro} \|E_L\|_{\fro} \leq \|\sfd \|^2_{\eu} /q$.  By noting that $\| (I+E_i)^{-1} - I \|_{\spec} \leq (10/9)\|E_i\|_{\spec}$, $i\in\{L,R\}$, we have $\| (I+E_R)^{-1} \botimes (I+E_L)^{-1} - I \botimes I\|_{\spec} \leq 3\|\sfd \|_{\eu} / \sqrt{q}$.  By combining these bounds we obtain
$\|\sfg\|_{\eu} \leq \| \cp_{\W^{\perp}}(\sfd) \|_{\eu} \| (I+E_R)^{-1} \botimes (I+E_L)^{-1} - I \botimes I\|_{\spec} + \| E_R \botimes E_L \|_{\eu} \| (I+E_R)^{-1} \botimes (I+E_L)^{-1} \|_{\spec} \leq 5 \| \sfd \|_{\eu}^2 / \sqrt{q}$.
\qed
\end{proof}

%% file: appendix_stopping.tex
\section{Proof of Proposition \ref{thm:cauchy}} \label{apx:cauchy}

\begin{proof}[Proposition \ref{thm:cauchy}]
	To simplify notation in the proof we denote $\alpha_8:=\alpha_8(q,\L^{\star}) = 96 \sqrt{q} \allowbreak \|\L^{\star}\|_{\spec} $.  We show that
	\begin{equation} \label{eq:pointwisesufficientcondition}
	\| \L^{(t)} - \L^{(t+1)} \|_{\spec} \leq \alpha_{9} \xi_{\L^{\star}}(\L^{(t)}),
	\end{equation}
	for some function $\alpha_9:=\alpha_9(q,r,\L^{\star})$ that we specify later.  In the proof of Theorem \ref{thm:localconvergence} we showed that $\xi_{\L^{\star}}(\L^{(t)}) \leq \gamma^{t} \xi_{\L^{\star}}(\L^{(0)})$ for some $\gamma<1$.  Hence establishing \eqref{eq:pointwisesufficientcondition} immediately implies that the sequence $\{\L^{(t)} \}_{t=1}^{\infty}$ is Cauchy.
	
	Our proof builds on the proof of Theorem \ref{thm:localconvergence}.  Let
	\begin{equation*}
	\L^{(t)} = \L^{\star} \circ (\sfi + \sfe^{(t)}) \circ \sfim
	\end{equation*}
	where $\sfe^{(t)}$ is a linear map that satisfies $\|\sfe^{(t)}\|_{\eu} < 1/\alpha_0$.  In the proof of Theorem \ref{thm:localconvergence} we show that
	\begin{equation*}
	\L^{(t+1)} = \L^{\star} \circ (\sfi+\sfe^{(t+1)}) \circ \sfiw \circ \sfim \circ \sfin,
	\end{equation*}
	where $\|\sfe^{(t+1)}\|_{\eu} \leq \|\sfe^{(t)}\|_{\eu}$, $\sfiw$ is a rank-preserver, and $\sfin$ is a positive definite rank-preserver.  Moreover, as a consequence of applying Proposition \ref{thm:induction} to establish \eqref{eq:nextdictestimate} in the proof, we obtain the bound $\|\sfiw - \sfi\|_{\spec} \leq 3 \alpha_7\|\sfe^{(t)}\|_{\eu}$.  We use these bounds and relations to prove \eqref{eq:pointwisesufficientcondition}.
	
	By the triangle inequality we have
	\begin{eqnarray} \label{eq:cauchytriangleineq}
		\| \L^{(t)} - \L^{(t+1)} \|_{\spec} & \leq &  \| \L^{\star} \circ \sfe^{(t)} \circ \sfim \|_{\spec} + \| \L^{\star} \circ \sfe^{(t+1)} \circ \sfiw \circ \sfim \circ \sfin \|_{\spec} \nonumber \\
		& + & \| \L^{\star}  \circ \sfim \circ (\sfin - \sfi) \|_{\spec} + \| \L^{\star} \circ (\sfiw - \sfi) \circ \sfim \circ \sfin \|_{\spec}.
	\end{eqnarray}
	
	By Proposition \ref{thm:normalizednearothogonal} applied to the pairs of linear maps $\L^{(t)},\L^{\star}$ and $\L^{(t+1)},\L^{\star}$ we have $\|\sfim - \sfiq_1 \|_{\spec}, \|\sfiw  \circ \sfim \circ \sfin - \sfiq_2\|_{\spec} \leq \alpha_8 \|\sfe^{(t)}\|_{\eu}$, for some pair of orthogonal rank-preservers $\sfiq_1, \sfiq_2$.  Since $\alpha_8/\alpha_0\leq 1$ we have $\|\sfim\|_{\spec} \leq 2$ and $\|\sfiw  \circ \sfim \circ \sfin\|_{\spec} \leq 2$.  Consequently $\| \L^{\star} \circ \sfe^{(t)} \circ \sfim \|_{\spec}, \| \L^{\star} \circ \sfe^{(t+1)} \circ \sfiw \circ \sfim \circ \sfin \|_{\spec} \leq 2\|\L^{\star}\|_{\spec}\|\sfe^{(t)}\|_{\eu}$.
	
	Next we bound $\|\sfin - \sfi\|_{\spec}$.  By utilizing $\|\sfiw  \circ \sfim \circ \sfin - \sfiq_2\|_{\spec} \leq \alpha_8 / \alpha_0$, $\|\sfim - \sfiq_1 \|_{\spec} \leq \alpha_8 \|\sfe^{(t)}\|_{\eu}$, and $\|\sfiw - \sfi\|_{\spec} \leq 3 \alpha_7\|\sfe^{(t)}\|_{\eu}$, one can show that $\|\sfin-\sfiq_3\|_{\spec} \leq (6 \alpha_7 + 2\alpha_8 + 2) \|\sfe^{(t)}\|_{\eu}$, where $\sfiq_3 = \sfiq_1^{\prime} \circ \sfiq_2$ is an orthogonal rank-preserver.  Since $\sfin$ is self-adjoint, we have $\|\sfin^2-\sfi\|_{\spec} \leq 3 (6 \alpha_7 + 2\alpha_8 + 2) \|\sfe^{(t)}\|_{\eu}$, and hence $\|\sfin-\sfi\|_{\spec} \leq 3(6 \alpha_7 + 2\alpha_8 + 2)\|\sfe^{(t)}\|_{\eu}$.  This also implies the bound $\|\sfin\|_{\spec} \leq 3$.
	
	We apply these bounds to obtain $\| \L^{\star}  \circ \sfim \circ (\sfin - \sfi) \|_{\spec} \leq 6 (6 \alpha_7 + 2\alpha_8 + 2) \|\L^{\star}\|_{\spec} \|\sfe^{(t)}\|_{\eu}$, and $\| \L^{\star} \circ (\sfiw - \sfi) \circ \sfim \circ \sfin \|_{\spec} \leq 9 \alpha_7 \|\L^{\star}\|_{\spec} \|\sfe^{(t)}\|_{\spec}$.
	
	We define $\alpha_9 := (4 + 6 (6 \alpha_7 + 2\alpha_8 + 2) + 9 \alpha_7 )\|\L^{\star}\|_{\spec}$ (these are exactly the sum of the coefficients of $\|\sfe^{(t)}\|_{\eu}$ in the above bounds).  The result follows by adding these bounds, and subsequently taking the infimum over $\sfe^{(t)}$ in \eqref{eq:cauchytriangleineq}.
	\qed
\end{proof}

%% file: main.bbl
\begin{thebibliography}{10}
\providecommand{\url}[1]{{#1}}
\providecommand{\urlprefix}{URL }
\expandafter\ifx\csname urlstyle\endcsname\relax
  \providecommand{\doi}[1]{DOI~\discretionary{}{}{}#1}\else
  \providecommand{\doi}{DOI~\discretionary{}{}{}\begingroup
  \urlstyle{rm}\Url}\fi

\bibitem{AAJN:16}
Agarwal, A., Anandkumar, A., Jain, P., Netrapalli, P.: {L}earning {S}parsely
  {U}sed {O}vercomplete {D}ictionaries via {A}lternating {M}inimization.
\newblock SIAM Journal on Optimization \textbf{26}(4), 2775--2799 (2016).
\newblock \doi{10.1137/140979861}

\bibitem{AAN:17}
Agarwal, A., Anandkumar, A., Netrapalli, P.: {A} {C}lustering {A}pproach to
  {L}earning {S}parsely {U}sed {O}vercomplete {D}ictionaries.
\newblock IEEE Transactions on Information Theory \textbf{63}(1), 575--592
  (2017).
\newblock \doi{10.1109/TIT.2016.2614684}

\bibitem{AEB:06}
Aharon, M., Elad, M., Bruckstein, A.: {K}-{SVD}: {A}n {A}lgorithm for
  {D}esigning {O}vercomplete {D}ictionaries for {S}parse {R}epresentation.
\newblock IEEE Transactions on Signal Processing \textbf{54}(11), 4311--4322
  (2006).
\newblock \doi{10.1109/TSP.2006.881199}

\bibitem{AGTM:15}
Arora, S., Ge, R., Ma, T., Moitra, A.: {S}imple, {E}fficient, and {N}eural
  {A}lgorithms for {S}parse {C}oding.
\newblock In: Conference on Learning Theory (2015)

\bibitem{AGM:14}
Arora, S., Ge, R., Moitra, A.: {N}ew {A}lgorithms for {L}earning {I}ncoherent
  and {O}vercomplete {D}ictionaries.
\newblock Journal of Machine Learning Research: Workshop and Conference
  Proceedings \textbf{35}, 1--28 (2014)

\bibitem{BKS:15}
Barak, B., Kelner, J.A., Steurer, D.: {D}ictionary {L}earning and {T}ensor
  {D}ecomposition via the {S}um-of-{S}quares {M}ethod.
\newblock In: Proceedings of the Forty-seventh Annual ACM Symposium on Theory
  of Computing. ACM (2015).
\newblock \doi{10.1145/2746539.2746605}

\bibitem{Bar:93}
Barron, A.R.: {U}niversal {A}pproximation {B}ounds for {S}uperpositions of a
  {S}igmoidal {F}unction.
\newblock IEEE Transactions on Information Theory \textbf{39}(3), 930--945
  (1993).
\newblock \doi{10.1109/18.256500}

\bibitem{BTR:13}
Bhaskar, B.N., Tang, G., Recht, B.: {A}tomic {N}orm {D}enoising with
  {A}pplications to {L}ine {S}pectral {E}stimation.
\newblock IEEE Transactions on Signal Processing \textbf{61}(23), 5987--5999
  (2013)

\bibitem{BD:09}
Blumensath, T., Davies, M.E.: {I}terative {H}ard {T}hresholding for
  {C}ompressed {S}ensing.
\newblock Applied and Computational Harmonic Analysis \textbf{27}, 265--274
  (2009).
\newblock \doi{10.1016/j.acha.2009.04.002}

\bibitem{BDE:09}
Bruckstein, A.M., Donoho, D.L., Elad, M.: {F}rom {S}parse {S}olutions of
  {S}ystems of {E}quations to {S}parse {M}odeling of {S}ignals and {I}mages.
\newblock SIAM Review \textbf{51}(1), 34--81 (2009).
\newblock \doi{10.1137/060657704}

\bibitem{CanPla:11}
Cand\`es, E.J., Plan, Y.: {T}ight {O}racle {I}nequalities for {L}ow-{R}ank
  {M}atrix {R}ecovery {F}rom a {M}inimal {N}umber of {N}oisy {R}andom
  {M}easurements.
\newblock IEEE Transactions on Information Theory \textbf{57}(4), 2342--2359.
\newblock \doi{10.1109/TIT.2011.2111771}

\bibitem{CanRec:09}
Cand\`es, E.J., Recht, B.: {E}xact {M}atrix {C}ompletion via {C}onvex
  {O}ptimization.
\newblock Foundations of Computational Mathematics \textbf{9}(6), 717--772
  (2009).
\newblock \doi{10.1007/s10208-009-9045-5}

\bibitem{CRT:06}
Cand\`es, E.J., Romberg, J., Tao, T.: Robust {U}ncertainty {P}rinciples:
  {E}xact {S}ignal {R}econstruction from {H}ighly {I}ncomplete {F}requency
  {I}nformation.
\newblock IEEE Transactions on Information Theory \textbf{52}(2), 489--509
  (2006).
\newblock \doi{10.1109/TIT.2005.862083}

\bibitem{CT:06}
Cand\`es, E.J., Tao, T.: {N}ear-{O}ptimal {S}ignal {R}ecovery {F}rom {R}andom
  {P}rojections: {U}niversal {E}ncoding {S}trategies?
\newblock IEEE Transactions on Information Theory \textbf{52}(12), 5406--5425
  (2006).
\newblock \doi{10.1109/TIT.2006.885507}

\bibitem{CPW:12}
Chandrasekaran, V., Parillo, P., Willsky, A.S.: {L}atent {V}ariable {G}raphical
  {M}odel {S}election via {C}onvex {O}ptimization.
\newblock The Annals of Statistics \textbf{40}(4), 1935--1967 (2012).
\newblock \doi{10.1214/11-AOS949}

\bibitem{CRPW:12}
Chandrasekaran, V., Recht, B., Parrilo, P.A., Willsky, A.S.: {T}he {C}onvex
  {G}eometry of {L}inear {I}nverse {P}roblems.
\newblock Foundations of Computational Mathematics \textbf{12}(6), 805--849
  (2012).
\newblock \doi{10.1007/s10208-012-9135-7}

\bibitem{CDS:98}
Chen, S.S., Donoho, D.L., Saunders, M.A.: {A}tomic {D}ecomposition by {B}asis
  {P}ursuit.
\newblock SIAM Journal on Scientific Computing \textbf{20}(1), 33--61 (1998).
\newblock \doi{10.1137/S1064827596304010}

\bibitem{Cho:13}
Cho, E.: {I}nner {P}roducts of {R}andom {V}ectors on ${S}^n$.
\newblock Journal of Pure and Applied Mathematics: Advances and Applications
  \textbf{9}(1), 63--68 (2013)

\bibitem{Cut:13}
Cuturi, M.: {S}inkhorn {D}istances: {L}ightspeed {C}omputation of {O}ptimal
  {T}ransportation {D}istances.
\newblock In: Advances in Neural Information Processing Systems (2013)

\bibitem{DavSza:01}
Davidson, K.R., Szarek, S.J.: {L}ocal {O}perator {T}heory, {R}andom {M}atrices
  and {B}anach {S}paces.
\newblock In: W.B. Johnson, J.~Lindenstrauss (eds.) {H}andbook of the
  {G}eometry of {B}anach {S}paces, chap.~8, pp. 317--366. Elsevier B. V. (2011)

\bibitem{DevTem:96}
DeVore, R.A., Temlyakov, V.N.: Some {R}emarks on {G}reedy {A}lgorithms.
\newblock Advances in Computational Mathematics \textbf{5}(1), 173--187 (1996).
\newblock \doi{10.1007/BF02124742}

\bibitem{Don:06}
Donoho, D.L.: {C}ompressed {S}ensing.
\newblock IEEE Transactions on Information Theory \textbf{52}(4), 1289--1306
  (2006).
\newblock \doi{10.1109/TIT.2006.871582}

\bibitem{Don:06b}
Donoho, D.L.: {F}or {M}ost {L}arge {U}nderdetermined {S}ystems of {L}inear
  {E}quations the {M}inimal $\ell_1$-norm {S}olution {I}s {A}lso the {S}parsest
  {S}olution.
\newblock Communications on Pure and Applied Mathematics \textbf{59}(6),
  797--829 (2006).
\newblock \doi{10.1002/cpa.20132}

\bibitem{DonHuo:01}
Donoho, D.L., Huo, X.: {U}ncertainty {P}rinciples and {I}deal {A}tomic
  {D}ecomposition.
\newblock IEEE Transactions on Information Theory \textbf{47}(7), 2845--2862

\bibitem{Ela:10}
Elad, M.: {S}parse and {R}edundant {R}epresentations: {F}rom {T}heory to
  {A}pplications in {S}ignal and {I}mage {P}rocessing.
\newblock Springer (2010).
\newblock \doi{10.1007/978-1-4419-7011-4}

\bibitem{Faz:02}
Fazel, M.: {M}atrix {R}ank {M}inimization with {A}pplications.
\newblock Ph.D. thesis, Department of Electrical Engineering, Stanford
  University (2002)

\bibitem{FCRP:08}
Fazel, M., Cand\`es, E., Recht, B., Parrilo, P.: {C}ompressed {S}ensing and
  {R}obust {R}ecovery of {L}ow {R}ank {M}atrices.
\newblock In: 42nd IEEE Asilomar Conference on Signals, Systems and Computers
  (2008)

\bibitem{GGOW:16}
Garg, A., Gurvits, L., Oliveira, R., Wigderson, A.: {A} {D}eterministic
  {P}olynomial {T}ime {A}lgorithm for {N}on-{C}ommutative {R}ational {I}dentity
  {T}esting with {A}pplications.
\newblock In: IEEE 57th Annual Symposium on Foundations of Computer Science
  (2016).
\newblock \doi{10.1109/FOCS.2016.95}

\bibitem{GLM:16}
Ge, R., Lee, J.D., Ma, T.: {M}atrix {C}ompletion has {N}o {S}purious {L}ocal
  {M}inimum.
\newblock In: Advances in Neural Information Processing Systems (2016)

\bibitem{GM:11}
Goldfarb, D., Ma, S.: {C}onvergence of {F}ixed-{P}oint {C}ontinuation
  {A}lgorithms for {M}atrix {R}ank {M}inimization.
\newblock Foundations of Computational Mathematics \textbf{11}, 183--210
  (2011).
\newblock \doi{10.1007/s10208-011-9084-6}

\bibitem{Gor:63}
Gorman, W.M.: {E}stimating {T}rends in {L}eontief {M}atrices.
\newblock Unplublished note, referenced in Bacharach (1970)  (1963)

\bibitem{GPT:13}
Gouveia, J., Parrilo, P.A., Thomas, R.R.: {L}ifts of {C}onvex {S}ets and {C}one
  {F}actorizations.
\newblock Mathematics of Operations Research \textbf{38}(2), 248--264 (2013).
\newblock \doi{10.1287/moor.1120.0575}

\bibitem{GJBKS:16}
Gribonval, R., Jenatton, R., Bach, F., Kleinsteuber, M., Seibert, M.: {S}ample
  {C}omplexity of {D}ictionary {L}earning and {O}ther {M}atrix
  {F}actorizations.
\newblock IEEE Transactions on Information Theory \textbf{61}(6), 3469--3486
  (2015).
\newblock \doi{10.1109/TIT.2015.2424238}

\bibitem{Gur:04}
Gurvits, L.: {C}lassical {C}omplexity and {Q}uantum {E}ntanglement.
\newblock Journal of Computer and Systems Sciences \textbf{69}(3), 448--484
  (2004).
\newblock \doi{10.1016/j.jcss.2004.06.003}

\bibitem{Ide:16}
Idel, M.: {A} {R}eview of {M}atrix {S}caling and {S}inkhorn's {N}ormal {F}orm
  for {M}atrices and {P}ositive {M}aps.
\newblock CoRR \textbf{abs/1609.06349} (2016)

\bibitem{JMD:10}
Jain, P., Meka, R., Dhillon, I.S.: {G}uaranteed {R}ank {M}inimization via
  {S}ingular {V}alue {P}rojection.
\newblock In: Advances in Neural Information Processing Systems (2009)

\bibitem{Jon:92}
Jones, L.K.: {A} {S}imple {L}emma on {G}reedy {A}pproximation in {H}ilbert
  {S}pace and {C}onvergence {R}ates for {P}rojection {P}ursuit {R}egression and
  {N}eural {N}etwork {T}raining.
\newblock The Annals of Statistics \textbf{20}(1), 608--613 (1992).
\newblock \doi{10.1214/aos/1176348546}

\bibitem{Kat:66}
Kato, T.: {P}erturbation {T}heory for {L}inear {O}perators.
\newblock Springer-Verlag (1966)

\bibitem{KhaKal:91}
Khachiyan, L., Kalantari, B.: {D}iagonal {M}atrix {S}caling and {L}inear
  {P}rogramming.
\newblock SIAM Journal on Optimization \textbf{2}(4), 668--672 (1991).
\newblock \doi{10.1137/0802034}

\bibitem{LSW:00}
Linial, N., Samorodnitsky, A., Wigderson, A.: {A} {D}eterministic {S}trongly
  {P}olynomial {A}lgorithm for {M}atrix {S}caling and {A}pproximate
  {P}ermanents.
\newblock Combinatorica \textbf{20}(4), 545--568 (2000).
\newblock \doi{10.1007/s004930070007}

\bibitem{MBP:14}
Mairal, J., Bach, F., Ponce, J.: {S}parse {M}odeling for {I}mage and {V}ision
  {P}rocessing.
\newblock Foundations and Trends in Computer Graphics and Vision
  \textbf{8}(2--3), 85--283 (2014).
\newblock \doi{10.1561/0600000058}

\bibitem{MarMoy:59}
Marcus, M., Moyls, B.N.: {T}ransformations on {T}ensor {P}roduct {S}paces.
\newblock Pacific Journal of Mathematics \textbf{9}(4), 1215--1221 (1959)

\bibitem{MeiBuh:06}
Meinhausen, N., B\"uhlmann, P.: {H}igh-{D}imensional {G}raphs and {V}ariable
  {S}election with the {L}asso.
\newblock The Annals of Statistics \textbf{34}(3), 1436--1462 (2006).
\newblock \doi{10.1214/009053606000000281}

\bibitem{Nat:93}
Natarajan, B.K.: {S}parse {A}pproximate {S}olutions to {L}inear {S}ystems.
\newblock SIAM Journal on Computing \textbf{24}(2), 227--234 (1993).
\newblock \doi{10.1137/S0097539792240406}

\bibitem{NesNem:94}
Nesterov, Y., Nemirovskii, A.: {I}nterior-{P}oint {P}olynomial {A}lgorithms in
  {C}onvex {P}rogramming.
\newblock SIAM Studies in Applied and Numerical Mathematics (1994).
\newblock \doi{10.1137/1.9781611970791}

\bibitem{OlsFie:96}
Olshausen, B.A., Field, D.J.: {E}mergence of {S}imple-{C}ell {R}eceptive
  {F}ield {P}roperties by {L}earning a {S}parse {C}ode for {N}atural {I}mages.
\newblock Nature \textbf{381}, 607--609 (1996).
\newblock \doi{10.1038/381607a0}

\bibitem{OymHas:16}
Oymak, S., Hassibi, B.: {S}harp {MSE} {B}ounds for {P}roximal {D}enoising.
\newblock Foundations of Computational Mathematics \textbf{16}(4), 965--1029
  (2016).
\newblock \doi{10.1007/s10208-015-9278-4}

\bibitem{PB:14}
Parikh, N., Boyd, S.: {P}roximal {A}lgorithms.
\newblock Foundations and Trends in Optimization \textbf{1}(3), 127--239
  (2014).
\newblock \doi{10.1561/2400000003}

\bibitem{Pis:81}
Pisier, G.: Remarques sur un r\'esultat non publi\'e de {B}. {M}aurey.
\newblock S\'eminaire Analyse fonctionnelle (dit "Maurey-Schwartz") pp. 1--12
  (1981)

\bibitem{RFP:10}
Recht, B., Fazel, M., Parrilo, P.A.: {G}uaranteed {M}inimum-{R}ank {S}olutions
  of {L}inear {M}atrix {E}quations via {N}uclear {N}orm {M}inimization.
\newblock SIAM Review \textbf{52}(3), 471--501 (2010).
\newblock \doi{10.1137/070697835}

\bibitem{Ren:01}
Renegar, J.: {A} {M}athematical {V}iew of {I}nterior-{P}oint {M}ethods in
  {C}onvex {O}ptimization.
\newblock MOS-SIAM Series on Optimization (2001).
\newblock \doi{10.1137/1.9780898718812}

\bibitem{Sch:14}
Schnass, K.: {O}n the {I}dentifiability of {O}vercomplete {D}ictionaries via
  the {M}inimisation {P}rinciple {U}nderlying {K}-{S}{V}{D}.
\newblock Applied and Computational Harmonic Analysis \textbf{37}(3), 464--491
  (2014).
\newblock \doi{10.1016/j.acha.2014.01.005}

\bibitem{Sch:16}
Schnass, K.: {C}onvergence {R}adius and {S}ample {C}omplexity of {I}{T}{K}{M}
  {A}lgorithms for {D}ictionary {L}earning.
\newblock Applied and Computational Harmonic Analysis  (2016).
\newblock \doi{10.1016/j.acha.2016.08.002}

\bibitem{SBTR:12}
Shah, P., Bhaskar, B.N., Tang, G., Recht, B.: {L}inear {S}ystem
  {I}dentification via {A}tomic {N}orm {R}egularization.
\newblock In: 51st IEEE Conference on Decisions and Control (2012)

\bibitem{Sin:64}
Sinkhorn, R.: {A} {R}elationship {B}etween {A}rbitrary {P}ositive {M}atrices
  and {D}oubly {S}tochastic {M}atrices.
\newblock The Annals of Mathematical Statistics \textbf{35}(2), 876--879
  (1964).
\newblock \doi{10.1214/aoms/1177703591}

\bibitem{SWW:12}
Spielman, D.A., Wang, H., Wright, J.: {E}xact {R}ecovery of {S}parsely-{U}sed
  {D}ictionaries.
\newblock Journal on Machine Learning and Research: Workshop and Conference
  Proceedings \textbf{23}(37), 1--18 (2012)

\bibitem{SteSun:1990}
Stewart, G., Sun, J.: {M}atrix {P}erturbation {T}heory.
\newblock Academic Press (1990)

\bibitem{SQW:17}
Sun, J., Qu, Q., Wright, J.: {A} {G}eometric {A}nalysis of {P}hase {R}etrieval.
\newblock Foundations of Computational Mathematics  (2017).
\newblock \doi{10.1007/s10208-017-9365-9}

\bibitem{SQW:16a}
Sun, J., Qu, Q., Wright, J.: {C}omplete {D}ictionary {R}ecovery over the
  {S}phere {I}: {O}verview and the {G}eometric {P}icture.
\newblock IEEE Transactions on Information Theory \textbf{63}(2), 853--884
  (2017).
\newblock \doi{10.1109/TIT.2016.2632162}

\bibitem{SQW:16b}
Sun, J., Qu, Q., Wright, J.: {C}omplete {D}ictionary {R}ecovery over the
  {S}phere {II}: {R}ecovery by {R}iemannian {T}rust-region {M}ethod.
\newblock IEEE Transactions on Information Theory \textbf{63}(2), 885--914
  (2017).
\newblock \doi{10.1109/TIT.2016.2632149}

\bibitem{Tib:94}
Tibshirani, R.: {R}egression {S}hrinkage and {S}election via the {L}asso.
\newblock Journal of the Royal Statistical Society, Series B \textbf{58},
  267--288 (1994)

\bibitem{sdpt3:1}
Toh, K.C., Todd, M.J., T\"{u}t\"{u}nc\"{u}, R.H.: {SDPT3} -- a {MATLAB}
  {S}oftware {P}ackage for {S}emidefinite {P}rogramming.
\newblock Optimization Methods and Software \textbf{11}, 545--581 (1999).
\newblock \doi{10.1080/10556789908805762}

\bibitem{Tro:12}
Tropp, J.A.: {U}ser-{F}riendly {T}ail {B}ounds for {S}ums of {R}andom
  {M}atrices.
\newblock Foundations of Computational Mathematics \textbf{12}(4), 389--434
  (2012).
\newblock \doi{10.1007/s10208-011-9099-z}

\bibitem{Tun:00}
Tun\c{c}el, L.: {P}otential {R}eduction and {P}rimal-{D}ual {M}ethods.
\newblock In: H.~Wolkowicz, R.~Saigal, L.~Vandenberghe (eds.) {H}andbook of
  {S}emidefinite {P}rogramming -- {T}heory, {A}lgorithms, and {A}pplications,
  chap.~9. Kluwer's International Series in Operations Research and Management
  Science (2000).
\newblock \doi{10.1007/978-1-4615-4381-7}

\bibitem{VMB:11}
Vainsencher, D., Mannor, S., Bruckstein, A.M.: The sample complexity of
  dictionary learning.
\newblock Journal of Machine Learning Research \textbf{12} (2011)

\bibitem{Yan:91}
Yannakakis, M.: {E}xpressing {C}ombinatorial {O}ptimization {P}roblems by
  {L}inear {P}rograms.
\newblock Journal of Computer and System Sciences \textbf{43}, 441--466 (1991).
\newblock \doi{10.1016/0022-0000(91)90024-Y}

\end{thebibliography}
